\documentclass[lettersize,journal]{IEEEtran}
\usepackage{amsmath,amsfonts}
\usepackage{algorithmic}
\usepackage{algorithm}
\usepackage{array}
\usepackage[caption=false,font=normalsize,labelfont=sf,textfont=sf]{subfig}
\usepackage{textcomp}
\usepackage{stfloats}
\usepackage{url}
\usepackage{verbatim}
\usepackage{graphicx}
\usepackage{cite}
\usepackage{mathrsfs}
\usepackage{color}
\usepackage{bm}
\usepackage{amssymb}
\usepackage{amsthm}
\usepackage{caption}
\usepackage{epstopdf}
\usepackage{booktabs}
\usepackage{epsfig}
\usepackage{float}
\usepackage{multirow}

\newtheorem{thm}{Theorem}
\newtheorem{defn}{Definition}
\newtheorem{remark}{Remark}
\newtheorem{lem}{Lemma}
\newtheorem{corollary}{Corollary}
\newtheorem{example}{Example}
\newtheorem{proposition}{Proposition}
\newtheorem{pro}{Property}

\begin{document}

\title{The Graph Fractional Fourier Transform \\ in Hilbert Space
%The graph fractional Fourier transform associate with Hilbert space
}

\author{Yu Zhang and Bing-Zhao Li$^{\ast}$,~\IEEEmembership{Member,~IEEE,}
	% <-this % stops a space
	\thanks{This paper was supported by the BIT Research and Innovation Promoting Project [No.2023YCXY053], the National Natural Science Foundation of China [No. 62171041], and Natural Science Foundation of Beijing Municipality [No. 4242011]. Corresponding author: Bing-Zhao Li.}% <-this % stops a space
	\thanks{The authors are with School of Mathematics and Statistics, Beijing Institute of Technology, Beijing 102488, China (e-mail: li$\_$bingzhao@bit.edu.cn; zhangyu$\_$bit@keio.jp).}}

\markboth{Journal of \LaTeX\ Class Files,~Vol.~, No.~, ~2024}%
{Shell \MakeLowercase{\textit{et al.}}: A Sample Article Using IEEEtran.cls for IEEE Journals}

\IEEEpubid{0000--0000/00\$00.00~\copyright~2024 IEEE}

\maketitle

\begin{abstract}
Graph signal processing (GSP) leverages the inherent signal structure within graphs to extract high-dimensional data without relying on translation invariance. It has emerged as a crucial tool across multiple fields, including learning and processing of various networks, data analysis, and image processing. In this paper, we introduce the graph fractional Fourier transform in Hilbert space (HGFRFT), which provides additional fractional analysis tools for generalized GSP by extending Hilbert space and vertex domain Fourier analysis to fractional order. First, we establish that the proposed HGFRFT extends traditional GSP, accommodates graphs on continuous domains, and facilitates joint time-vertex domain transform while adhering to critical properties such as additivity, commutativity, and invertibility. Second, to process generalized graph signals in the fractional domain, we explore the theory behind filtering and sampling of signals in the fractional domain. Finally, our simulations and numerical experiments substantiate the advantages and enhancements yielded by the HGFRFT.
\end{abstract}


\section{Introduction}
\label{Intro}
\IEEEPARstart{G}{iven} the rising complexity of data on non-Euclidean spaces and irregular domains, the fields of theory and application in graph signal processing (GSP) have seen rapid development in recent years \cite{GFTlaplace,GFTadjacency1,GFTadjacency2, Goverview, Ghistory, GFTsampling, GFTSSS, Gfilters, Gfrequency, GFTrecovery, GFTuncertainty, Gglobal, Gvertex}. Within the GSP framework, two primary approaches to signal processing have been established \cite{GFTlaplace,GFTadjacency1}. The first leverages the concept of the Laplacian matrix \cite{GFTlaplace}, while the second, rooted in algebraic signal processing, employs the graph adjacency matrix, also recognized as the graph shift operator \cite{GFTadjacency1,GFTadjacency2}.

The extension of traditional time-frequency domain signal processing techniques into the GSP realm has been notable in recent years. These expansions encompass methods for sampling and interpolation \cite{GFTsampling,GFTSSS}, filtering \cite{Gfilters}, frequency analysis \cite{Gfrequency}, signal reconstruction \cite{GFTrecovery}, exploring the uncertainty principle \cite{GFTuncertainty, Gglobal}, the discrete Fourier transform (DFT) \cite{GFTlaplace,GFTadjacency1}, and the windowed Fourier transform \cite{Gvertex}. Analogous to the role of the DFT in traditional time-frequency domains, the graph Fourier transform (GFT) has been adapted for various GSP applications, including denoising \cite{JFT,JFRFT,JLCT}, classification \cite{GFTsampling}, clustering \cite{JFT,JFRFT,JLCT,Gclustering}, semi-supervised learning \cite{Glearning}, and advancements in machine and deep learning \cite{Gdeep}.

Traditional GSP theory \cite{GFTlaplace,GFTadjacency1,GFTadjacency2, Goverview, Ghistory} is predicated on the variation of the orthonormal basis within finite-dimensional vector spaces. A graph signal $f$ assigns a real value to each vertex, rendering $f$ an element of $\mathbb{R}^N$, where $\mathbb{R}$ denotes the set of real numbers. Regardless of whether a Laplacian or adjacency matrix is used, all eigenvalues are real and form an orthonormal basis \cite{Functional}. Most current research in graph signals adheres to this principle, with more complex graph structures receiving less attention.

However, in GSP, irregular graph data are typically modeled with vertices and edges, where each vertex can be assigned not just a real number but a mathematical object of richer structure. Ji and Tay \cite{HGFT,Ggeneralized} introduced the innovative \textit{generalized GSP} approach within Hilbert space. Their method allows for the processing of information in Hilbert space alongside the vertex domain. The graph Fourier transform in Hilbert space (HGFT), as a principal tool for joint Fourier analysis in this framework, is defined through compact operators in Hilbert space combined with graph adjacency matrices. The HGFT merges DFT in Hilbert space with GFT in the vertex domain, aiming to simplify signals from infinite continuous domains into more tractable finite ones. This framework encompasses filtering \cite{HGFT,Ggeneralized,Gl2norm,Gprobability,JFRFTfilter,JFRFTwiener}, sampling \cite{Jsampling,Jproduct,Jdirected}, and estimation \cite{Jestimation} techniques, applicable across three scenarios:
\begin{itemize}
	\item{Traditional GSP applies when vertices are assigned real numbers \cite{GFTlaplace,GFTadjacency1,GFTadjacency2, Goverview, Ghistory, GFTsampling, GFTSSS, Gfilters, Gfrequency, GFTrecovery, GFTuncertainty, Gglobal, Gvertex}.}
	\item{A joint time-vertex GSP is identified when vertices correspond to a specified finite closed interval $[a, b]$ adhering to the Lebesgue measure \cite{JFT,JFRFT,JLCT}.}
	\item{A novel framework arises when vertices are allocated a new graph structure within a Hilbert space, characterized by a discrete metric in $L^2$ or $\ell^2$ spaces \cite{HGFT,Ggeneralized,Gl2norm}.}
\end{itemize}

While the HGFT effectively facilitates basic graph frequency analysis within its framework, it relies on access to the entire graph signal in Hilbert space, which is challenging when transitioning from the Hilbert space-vertex domain to the graph frequency domain. As a result of these limitations, the application of HGFT across various fields has remained relatively superficial \cite{HGFT}, and it struggles to handle complex graph signals, such as \textit{chirp signals} on graphs, as illustrated in Fig. \ref{fig1}. This constraint underscores the necessity of developing more general and flexible GSP methodologies to address the demands of complex signal analysis.
\IEEEpubidadjcol
\begin{figure}[h!]
	\begin{center}
		\begin{minipage}[t]{1\linewidth}
			\centering
			\includegraphics[width=\linewidth]{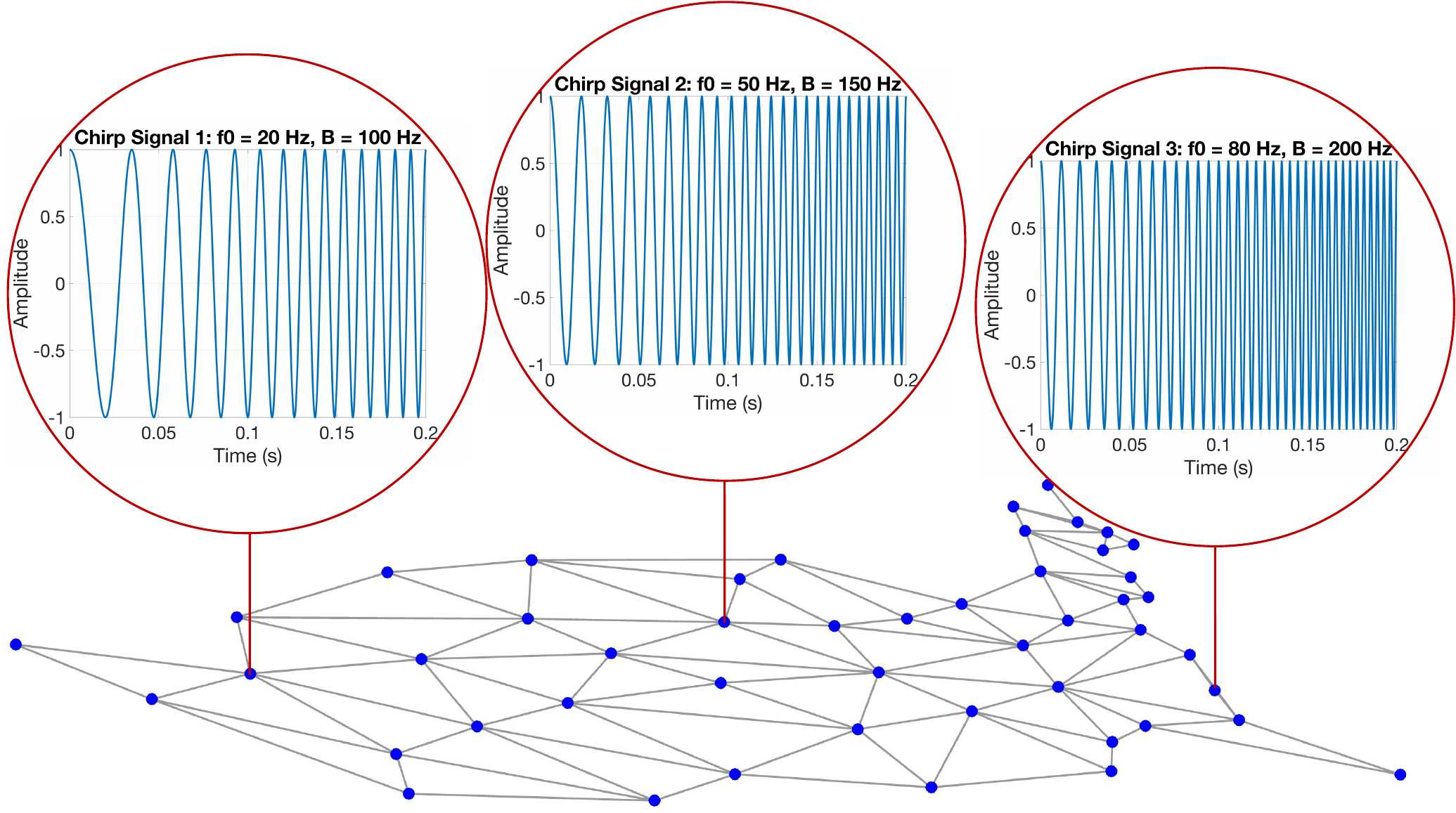}
			%\parbox{3cm}{\tiny \textcolor{red}{(The path graph of the signal is $\Psi^{\alpha}_1 + 0.5\Psi^{\alpha}_2 + 2\Psi^{\alpha}_3$.}}
		\end{minipage}
	\end{center}
	\caption{Chirp signals on the U.S. graph.}
	\vspace*{-3pt}
	\label{fig1}
\end{figure}

To address these challenges, this study explores the application of the fractional Fourier transform (FRFT) within the domain of generalized GSP. First introduced in 1980 \cite{FRFT1} and later applied to signal processing in 1994 \cite{FRFT2}, the FRFT extends the Fourier transform (FT) by rotating a signal into an intermediate domain between time and frequency. For a given signal $f(t)$, the $\alpha$th-order  FRFT is defined for $0<|\alpha|<2$ as
\begin{align}
	\begin{split}
		{\mathcal{F}^{\alpha}_{f}}(u)&:=\int_{-\infty}^{\infty} K_{\alpha}(u,t)f(t)\mathrm{d}t \\
		K_{\alpha}(u,t)&:=K_{\theta}\mathrm{e}^{j\pi(u^2\cot\theta  - 2ut\csc\theta  + t^2\cot\theta  )}
	\end{split}
\end{align}
where $\theta :=\alpha\pi/2$, and $K_\theta := \exp[-j(\pi\mathrm{sgn}(\theta)/4-\theta/2)]/|\sin\theta|^{0.5}.$
The kernel $K_\alpha(u,t)$ is defined separately for $\alpha=0$ and $\alpha=\pm 2$ as $K_0(u,t):=\delta(u-t)$ and $K_{\pm 2}(u,t):=\delta(u+t)$. The parameter $\alpha$ governs the rotation of the signal in the time-frequency plane, where $\alpha = 0$ corresponds to the FT and $\alpha = \pm 1$ denotes time reversal. This transformation disperses the signal across a series of linear frequency-modulated pulses, providing a more flexible approach to time-frequency signal processing.

The FRFT has widespread applications in signal processing, particularly in analyzing non-stationary signals. It excels in radar, communication, image, audio, and biomedical signal processing \cite{FRFT1,FRFT2,DFRFTlinear,STFRFT,DFRFTaffine,DFRFTSamplingMethod,DFRFT1,DFRFT2,DFRFT3,DFRFT4}. Adjusting the transformation order allows it to align better with instantaneous frequency characteristics, improving detection accuracy and feature extraction. However, most practical applications necessitate discretization, including in GSP. The challenge of discretizing the FRFT without direct time domain sampling has garnered considerable research attention. The methods for defining and calculating the discrete fractional Fourier transform (DFRFT) primarily fall into three categories: linear combination methods \cite{DFRFTlinear,DFRFTaffine}, sampling methods \cite{DFRFTSamplingMethod}, and eigendecomposition methods \cite{DFRFT1,DFRFT2,DFRFT3,DFRFT4}.

In this paper, we clarify the advantages of the FRFT by applying it to radar chirp signals \cite{Radar} on graphs as illustrated in Fig. \ref{fig1} and comparing its performance to the HGFT. This example illustrates the benefits of FRFT across various graph signal processing tasks, highlighting its potential for expanded applications within GSP. Integrating FRFT into GSP, similar to the FT extension, offers new avenues for enhanced performance and broader application in this field.

The development of the graph fractional Fourier transform (GFRFT) \cite{GFRFT,GFRFTspectral} represents a significant advancement, allowing for the transformation of graph signals into an intermediary vertex frequency or vertex spectral domain. This breakthrough, along with the generalization of the windowed fractional Fourier transform and the linear canonical transform to GSP \cite{WGFRFT,GLCT}, has facilitated research into operations within the fractional domain such as classification \cite{GFRFTsampling,GLCTsampling}, sampling \cite{GFRFTsampling,GLCTsampling,Ggeneralizedsamp}, filtering \cite{GFRFTfiltering}, and denoising \cite{GFRFTdirected}.

The GFRFT offers a distinct advantage by permitting the observation of the spectrum of graph signals from various angles without needing information on the entire graph signal. Its flexibility, attributed to the presence of an order, makes it more versatile than the GFT, particularly for processing ``graph chirp signals''. Therefore, incorporating the GFRFT into the generalized GSP framework \cite{HGFT,Ggeneralized} not only deepens the understanding of HGFT characteristics but also opens exciting research avenues. Employing fractional order in joint Hilbert space-vertex signal spectrum analysis enables the processing of signals with complex structures within the fractional domain, marking a promising direction for current research efforts. In the vein of the FRFT, which generalizes the FT, the proposal of the graph fractional Fourier transform in Hilbert space (HGFRFT) introduces a method for analyzing signals across both the graph fractional domain and the Hilbert space fractional domain. This framework is also applied to real-world datasets, with strategies developed for various scenarios of interest, further extending its applicability and utility in the field of GSP.
%The GFRFT enables spectral analysis of graph signals from multiple perspectives without requiring full signal information. Its flexibility, driven by fractional order, makes it more versatile than the GFT, especially for ``graph chirp signals''. Integrating GFRFT into the generalized GSP framework \cite{HGFT,Ggeneralized} enhances the understanding of HGFT properties and opens new research opportunities. By incorporating fractional orders in joint Hilbert space-vertex spectrum analysis, the GFRFT processes complex signals in the fractional domain, offering a promising research direction. Like the FRFT generalizing the FT, the proposed graph fractional Fourier transform in Hilbert space (HGFRFT) extends signal analysis to both graph fractional and Hilbert space fractional domains, with applications to real-world datasets and tailored strategies for various scenarios.

\textit{Contributions:} 
\begin{itemize}
	\item{The introduction of the HGFRFT allows for the analysis of vertex signals within Hilbert space across a broader spectrum of transformation categories, offering a novel solution for graph chirp signals.}
	\item{Filters grounded in the HGFRFT framework for graph signals in Hilbert space are proposed, enriching GSP tools within the fractional domain.}
	\item{A pioneering sampling operator rooted in the HGFRFT for graph signals in Hilbert space is introduced, adept at addressing intricate sampling challenges in the Hilbert space-vertex domain.}
\end{itemize}

The remainder of this paper is organized as follows. In Section II, we provide preliminary information on the generalized GSP, GFRFT and HGFT. In Section III, we define the proposed HGFRFT and provide its properties and examples. In Sections IV and V, we introduce the HGFRFT-based filters and sampling theorem, respectively. In Section VI, we demonstrate the superiority of the HGFRFT in fractional sampling, frequency analysis and denoising through simulation experiments and three applied numerical experiments. In Section VII, we conclude the paper.

\textit{Notations:} The symbols $\otimes$ and $\circ$ signify the tensor product and function composition, respectively, while $\cong$ indicates isomorphism. $\mathbf{I}_{N}$ represents an identity matrix of size $N$; without specific size requirements, we abbreviate it as $\mathbf{I}$. For square matrices $\mathbf{A} \in \mathbb{C}^{N\times N}$ and $\mathbf{B} \in \mathbb{C}^{M\times M}$, $\mathbf{A}\otimes_{K} \mathbf{B}$ symbolizes their Kronecker product. Here, $\otimes_{K}$ applies to the matrix (thus viewed as a second-order tensor) and exemplifies a particular type of tensor product. $\mathbf{A}^{\top}$, $\mathbf{A}^{\text{H}}$, and $\mathbf{A}^{\dag}$ denote the transpose, conjugate transpose, and pseudo-inverse of matrix $\mathbf{A}$, respectively. The space $L^{2}(\Omega)$ is defined over the measure space $(\Omega, \mathcal{F}, \mu)$, and $\mathbb{C}^{N}$ comes equipped with the inner product $\left<\cdot, \cdot\right>_{\mathbb{C}^{N}}$ and the associated norm $||\cdot||_{\mathbb{C}^{N}}$. This space encompasses functions $f : \Omega \rightarrow \mathbb{C}$ for which the integral $\int_{\Omega} \left| f\right|^{2} \mathrm{d} \mu$ is finite. Other notations used in this paper are summarized in Table \ref{tab1}.
\begin{table}
	\caption{Summary of Notations Used: $f$ Represents Any Vector, $\Psi = \left\{ \psi_{i} \right\}$ and $\Phi = \left\{ \phi_{j} \right\}$ Denote Arbitrary Orthonormal Bases, and $\alpha, \beta \in \mathbb{R}$ Signify Fractional Orders. $\mathcal{H}_{f}$ Refers to the DFT of $f$ in Hilbert Space, $\mathcal{G}_{f}$ Represents the GFT of $f$, and $\mathcal{F}_{f}$ Stands for the HGFT of $f$}
	\label{tab1}
	\begin{tabular}{c|l}
		\hline
		Symbol&Description\\
		\hline
		$\mathcal{H}$ & A Hilbert space\\
		\hline
		$\mathcal{G}$ & A simple, finite, undirected, weighted graph\\
		\hline
		\multirow{2}{*}{$S(\mathcal{H},\mathcal{G})$} & A separable Hilbert space that is isomorphic\\
		& to $\mathcal{H} \otimes \mathbb{C}^{N}$\\
		\hline
		\multirow{2}{*}{$\Psi \otimes \Phi$} & The tensor product of any two orthonormal\\
		& bases\\
			\hline
		$\mathcal{H}^{\alpha}_{f}$ & The DFRFT of $f$ in Hilbert space\\
			\hline
		$\mathcal{G}^{\beta}_{f}$ & The GFRFT of $f$ \\
		\hline
		$\mathcal{F}^{\alpha,\beta}_{f}$ & The HGFRFT of $f$\\
		\hline
		\multirow{2}{*}{$\Psi^{\alpha}$} & The fractional rotation of the orthonormal\\
		& basis $\Psi$ by the fractional order $\alpha$\\
		\hline
		\multirow{2}{*}{$\Psi^{\alpha}\otimes \Phi^{\beta}$ } & The tensor product of any two orthonormal\\
		& bases rotated by different fractional orders\\
		\hline
		\multirow{2}{*}{$(\psi^{\ast}_{i})^{\alpha}$} & The fractional rotation of $\alpha$ after conjugation\\
		& of the orthonormal basis $\psi$\\
		\hline
		\multirow{3}{*}{$(\psi^{\ast}_{i})^{\alpha} \otimes (\phi^{\ast}_{j})^{\beta}$} & The tensor product of any two orthonormal\\
		& bases after conjugation and rotation of \\
		& different fractional orders\\
		\hline
	$\overline{\Psi}^{\alpha}$ & A finite subset of $\Psi^{\alpha}$ \\
		\hline
	\multirow{2}{*}{$\overline{\Psi}^{\alpha} \otimes \overline{\Phi}^{\beta}$} & The tensor product of two different finite \\
	& subsets\\
	\hline
	$|\Psi^{\alpha}|$ & The dimension of the orthonormal basis $\Psi^{\alpha}$ \\
		\hline
	\end{tabular}
\end{table}

\section{Preliminaries}
\label{Preliminaries}
In traditional GSP, the signals on the vertices of $\mathcal{G}=(\mathcal{V},\mathcal{E})$ are assumed to be real or complex, where $\mathcal{V}$ is a vertex and $\mathcal{E}$ is an edge. Before extending the GFRFT to Hilbert space, we review the basic concepts and definitions of graph signals in Hilbert space, the HGFT, and the GFRFT.

\subsection{Graph Signals in Hilbert Space}
Consider $\mathcal{H}$ as a Hilbert space, which is a complete space with an inner product and $\mathcal{G}$ as a simple, finite, undirected, weighted graph. Let a function $f: \mathcal{V} \rightarrow \mathcal{H}$ represent a graph signal within $\mathcal{H}$, and denote the space of such graph signals as $S(\mathcal{H}, \mathcal{G})$. In many instances of physical signals, $\mathcal{H}$ is considered separable if the signal possesses a countable orthonormal basis \cite{Hilbert}. For a given measurable space $\Omega$ equipped with measure $\mu$, $\mathcal{H}$ can be identified with $L^{2}(\Omega)$. For any $f\in S(\mathcal{H}, \mathcal{G})$ and any vertex $v \in \mathcal{V}$, the function $f(v) \in L^{2}(\Omega)$, and $f$ can be interpreted as a function of two variables, $f(x, v) = f(x)(v) \in \mathbb{C}$, for each $x \in \Omega$.

For a finite-dimensional vector space $\mathbb{C}^{N}$, it is possible to construct a tensor product $\mathcal{H}\otimes \mathbb{C}^{N}$. This leads to the formation of $x\otimes v$ which upholds the following relations \cite{Algebra} for all $x, x_{i} \in \mathcal{H}, v, v_{i} \in \mathbb{C}^{N}, r \in \mathbb{R}$, where $i = 1, ..., n$.
\begin{enumerate}
	\item $(x_1 + x_2) \otimes v=x_1 \otimes v + x_2 \otimes v$;
	\item$x \otimes (v_1 + v_2)=x \otimes v_1 + x \otimes v_2$;
	\item$xr \otimes v = x \otimes rv$.
\end{enumerate}

To establish a metric on $\mathcal{H} \otimes \mathbb{C}^{N}$, since $\mathbb{C}^N$ is of finite dimension, the space $\mathcal{H} \otimes \mathbb{C}^{N}$ achieves completeness and thus qualifies as a Hilbert space. Equipping the tensor product $\mathcal{H} \otimes \mathbb{C}^{N}$ with the inner product $\left< \cdot ,\cdot \right>_{\mathcal{H} \otimes \mathbb{C}^{N}}$ allows for the derivation of the following properties
\begin{equation}
	\left< x_{1}\otimes v_{1},x_{2}\otimes v_{2}\right>_{\mathcal{H} \otimes \mathbb{C}^{N}  }  =\left< x_{1},x_{2}\right>_{\mathcal{H}}  \left<v_{1},v_{2} \right>_{\mathbb{C}^{N}}.\label{CnH}
\end{equation}

The signal space $S(\mathcal{H}, \mathcal{G})$ under examination is notably complex. Nonetheless, it is isomorphic to $\mathcal{H} \otimes \mathbb{C}^{N}$, as indicated in \cite{Algebra}. This relationship allows for $S(\mathcal{H}, \mathcal{G})$ to be expressed through a specific mapping
\begin{equation}
	\varphi(f) = \sum_{v\in \mathcal{V}} f(v) \otimes v, \quad f \in S(\mathcal{H}, \mathcal{G}), \label{varphi}
\end{equation}
where $v \in \mathbb{C}^N$, and $\mathcal{V}$ is represented using the standard basis of $\mathbb{C}^N$. This mapping is achieved by constructing a tensor product, and $S(\mathcal{H}, \mathcal{G})$ is also a separable Hilbert space.

The mapping $\varphi$, as described by the equation above, facilitates the translation of the structural properties of $\mathcal{H} \otimes \mathbb{C}^{N}$ to the Hilbert space $S(\mathcal{H}, \mathcal{G})$. For functions $f, g \in S(\mathcal{H}, \mathcal{G})$, an inner product is defined as follows
\begin{equation}
	\left< f, g \right> = \left< \varphi(f), \varphi(g) \right>_{\mathcal{H} \otimes \mathbb{C}^{N}}.
\end{equation}
Given the isomorphism between $S(\mathcal{H}, \mathcal{G})$ and $\mathcal{H} \otimes \mathbb{C}^{N}$, any function $f \in S(\mathcal{H}, \mathcal{G})$ can also be considered an element of $\mathcal{H} \otimes \mathbb{C}^{N}$.

\subsection{Graph Fourier Transform in Hilbert Space}
The HGFT represents the DFT within the broader context of generalized GSP \cite{Ggeneralized}. This approach extends the concept of the frequency domain to the realm of generalized GSP. Let $\Psi = \left\{ \psi_{i} \right\}_{i \geq 1}$ be an orthonormal basis for $\mathcal{H}$ and $\Phi = \left\{ \phi_{j} \right\}_{1 \leq j \leq N}$ for $\mathbb{C}^N$. Then, $\left\{ \psi_{i} \otimes \phi_{j} \right\}_{i \geq 1, 1 \leq j \leq N}$ forms the orthonormal basis of $\Psi \otimes \Phi$ in the signal space $S(\mathcal{H}, \mathcal{G})$. For each $f \in S(\mathcal{H}, \mathcal{G})$, the joint transform, or the HGFT, is defined as
\begin{equation}
	\mathcal{F}_{f}\left( \psi \otimes \phi \right) = \left< f,\psi \otimes \phi \right>  := \left< \varphi \left( f\right), \psi \otimes \phi \right>_{\mathcal{H} \otimes \mathbb{C}^{N}},
\end{equation}
where $\varphi$ is the isomorphic mapping in Eq. \eqref{varphi}, and the inner product is as defined in Eq. \eqref{CnH}. According to Parseval's theorem, for all $f \in S(\mathcal{H}, \mathcal{G})$, $\sum_{\psi \otimes \phi} \left| \mathcal{F}_{f}\left( \psi \otimes \phi \right) \right|^{2} < \infty$.

Since signals in the Hilbert space $\mathcal{H}$ have their own transform space, the HGFT can be decomposed into simpler components, known as partial HGFTs, which include the DFT in $\mathcal{H}$ and the GFT \cite{JFT,JFRFT,JLCT,Gclustering}. Given that $\mathcal{V}$ is linked to the standard basis of $\mathbb{C}^{N}$, thus $f(x, \cdot) = \sum_{v \in \mathcal{V}} f(x, v) v \in \mathbb{C}^{N}$. Therefore, for $v \in \mathcal{V}$ and $x \in \Omega$, the partial HGFTs are defined as
\begin{equation}
	\text{DFT:} \quad \mathcal{H}_{f}(\psi)(v) = \left< f(\cdot, v), \psi \right>_{\mathcal{H}},
\end{equation}
\begin{equation}
	\text{GFT:} \quad \mathcal{G}_{f}(\phi)(x) = \left< f(x, \cdot), \phi \right>_{\mathbb{C}^{N}}.\label{Gfphi}
\end{equation}
Furthermore, the following relations exist
\begin{equation}
	\mathcal{F}_{f}\left( \psi \otimes \phi \right)  =\left< \mathcal{H}_{f}\left( \psi \right)  ,\phi \right>_{\mathbb{C}^{N}}  =\left< \mathcal{G}_{f}\left( \phi \right)  ,\psi \right>_{\mathcal{H}}  .
\end{equation}

The inverse HGFT follows from the properties of the orthonormal basis $\Psi \otimes \Phi$
\begin{equation}
	\mathcal{F}^{-1}_{f} = \sum_{\psi \otimes \phi} f(\psi \otimes \phi) \cdot \psi \otimes \phi,
\end{equation}
where $\sum_{\psi \otimes \phi} \left|  f(\psi \otimes \phi) \right| < \infty$.

When defining HGFT in conjunction with graphs, the orthonormal basis $\Phi$ is typically chosen as the eigenbasis of a symmetric graph shift operator $\mathbf{A}$. Additionally, a bounded, self-adjoint operator $\mathbf{B}$ can be defined such that for any $x, y \in \mathcal{H}$, it holds that $\left< \mathbf{B}x, y\right>_{\mathcal{H}} = \left< x, \mathbf{B}y\right>_{\mathcal{H}}$. As a result, all eigenvalues of $\mathbf{B}$ are real, and $\mathcal{H}$ has an orthonormal basis composed of the eigenvectors of $\mathbf{B}$.

\subsection{Graph Fractional Fourier Transform}
The GFT \cite{GFTlaplace,GFTadjacency1,GFTadjacency2} is defined by eigendecomposition of adjacency matrix $\mathbf{A=V\Lambda_{A}V}^{-1}$, or the Laplace matrix $\mathcal{L}=\mathbf{\chi} \mathbf{\Lambda}_{\mathcal{L}} \mathbf{\chi}^{-1}$, where the columns $\{ \phi_{j} \}^{N-1}_{j=0}$ of $\mathbf{V}$ are the eigenvectors of $\mathbf{A}$, and the eigenvalues in the corresponding diagonal matrix $\mathbf{\Lambda_{A}}$ are $\{ \lambda_{\mathbf{A}_{j}}\}^{N-1}_{j=0}$. The columns of $\mathbf{\chi}$ are the eigenvectors of $\mathcal{L}$, corresponding to the eigenvalues $
\{ \lambda_{\mathcal{L}_{j}}\}^{N-1}_{j=0}$ in the diagonal matrix $\mathbf{\Lambda}_{\mathcal{L}}\in\mathbb{C}^{N\times N}$.

In a cyclic graph\footnote{Despite being directed, the adjacency matrix of a cyclic graph is circulant, allowing unitary diagonalization, thus satisfying $\mathbf{V}^{-1} = \mathbf{V}^{\text{H}}$.}, the GFT corresponds to the DFT in the time domain. To align it with the DFT in Hilbert space $\mathcal{H}$ and extend its applicability, the adjacency matrix method is used to define the GFT of graph signal $f$ \cite{GFTadjacency2}
\begin{equation}
\hat{f}=\mathbf{F}_{\mathcal{G}}f=\mathbf{V}^{-1}f,
\end{equation}
which, consistent with Eq. \eqref{Gfphi}, can also be expressed as an inner product
\begin{equation}
\hat{f}(\lambda_{\mathbf{A}_j}) := \left< f, \phi_j \right> = \sum^{N}_{\ell=1} f(\ell) \phi^{\ast}_j(\ell).\label{GFT}
\end{equation}

Thus, the $\beta$th GFRFT of a graph signal $f \in \mathbb{C}^{N \times 1}$ is defined as \cite{GFRFT,GFRFTspectral}
\begin{equation}
	\hat{f}= \mathbf{F}^{\beta }_{\mathcal{G}}f := \mathbf{Q}\mathbf{\Lambda}^{\beta } \mathbf{Q}^{-1}f,\  \  \beta \in \mathbb{R}, \label{GFRFT}
\end{equation}
where the matrices $\mathbf{Q}$ and $\mathbf{\Lambda}$ arise from the spectral decomposition of the orthogonal and diagonalized GFT matrix $\mathbf{V}^{-1}$, such that $\mathbf{F}_{\mathcal{G}} = \mathbf{V}^{-1} = \mathbf{Q \Lambda Q }^{-1}$. According to Eq. \eqref{GFT}, the GFRFT can also be expressed as the inner product
\begin{equation}
\hat{f}(\lambda_j) = \left< f, \phi^{\beta}_j \right> = \sum^{N}_{\ell=1} f(\ell)( \phi^{\ast}_j )^{\beta}(\ell), \label{phibeta}
\end{equation}
where $( \phi^{\ast}_j)^{\beta}$ is formed by the columns of the matrix $\mathbf{F}^{\beta }_{\mathcal{G}}$. Its inverse transform is given by
\begin{equation}
f(\ell) = \sum^{N-1}_{j=0} \hat{f}(\lambda_j) \phi^{\beta}_j(\ell).
\end{equation}
%\begin{remark}\label{remark1}
%When the adjacency matrix is non-diagonalizable, the GFT matrix can be defined using the Jordan decomposition $\mathbf{A} = \mathbf{VJ_A V}^{-1}$. Thus, both GFT and GFRFT are characterized by the Jordan eigenvectors. Notably, for a power function of order $\beta$, the $k$-th superdiagonal element is $\tbinom{\beta}{k} \lambda^{\beta-k}$, where the binomial coefficient is defined as $\tbinom{\beta}{k} = \prod_{i=1}^{k} \frac{\beta+1-i}{i}$  \cite{GFRFT}.
%\end{remark}

The uniqueness of the fractional powers are presented in Appendix \ref{appendixA}. It is easy to verify that when $\beta =0$, the GFRFT of the graph signal is the signal itself. When $\beta =1$, the GFRFT is reduced to the GFT.

\section{Graph Fractional Fourier Transform in Hilbert Space}
\label{Fa-transform}
To introduce a novel transformation method, similar to how the DFT is embedded in generalized GSP, we extend the DFRFT into this framework, referring to it as the HGFRFT. This approach defines the frequency domain within the generalized fractional GSP, increasing flexibility and enabling clear observation of transformations between the vertex and graph spectral domains—key interactions in GSP. Since signals in the Hilbert space $\mathcal{H}$ have their own transformation space, like the HGFT, the HGFRFT can be decomposed into partial components, referred to as partial HGFRFTs.

\subsection{Definition}
\label{3.1}
Before defining the HGFRFT, we give a few assumptions:
\begin{enumerate}
	\item $\mathbf{A}$ is a unitarily diagonalizable graph shift operator on the graph $\mathcal{G}$.
	\item $\mathbf{B}$ is a compact, self-adjoint, and injective operator acting on the Hilbert space $\mathcal{H}$, which can be infinite-dimensional. \label{A2}
	\item $\Psi$ is an orthonormal basis of $\mathcal{H}$ consisting of the eigenvectors of $\mathbf{B}$. $\Phi$ is an orthonormal basis of $\mathbb{C}^{N}$ consisting of the eigenvectors of $\mathbf{A}$, which we will later restrict to finite dimensions for illustrative examples. \label{A3}
\end{enumerate}

By the \textit{spectral theorem} \cite{Functional}, $\mathbf{B}$ has an eigendecomposition, guaranteeing an orthonormal eigenbasis and enabling a rigorous extension of the transform to infinite-dimensional Hilbert spaces. The proof of Assumption \ref{A2} is provided in Appendix \ref{appendixB}. Since the DFRFT is equivalent to the GFRFT for cyclic graphs, we extend its applicability using the spectral theorem, which ensures an orthonormal basis for $\mathcal{H}$, thus allowing the DFRFT to extend beyond finite-dimensional graphs. For a signal $f \in \mathbb{C}^{N \times 1}$, the $\alpha$th DFRFT \cite{DFRFT2} is defined as
\begin{equation}
	\hat{f} (\lambda_{i} )=\left< f,\psi^{\alpha }_{i} \right>  =\sum^{N}_{\ell =1} f(\ell )(\psi^{\ast }_{i} )^{\alpha }(\ell ),\ \ \alpha \in \mathbb{R}, \label{DFRFT}
\end{equation}
where $\lambda_{i}$ denotes the eigenvalue of $\mathbf{B}$, and the definition of $(\psi^{\ast }_{i} )^{\alpha }$ follows the same form as $( \phi^{\ast}_j)^{\beta}$ in Eq. \eqref{phibeta}.

Similar to the construction of the HGFT, we define the order pair $(\alpha, \beta)$ such that $\Psi^{\alpha} = \{ \psi^{\alpha}_{i} \}_{i\geq 1}$ is an orthonormal basis for $\mathcal{H}$, and $\Phi^{\beta} = \{ \phi^{\beta}_{j}\}_{1 \leq j \leq N}$ serves as an orthonormal basis for $\mathbb{C}^N$. Together, $\{ \psi^{\alpha}_{i} \otimes \phi^{\beta}_{j} \}_{i\geq 1, 1 \leq j \leq N}$ forms the orthonormal basis of $\Psi^{\alpha} \otimes \Phi^{\beta}$ in the Hilbert space $\mathcal{H} \otimes \mathbb{C}^{N}$.

Additionally, we assume that the basis functions $(\psi_i^{\ast})^{\alpha}$ and $(\phi_j^{\ast })^{\beta}$, adjusted by parameters $\alpha$ and $\beta$, are defined as complete, self-adjoint, and dense in their respective Hilbert spaces. This assumption is based on the original properties of $\psi_i$ and $\phi_j$ (cf. Assumption \ref{A3}). Specifically, since $\psi_i$ and $\phi_j$ form orthonormal bases in $\mathcal{H}$ and $\mathbb{C}^N$, respectively, their orthogonality, completeness, and density are preserved under conjugation and parameter adjustment. Hence, $(\psi_i^*)^{\alpha}$ and $(\phi_j^*)^{\beta}$ retain these properties.

We then define the operator $\mathcal{F}^{\alpha,\beta}(\cdot)$ as an equivalent transform between $S(\mathcal{H}, \mathcal{G})$ spaces:
\begin{defn}
	For $f \in S\left( \mathcal{H} ,\mathcal{G} \right)$, the HGFRFT of $f$ is defined as
	\begin{equation}
		\hat{f }  =\mathcal{F}^{\alpha ,\beta }_{f}=\left< f,\psi^{\alpha} \otimes \phi^{\beta} \right>  :=\left< \varphi \left(f\right) , \psi^{\alpha} \otimes \phi^{\beta} \right>_{\mathcal{H}\otimes \mathbb{C}^{N}},\label{Falphabeta}
	\end{equation}
	where $\varphi$ denotes the isomorphic mapping as defined in Eq. \eqref{varphi}, and the inner product on the right-hand side follows Eq. \eqref{CnH}. Additionally, $\sum_{\psi^{\alpha} \otimes \phi^{\beta}} | \mathcal{F}^{\alpha, \beta}_{f} |^{2} < \infty$.
\end{defn}
\begin{remark}
The HGFRFT can also be represented in terms of matrix operators for the finite-dimensional case:
	\begin{equation}
		\hat{f } =\mathcal{F}^{\alpha ,\beta }_{f }=\mathbf{F}^{\alpha ,\beta } f=\left( \mathbf{F}^{\alpha }_{\mathcal{H} } \otimes_{K} \mathbf{F}^{\beta }_{\mathcal{G} } \right)  f, \label{FHCN}
	\end{equation}
	where $\otimes_{K}$ denotes the Kronecker product, and the matrices $\mathbf{F}^{\alpha}_{\mathcal{H}}$ and $\mathbf{F}^{\beta}_{\mathcal{G}}$ are composed of the orthonormal bases $(\psi^{\ast}_{i})^{\alpha}$ and $(\phi^{\ast}_{j})^{\beta}$, respectively. However, in infinite-dimensional Hilbert spaces, this should be understood in terms of operators acting on the respective Hilbert spaces.
\end{remark}

The HGFRFT of $f$ can be decomposed into simpler parts, called partial HGFRFTs, which we define as
\begin{equation}
	\mathcal{H}^{\alpha }_{f}   \left( v\right)  =\left< f \left( \cdot ,v\right)  ,\psi^{\alpha } \right>_{\mathcal{H} } , \label{Halpha}
\end{equation}			
\begin{equation}
	\mathcal{G}^{\beta }_{f}  \left( x\right)  =\left< f \left( x,\cdot \right)  ,\phi^{\beta } \right>_{\mathbb{C}^{N} }  , \label{Gbeta}
\end{equation}
for every $v \in \mathbb{C}^{N}$ and $x \in \mathcal{H}$. Eq. \eqref{Halpha} represents the DFRFT, and Eq. \eqref{Gbeta} corresponds to the GFRFT. Both can also be expressed in matrix form as $\mathcal{H}^{\alpha }_{f} = \mathbf{F}^{\alpha}_{\mathcal{H}} f$, and $	\mathcal{G}^{\beta }_{f} = \mathbf{F}^{\beta}_{\mathcal{G}} f$. These matrix representations are specific to finite dimensions, while in infinite-dimensional spaces, the respective operators replace the matrix form.

\subsection{Properties}
In this subsection, we introduce the important properties, propositions and corollaries of the proposed HGFRFT and give corresponding proofs.

\begin{pro}
	(Nestability) $\mathcal{H}^{\alpha}_{f} \in \mathbb{C}^{N}$ and $\mathcal{G}^{\beta}_{f} \in \mathcal{H}$.
\end{pro}
\begin{proof}
We can write $\mathcal{H}^{\alpha }_{f } =\sum_{v\in \mathcal{V}} \mathcal{H}^{\alpha }_{f} (v)v\in \mathbb{C}^{N} $. We use the standard basis of $\mathbb{C}^{N}$ to identify $\mathcal{V}$ and consider $\mathcal{G}^{\beta}_{f}$ as a mapping $\Omega \rightarrow \mathbb{C}$. From the Cauchy-Schwarz inequality \cite{Ggeneralized}, we have
	\begin{equation*}
		\begin{aligned}\int_{\Omega } \left| \mathcal{G}^{\beta }_{f } \left( x\right)  \right|^{2}  \mathrm{d} \mu \left( x\right)  \leqslant &\int_{\Omega } \left| \left| f \left( x,\cdot\right)  \right|  \right|^{2}_{\mathbb{C}^{N} }  \mathrm{d} \mu \left( x\right)  \\ =\int_{\Omega } \sum_{v\in \mathcal{V}} \left| f \left( x,v\right)  \right|^{2}  \mathrm{d} \mu \left( x\right)  =&\sum_{v\in \mathcal{V}} \int_{\Omega } \left| f \left( x,v\right)  \right|^{2}  \mathrm{d} \mu \left( x\right)  <\infty ,\end{aligned} 
	\end{equation*}
	where the first equality follows from Parseval’s formula, and the last inequality is because $f (x,\cdot ) \in \mathcal{H} = L^{2}(\Omega)$ and $\mathcal{V}$ is finite. Therefore, $\mathcal{G}^{\beta}_{f} \in \mathcal{H}$.
\end{proof}

\begin{pro}
	(Associativity) According to Eqs. \eqref{FHCN}, \eqref{Halpha}, and \eqref{Gbeta}, the following relationship can be obtained
	\begin{equation}
	\mathcal{F}^{\alpha ,\beta }_{f } =\left< \mathcal{H}^{\alpha}_{f}, \phi^{\beta}  \right>_{\mathbb{C}^{N}}  =\left< \mathcal{G}^{\beta}_{f} ,\psi^{\alpha}  \right>_{\mathcal{H}}. \label{FHG}
	\end{equation}
\end{pro}
\begin{proof}
	For operator $\mathcal{H}^{\beta }_{f}$, we have
	\begin{equation*}
		\begin{aligned}\left< \mathcal{H}^{\alpha}_{f}  ,\phi^{\beta}  \right>_{\mathbb{C}^{N}} 
			=&\left< \sum_{v} \mathcal{H}^{\alpha}_{f}  \left( v\right)  v, \phi^{\beta}  \right>_{\mathbb{C}^{N}}  \\
			=&\sum_{v} \mathcal{H}^{\alpha}_{f}  \left( v\right)  \left< v, \phi^{\beta}  \right>_{\mathbb{C}^{N}}  \\ 
			\overset{(a)}{=}&\sum_{v} \left< f\left( v\right)  ,\psi^{\alpha}  \right>_{\mathcal{H}}  \left< v, \phi^{\beta}  \right>_{\mathbb{C}^{N}}  \\
			\overset{(b)}{=}&\sum_{v} \left< f\left( v\right) \otimes v ,\psi^{\alpha}  \otimes \phi^{\beta}  \right>_{\mathcal{H}\otimes \mathbb{C}^{N}}  \\ 
			\overset{(c)}{=}&\left< f,\psi^{\alpha}  \otimes \phi^{\beta}  \right>  =\mathcal{F}^{\alpha,\beta}_{f}.  \end{aligned} 
	\end{equation*}
	Equality $(a)$ is obtained by Eq. \eqref{Halpha}, $(b)$ is obtained by Eq. \eqref{CnH}, and $(c)$ is obtained by Eqs. \eqref{varphi} and \eqref{Falphabeta}. This proves the first equality in Eq. \eqref{FHG}. The second equation can be proved similarly.
\end{proof}

\begin{pro}
	(Zero rotation) $\mathcal{F}_{f}^{0 ,0 }=\left< f,\psi^{0} \otimes \phi^{0} \right>  =\left< f,\mathbf{1} \right> =f$, which is the reduction to the identity matrix.
\end{pro}
\begin{proof}
	 This proof follows from definition of the HGFRFT and the simplification of the identity properties of the DFRFT and the GFRFT.
\end{proof}

\begin{pro}
	(Additivity) The additivity of the DFRFT is expressed as exponential additivity on the fractional order, while the additivity of the HGFRFT can be realized by replacing the matrix by
	\begin{equation}
		\mathcal{F}^{\alpha_{3} ,\beta_{3} }_f=\mathcal{F}^{\alpha_{1} +\alpha_{2} ,\beta_{1} +\beta_{2} }_f=\mathcal{F}^{\alpha_{1} ,\beta_{1} }_f ( \mathcal{F}^{\alpha_{2} ,\beta_{2} }_f )  ,
	\end{equation}
	where $(\alpha_{1}, \beta_{1})$, $(\alpha_{2}, \beta_{2})$ and $(\alpha_{3}, \beta_{3})$ are real-valued pairs, and $\alpha_{1}+\alpha_{2}=\alpha_{3},\beta_{1}+\beta_{2}=\beta_{3}$.
\end{pro}
\begin{proof}
	By definition, we have
	\begin{equation*}
			\begin{aligned}
			\mathcal{F}^{\alpha_{3} ,\beta_{3} }_f &= \left< f, \psi^{\alpha_{3}} \otimes \phi^{\beta_{3}} \right> \\
			&\overset{(a)}{=} \left< f, \left( \psi^{\alpha_{1}} \otimes \phi^{\beta_{1}} \right) \cdot \left( \psi^{\alpha_{2}} \otimes \phi^{\beta_{2}} \right) \right> \\
			&\overset{(b)}{=} \left< f, \psi^{\alpha_{1}} \otimes \phi^{\beta_{1}} \right> \cdot \left< f, \psi^{\alpha_{2}} \otimes \phi^{\beta_{2}} \right> \\
			&= \mathcal{F}^{\alpha_{1} ,\beta_{1} }_f ( \mathcal{F}^{\alpha_{2} ,\beta_{2} }_f ),
		\end{aligned}
	\end{equation*}
where equality $(a)$ follows from the properties of tensor product multiplication, and equality $(b)$ follows from the linearity of the inner product.
\end{proof}

\begin{pro}
	(Commutativity) For two real-valued pairs $(\alpha_{1}, \beta_{1})$ and $(\alpha_{2}, \beta_{2})$, the HGFRFT is commutative such that
	\begin{equation}
		\mathcal{F}^{\alpha_{1} ,\beta_{1} }_f( \mathcal{F}^{\alpha_{2} ,\beta_{2} }_f)  =\mathcal{F}^{\alpha_{2} ,\beta_{2} }_f( \mathcal{F}^{\alpha_{1} ,\beta_{1} }_f)  ,
	\end{equation}
	and the HGFRFT is cross-commutative such that
	\begin{equation}
		\mathcal{F}^{\alpha_{1} ,\beta_{1} }_f( \mathcal{F}^{\alpha_{2} ,\beta_{2} }_f)  =\mathcal{F}^{\alpha_{1} ,\beta_{2} }_f( \mathcal{F}^{\alpha_{2} ,\beta_{1} }_f)  =\mathcal{F}^{\alpha_{2} ,\beta_{1} }_f( \mathcal{F}^{\alpha_{1} ,\beta_{2} }_f) .
	\end{equation}
\end{pro}
\begin{proof}
	This property follows the exponential additive property
	\begin{equation*}
		\mathcal{F}^{\alpha_{1} ,\beta_{1} }_f ( \mathcal{F}^{\alpha_{2} ,\beta_{2} }_f )  =\mathcal{F}^{\alpha_{1} +\alpha_{2} ,\beta_{1} +\beta_{2} }_f=\mathcal{F}^{\alpha_{2} ,\beta_{2} }_f ( \mathcal{F}^{\alpha_{1} ,\beta_{1} }_f).
	\end{equation*} 
	The cross-commutativity property also follows from exponential additivity
	\begin{equation*}
		\begin{aligned}\mathcal{F}^{\alpha_{1} ,\beta_{1} }_f ( \mathcal{F}^{\alpha_{2} ,\beta_{2} }_f )  =&\mathcal{F}^{\alpha_{1} +\alpha_{2} ,\beta_{1} +\beta_{2} }_f \\ =&\mathcal{F}^{\alpha_{1} ,\beta_{2} }_f ( \mathcal{F}^{\alpha_{2} ,\beta_{1} }_f)  \\ =&\mathcal{F}^{\alpha_{2} ,\beta_{1} }_f ( \mathcal{F}^{\alpha_{1} ,\beta_{2} }_f)  ,\end{aligned} 
	\end{equation*} 
	where the last equality is from commutativity.
\end{proof}

\begin{pro}
	(Invertibility) The invertibility of the HGFRFT means that the inverse HGFRFT can be realized by another HGFRFT whose parameter matrix is equal to the inverse of the forward transformation matrix, and the invertibility can be obtained through the abovementioned additivity. For an order pair $(\alpha, \beta)$, we have
	\begin{equation}
		\mathcal{F}^{-\alpha ,-\beta }_f( \mathcal{F}^{\alpha ,\beta }_f)  =\mathcal{F}^{0 ,0}_f=f.
	\end{equation}
\end{pro}
\begin{proof}
	This follows from the exponential additivity and reduction of the HGFRFT identity matrix.
\end{proof}

\begin{remark}
The inverse transform of the HGFRFT can be obtained through its invertibility, as follows
\begin{equation}
		f=( \mathcal{F}^{\alpha ,\beta }_{\hat{f} } )^{-1} \overset{(a)}{=} \mathcal{F}^{-\alpha ,-\beta }_ {\hat{f}}=\left<\hat{f},\psi^{-\alpha}\otimes \phi^{-\beta} \right> ,
\end{equation}
where the equality $(a)$ follows from the invertibility of the tensor product, that is, $(\psi^{\alpha}\otimes \phi^{\beta})^{-1}=(\psi^{-\alpha}\otimes \phi^{-\beta})$.
\end{remark}

\begin{pro}
	(Separability) The HGFRFT is separable on $\mathbb{C}^{N}$ and $\mathcal{H}$; in other words, it is separable on the GFRFT and the DFRFT. That is, for a given order $(\alpha, \beta)$, we have
	\begin{equation}
		\mathcal{F}^{\alpha ,\beta }_f=\mathcal{F}^{0,\beta }_f( \mathcal{F}^{\alpha ,0}_f)  =\mathcal{F}^{\alpha ,0}_f( \mathcal{F}^{0,\beta }_f)  .\label{Separable}
	\end{equation}
\end{pro}
\begin{proof}
The first equality follows from exponential additivity, while the second follows from commutativity. By incorporating Eq. \eqref{FHG}, we can also obtain $\mathcal{F}^{0 ,\beta }_f= \mathcal{G}^{\beta}_f$ and $\mathcal{F}^{\alpha ,0 }_f= \mathcal{H}^{\alpha}_f$. In particular, using reduction from the identity properties of the GFRFT and DFRFT and the definition of the HGFRFT, we obtain Eq. \eqref{Separable} as
	\begin{equation*}
		\begin{aligned}	\mathcal{F}^{\alpha ,\beta }_{f}=&\left< f,\left( \psi^{\alpha } \otimes \mathbf{1}\right)  \left( \mathbf{1}\otimes \phi^{\beta } \right)  \right> \\ 
		=&\left< f,\left( \psi^{\alpha } \otimes \mathbf{1}\right)\right> \cdot\left< f, \left( \mathbf{1}\otimes \phi^{\beta } \right)  \right> \\ 
		=&\mathcal{F}^{\alpha,0}_{f}(\mathcal{F}^{0,\beta}_f)=\mathcal{F}^{0,\beta}_f (\mathcal{F}^{\alpha,0}_{f}) .\end{aligned}
	\end{equation*}
	This clearly demonstrates the separability of the HGFRFT.
\end{proof}

\begin{pro}
	(Reducibility) When $\alpha=\beta=1$, the HGFT becomes a special case of the HGFRFT.
\end{pro}
\begin{proof}
The ordinary transform properties can be simplified to
\begin{equation}
\mathcal{F}^{1,1}_f=\left< f, \psi^{1}\otimes \phi^{1}\right>=\mathcal{F}_{f}(\psi\otimes \phi).
\end{equation}
This also follows the definition of the HGFRFT and the simplification of the general transform properties of the DFRFT and the GFRFT.
\end{proof}

\begin{pro}
	(Unitary) The HGFRFT is a unitary transform, meaning that it preserves the inner product between functions. For any signal $f$ and $g$, and for any $(\alpha,\beta)$
	\begin{equation}
		\langle \mathcal{F}^{\alpha,\beta}_f, \mathcal{F}^{\alpha,\beta}_g \rangle = \langle f, g \rangle.\label{unitary}
	\end{equation}
\end{pro}
\begin{proof}
Let $\psi^{\alpha}$ and $\phi^{\beta}$ be appropriately chosen basis functions, and $\psi^{\alpha} \otimes \phi^{\beta}$ be the tensor product basis functions. To demonstrate that the HGFRFT preserves the inner product, as stated in Eq. \eqref{unitary}, consider
\[
\langle \mathcal{F}^{\alpha,\beta}_f, \mathcal{F}^{\alpha,\beta}_g \rangle = \left< \left< f, \psi^{\alpha} \otimes \phi^{\beta} \right>, \left< g, \psi^{\alpha} \otimes \phi^{\beta} \right> \right>
\]
Given that $\psi^{\alpha} \otimes \phi^{\beta}$ forms an orthonormal basis
\[
\langle \mathcal{F}^{\alpha,\beta}_f, \mathcal{F}^{\alpha,\beta}_g \rangle = \left< f, \left< \psi^{\alpha} \otimes \phi^{\beta}, \psi^{\alpha} \otimes \phi^{\beta} \right> \cdot g \right> = \left< f, g \right>
\]
where $\left< \psi^{\alpha} \otimes \phi^{\beta}, \psi^{\alpha} \otimes \phi^{\beta} \right> = 1$ due to the normalization of the basis functions. This shows that the HGFRFT preserves the inner product between vectors, thus confirming that the HGFRFT is a unitary transformation.
\end{proof}

\begin{remark}
	When the underlying graph is a cyclic graph, the GFRFT simplifies to the DFRFT \cite{cycle}. Consequently, in the generalized GSP framework, one of the domains becomes the time domain, with the adjacency matrix represented as
	\[
	\mathbf{A} = \begin{bmatrix}
		& & & 1 \\
		1 & & & \\
		& \ddots & & \\
		& & 1 & 
	\end{bmatrix}.
	\]
\end{remark}

Then, we have the following proposition.
\begin{proposition}
	If the adjacency matrix $\mathbf{A}$ is similarly constructed as compact self-adjoint operators, then $\mathcal{F}^{\alpha,\beta}_f$ reduces to the two-dimensional DFRFT of orders $\alpha$ and $\beta$ in Hilbert space. \label{pro1}
\end{proposition}
\begin{proof}
	For $\mathcal{F}^{\alpha,\beta}_f = \left< f, \psi^{\alpha} \otimes \phi^{\beta} \right>$, when $ \mathbf{A}$ is a compact self-adjoint operator, let $\Psi'$ be an orthonormal basis of $ \mathcal{H}$ consisting of the eigenvectors of $\mathbf{A}$. Then, $(\Psi')^{\beta} = \{ (\psi')^{\beta}_j \}_{j \geq 1}$, and consequently $\mathcal{F}^{\alpha,\beta}_f = \left< f, \psi^{\alpha} \otimes (\psi')^{\beta} \right>$, which corresponds to the two-dimensional DFRFT in Hilbert space.
\end{proof}
\begin{corollary}
	Since the HGFRFT satisfies the fundamental properties of the GFRFT, if $ \mathcal{F}^{\alpha,\beta}_{f}$ corresponds to the same graph structure as $ \mathcal{G} $ in $ \mathcal{H} $ (i.e., $ \mathbf{B} = \mathbf{A} $), then $ \mathcal{F}^{\alpha,\beta}_{f} $ is a two-dimensional GFRFT of orders $ \alpha $ and $ \beta $ in Hilbert space.
\end{corollary}
\begin{proof}
	Following the proof of Proposition \ref{pro1}, $ \mathcal{H} $ is equivalent to $ \mathcal{G} $. Therefore, $ \mathcal{F}^{\alpha,\beta}_{f} = \left< f, (\phi')^{\alpha} \otimes \phi^{\beta} \right> $, where $ (\Phi')^{\alpha} = \{ (\phi')^{\alpha}_i \}_{1 \leq i \leq N} $ is an orthonormal basis of $ \mathbb{C}^{N} $ consisting of the eigenvectors of $ \mathbf{B} $.
\end{proof}

The HGFRFT exhibits properties similar to the HGFT, such as nestability, associativity, invertibility, separability, and unitarity. Additionally, it possesses new characteristics, including zero rotation, additivity, commutativity, and reducibility. In the following, we will illustrate the application of HGFRFT within the generalized GSP framework through examples.

\subsection{Examples}\label{sec3.1}
In this subsection, we present several examples to demonstrate the implementation of our proposed HGFRFT. Specifically, we explore its application in traditional GFRFT contexts, within finite closed intervals, and across both finite-dimensional and infinite-dimensional Hilbert spaces.
\begin{example}
Traditional GFRFT
\end{example}
In this example, we consider $\mathcal{H} = \mathbb{C}$, where $f \in S\left( \mathcal{H}, \mathcal{G}\right)$ represents a traditional graph signal with each vertex $v \in \mathcal{V}$ assigned a complex value $f(v) \in \mathbb{C}$. Here, $f$ is a complex-valued signal defined on the vertex set of graph $\mathcal{G}$. By substituting the GFT matrix $\mathbf{V}^{-1}$ with $\mathbf{Q\Lambda}^{\beta}\mathbf{Q}^{-1}$, we integrate the FRFT into the GSP framework \cite{GFRFT}. Traditional GFRFT examples, such as the path graph and ring graph, are illustrated in Fig. \ref{fig3} of Section \ref{sec6.1}.

\begin{example}
GFRFT in Finite Closed Intervals
\end{example} 
We consider $\mathcal{H} = L^{2}\left( [a,b] \right)$, representing the space of complex-valued $L^{2}$ functions defined on a finite closed interval $[a, b]$ \cite{JFT}. Setting $a = 0$ and $b = 2\pi$, the graph signal $f$ maps $L^{2}$ functions to the vertices of graph $\mathcal{G}$. Thus, $f$ maps from the interval $[0, 2\pi]$ to the vertex set $\mathcal{V}$. For any $v \in \mathcal{V}$, the $\mathcal{H}^{\alpha}$ transform of $f(v)$ is its DFRFT. Similarly, the $\mathcal{G}^{\beta}$ transform of $f(x, \cdot)$ represents the GFRFT of the graph signal. Interpreting $[a, b]$ as time, HGFRFT extends to the joint time-vertex fractional Fourier transform (JFRFT) \cite{JFRFT}. Examples in the time-vertex graph fractional domain are provided in Section \ref{sec6.2}.

\begin{example}
	GFRFT in Finite-Dimensional Hilbert Space
\end{example} 
Consider $\mathcal{H} = L^2(\mathcal{G}^{\prime})$, where $\mathcal{G}^{\prime} = (\mathcal{V}^{\prime}, \mathcal{E}^{\prime})$ is a secondary finite graph with a discrete measure. This space $\mathcal{H}$ contains finite complex graph signals defined on $\mathcal{G}^{\prime}$. To construct the graph structure of the Cartesian product $\mathcal{G}^{\prime} \times \mathcal{G}$, let $(u_1, v_1)$ and $(u_2, v_2)$ be vertex pairs in $\mathcal{V}^{\prime} \times \mathcal{V}$. An edge exists between $(u_1, v_1)$ and $(u_2, v_2)$ if $v_1 = v_2$ (with edge weights from $\mathcal{E}^{\prime}$) or $u_1 = u_2$ (with edge weights from $\mathcal{E}$).

In this setup, $S(\mathcal{H}, \mathcal{G}) = L^2(\mathcal{G}^{\prime} \times \mathcal{G})$, meaning signals in $S(\mathcal{H}, \mathcal{G})$ are defined on the composite graph $\mathcal{G}^{\prime} \times \mathcal{G}$. The HGFRFT then represents the GFRFT of signals on this composite graph. Hence, $S(\mathcal{H}, \mathcal{G})$ can be viewed as a traditional graph signal on the merged graph $\mathcal{G}^{\prime} \times \mathcal{G}$. Details and sampling simulations of this example are presented in Section \ref{sec6.1}.

\begin{example}
	GFRFT in Infinite-Dimensional Hilbert Space \label{examp4}
\end{example} 
Consider continuous signals over the complex domain, where $\mathcal{H} = L^2(\mathbb{C})$ denotes the space of square-integrable functions on the complex plane \cite{Hilbert}. This space is infinite-dimensional, as each point $z \in \mathbb{C}$ can carry information.

Let $f(z)$ be a signal in $L^2(\mathbb{C})$. For example, $f(z) = \exp(-|z|^2)$ is a Gaussian function whose FT is also a Gaussian. To apply the HGFRFT, use basis functions $\psi^\alpha$ and $\phi^\beta$. The HGFRFT is defined as
\[
\hat{f}(z, \lambda) = \int_{-\infty}^{\infty} \int_{\mathcal{V}} f(z') \psi^\alpha(z - z') \phi^\beta(\lambda - \lambda') \, dz \, d\lambda',
\]
where $\psi^\alpha(z - z')$ is the $\mathcal{H}^{\alpha}_{f}$-transform basis function, $\phi^\beta(\lambda - \lambda')$ is the $\mathcal{G}^{\beta}_{f}$-transform basis function, $f(z')$ is the signal defined on the complex plane, and $\lambda$ is the frequency domain variable for the graph signal. This integral form combines the continuous and graph-based components of the signal processing framework.

By mapping the graph $\mathcal{G}$ to the complex plane, the GFRFT can be applied to each vertex of $\mathcal{G}$, allowing analysis of the signal's frequency representation within the graph structure. This approach integrates complex domain signal processing with graph signal processing, enhancing analytical capabilities. The specific examples are provided in Section \ref{sec6.3} and \ref{sec6.4}.

\section{Filtering}
\label{Filter}
Filters fulfill a variety of roles. They are instrumental in eliminating noise from signals, revealing the inherent connections between datasets, converting signals to domains that facilitate simpler analysis, and beyond. In this section, inspired by \cite{HGFT}, we delve into the application of filtering within the HGFRFT framework for generalized graphs. Our approach to filtering in the generalized fractional graph domain parallels the conventional principles of generalized GSP and signal processing over $\mathbb{C}$. However, it also uncovers the more complex structure of the general Hilbert space $\mathcal{H}$, introducing new dimensions and features to our theoretical understanding and practical applications.

The HGFRFT is linearly invertible, hence to verify that the continuous map $\mathbf{L}$ is a filter, we verify that for two graph signals $f, g \in S\left( \mathcal{H}, \mathcal{G}\right) $,
\[
	\mathbf{L}\left( a\mathcal{F}^{\alpha, \beta}_{f} + b\mathcal{F}^{\alpha, \beta}_{g}\right)  = a\mathbf{L}\left( \mathcal{F}^{\alpha, \beta}_{f}\right)  + b\mathbf{L}\left( \mathcal{F}^{\alpha, \beta}_{g}\right) ,
\]
where for all $a,b \in \mathbb{C}$, and since $S\left( \mathcal{G}, \mathcal{H}\right)$ is a Hilbert space, any filter $\mathbf{L}$ is continuous because it is bounded. From isomorphism, we equivalently consider any filter $\mathbf{L}$ on $S\left( \mathcal{H}, \mathcal{G}\right)$ to be a filter on $\mathcal{H} \otimes \mathbb{C}^{N}$.

\subsection{Shift Invariant Filters in the Fractional Domain}
The filters in the fractional domain are similar to the shift invariant filters in generalized graphs. The two operators $\mathbf{A}$ and $\mathbf{B}$ in Section \ref{3.1} have the Kronecker product $\mathbf{B} \otimes_{K} \mathbf{A}$, which is in the Hilbert space $S\left( \mathcal{H}, \mathcal{G}\right) \cong \mathcal{H} \otimes \mathbb{C}^{N}$. Since all operators on the finite dimensional space $\mathbb{C}^{N}$ are compact, $\mathbf{A}$ and $\mathbf{B}$ are both compact and self-adjoint, and $\mathbf{B} \otimes_{K} \mathbf{A}$ is also compact. The lemma for a shift invariant filter is as follows.

\begin{lem}
	Shift invariant and weakly shift invariant.
	\label{lem1}
	\begin{enumerate}
		\item A filter $\mathbf{L}$ is called shift invariant when $\mathbf{A} \circ \mathbf{L} = \mathbf{L} \circ \mathbf{A}$ and $\mathbf{B} \circ \mathbf{L} = \mathbf{L} \circ \mathbf{B}$ are both true for $\mathbf{A}$ and $\mathbf{B}$.
		\item A filter $\mathbf{L}$ is called weakly shift invariant when $\left( \mathbf{B} \otimes_{K} \mathbf{A}\right) \circ \mathbf{L} = \mathbf{L} \circ \left( \mathbf{B} \otimes_{K} \mathbf{A}\right) $ holds.
	\end{enumerate}
\end{lem}

In general, weakly shift invariant filters are not necessarily shift invariant. According to the above lemma, we discuss the relationship between weakly shift invariant filters and shift invariant filters in the fractional domain.

\begin{thm}
	\label{thm2}
	Suppose $\mathbf{L}$ is a filter on $S\left( \mathcal{H}, \mathcal{G}\right)$.
	\begin{enumerate}
		\item If $\Psi^{\alpha} \otimes \Phi^{\beta}$ consists of eigenvectors of $\mathbf{L}$, then $\mathbf{L}$ is shift invariant.
		\item  Suppose $\mathbf{L}$ is self-adjoint; then, $\mathbf{L}$ is weakly shift invariant if and only if $\mathbf{L}$ is shift invariant.
	\end{enumerate}
\end{thm}
\begin{proof} Given the premise that if $\mathbf{L}$ is shift invariant, then it is also weakly shift invariant.
	
\begin{enumerate}
\item Since $\Psi^{\alpha} \otimes \Phi^{\beta}$ is also a basis, its vectors are also eigenvectors of $\mathbf{B} \otimes_{K} \mathbf{I}$ and $\mathbf{I} \otimes_{K} \mathbf{A}$, and the shift invariance of $\mathbf{L}$ follows Lemma \ref{lem1}.
\item Suppose $\mathbf{L}$ is weakly shift invariant. The basis $\Psi^{\alpha} \otimes \Phi^{\beta}$ contains the eigenvector of $\mathbf{L}$. By 1), $\mathbf{L}$ is shift invariant, and the necessary information is proved. 
\end{enumerate}
\end{proof}

\subsection{Fractional Order Convolution Filters}
Let a $g \in S\left(\mathcal{H}, \mathcal{G}\right)$. For each $f \in S\left(\mathcal{H}, \mathcal{G}\right)$, their convolution $g \ast f$ is the multiplication of the transform domain and is defined as
\begin{equation}
	\begin{aligned}\mathcal{F}^{\alpha ,\beta }_{g\ast f} =&\mathcal{F}^{\alpha ,\beta }_{g} \cdot \mathcal{F}^{\alpha ,\beta }_{f} ,\\ g\ast f=&\sum_{\psi^{\alpha} \otimes \phi^{\beta}} \mathcal{F}^{\alpha ,\beta }_{g} \mathcal{F}^{\alpha ,\beta }_{f} \cdot \psi^{\alpha}\otimes \phi^{\beta} ,\end{aligned}
\end{equation}
where $\mathcal{F}^{\alpha ,\beta}_{g\ast f}$ is an element of $S\left(\mathcal{H}, \mathcal{G}\right)$ and $g \in S\left(\mathcal{H}, \mathcal{G}\right)$, thus $\sum_{\psi^{\alpha} \otimes \phi^{\beta}} \left| \mathcal{F}^{\alpha ,\beta}_{g} \right|^{2}  <\infty $.

For a convolution filter $g\ast$, $g\ast: S\left( \mathcal{H} ,\mathcal{G} \right)  \rightarrow S\left( \mathcal{H} ,\mathcal{G} \right)$ is a bounded map bounded by $\sup_{\psi^{\alpha} \otimes \phi^{\beta}} \left| \mathcal{F}^{\alpha ,\beta}_{g} \right|  <\infty $. It is easy to prove that this satisfies
\begin{equation}
	g\ast (cf+h)=cg\ast f+g\ast h,\  \text{for} \  c\in \mathbb{C},\  h\in S\left( \mathcal{H} ,\mathcal{G} \right).
\end{equation}
When $\mathcal{H} = \mathbb{C}$, the concept of convolution filters is consistent with traditional GSP.

When $f =\psi^{\alpha}\otimes \phi^{\beta}$, by definition $g \ast f = \mathcal{F}^{\alpha ,\beta}_{g}\cdot \psi^{\alpha}\otimes \phi^{\beta}$. Thus $\psi^{\alpha}\otimes \phi^{\beta}$ is the eigenvector of $g\ast$ with eigenvalue $\mathcal{F}^{\alpha ,\beta}_{g}$. By Theorem \ref{thm2}. 1), $g\ast$ is a shift invariant operator. Furthermore, since $\sum_{\psi^{\alpha}\otimes \phi^{\beta}} \left| \mathcal{F}^{\alpha ,\beta}_{g} \right|^{2} <\infty$, $g\ast$ is also the Hilbert‒Schmidt operator \cite{Hilbert}. It is also compact and the limit of finite rank filters.

If $\mathcal{H}$ is infinite dimensional, it is noncompact and thus not a convolution filter. This is different from traditional GSP, for which when $\mathbf{A}$ has no repeated eigenvalues, all shift-invariant filters are convolution filters, because in traditional GSP, the shift invariant filter is a polynomial of $\mathbf{A}$ \cite{GFTadjacency1} and $\mathcal{H} = \mathbb{C}$ is finite dimensional.

\subsection{$\alpha, \beta$-Bandlimited Signals and Bandpass Filters}
For $\psi^{\alpha}$ and $\phi^{\beta}$, we use $\mathbf{\Lambda_{B}}$ and $\mathbf{\Lambda_{A}}$ to denote their corresponding eigenvalues. We can obtain the following definition.
\begin{defn}
	\label{defn2}
	For $f \in S\left( \mathcal{H}, \mathcal{G}\right) $, the frequency range of $f$ is defined as
	\[
\left\{ (\mathbf{\Lambda}_{\mathbf{B}}^{-1} ,\mathbf{\Lambda_{A} } )\in \mathbb{R} \times \mathbb{R} :\mathcal{F}^{\alpha ,\beta }_{f} \neq 0\right\}  ,
	\]
where $\mathbf{\Lambda_{B} }$ and $\mathbf{\Lambda_{A} }$ are diagonal matrices composed of the eigenvalues of matrices $\mathbf{B}$ and $\mathbf{A}$, respectively.
	
\end{defn}
We use $\mathbf{\Lambda}_{\mathbf{B}}^{-1}$ in Definition \ref{defn2}, which is more convenient when handling bandlimitedness. %in Section IV-D. 

A signal $f \in S\left( \mathcal{H} ,\mathcal{G} \right)$ is said to be $\alpha, \beta$-bandlimited if its frequency range is a bounded subset of $\mathbb{R}\times \mathbb{R}$. $S_{K}\left( \mathcal{H} ,\mathcal{G} \right)$ represents a set of signals whose frequency range is within $K$, where $K \subset \mathbb{R} \times \mathbb{ R}$. If $\mathcal{G}$ is a point and $\mathcal{H} = L^{2}([a, b])$, the notion of bandwidth limitation is consistent with its classic counterpart in the fractional Fourier series setting.

\begin{lem}
	\label{lem2}
	For a bandlimited signal $f \in S\left( \mathcal{H} ,\mathcal{G} \right)$, and $K \subset \mathbb{R} \times \mathbb{R}$,
	\begin{enumerate}
		\item $f$ is bandlimited if and only if its frequency range is a finite set. 
		\item If $K$ is bounded, then $S_{K}\left( \mathcal{H} ,\mathcal{G} \right)$ is a finite-dimensional subspace of $S\left( \mathcal{H} ,\mathcal{G} \right)$.
	\end{enumerate}
\end{lem}

For each $f \in S\left( \mathcal{H} ,\mathcal{G} \right)$ and $K \subset \mathbb{R} \times \mathbb{R}$, we define an $\alpha, \beta$-bandpass filter as the projection
\begin{equation}
	P_{K}\left( f\right)  =\sum_{\left(\mathbf{\Lambda}_{\mathbf{B}}^{-1},\mathbf{\Lambda_{A}} \right)  \in K} \mathcal{F}^{\alpha,\beta}_{f}\cdot \psi^{\alpha}\otimes \phi^{\beta}.
\end{equation}

We can also write this projection as follows
\begin{equation}
P_{K}\left( f\right)  = \left< \left< f, \psi^{\alpha} \otimes \phi^{\beta} \right> \cdot \delta_{K}, (\psi^{-\alpha} \otimes \phi^{-\beta}) \right>,
\end{equation}
where \(\delta_{K}\) is the frequency selection vector used to select the joint frequency component. If the frequency domain of the signal falls within the $\alpha,\beta$-bandpass, $\delta_{K}=1$; otherwise, $0$.

In the HGFRFT, $P_{K}$ is simply multiplied by the eigenfunction of $K$, and if $K$ is unbounded, then the bandpass filter is not a convolution. Therefore, similar to the properties of $P_{K}$ in generalized graphs, the following theorem exists.

\begin{thm}
	$P_{K}$ is shift invariant; it is a convolution filter if and only if $K$ is bounded for any $K \subset \mathbb{R} \times \mathbb{R}$.
\end{thm}
\begin{proof}
$\psi^{\alpha}\otimes \phi^{\beta}$ is the eigenvector of $P_{K}$, it is shift invariant by Theorem \ref{thm2}. 1). Furthermore, it is a convolution filter if and only if the eigenfunction of $K$ is in the discrete set $\left\{ \left(\mathbf{\Lambda}_{\mathbf{B}}^{-1}, \mathbf{\Lambda_{A} } \right)  \right\}  $. In other words, $K$ is finite by Lemma \ref{lem2}. 1), and $K$ is finite if and only if it is bounded.
\end{proof}

Building on the discussions in this section, we have introduced several practical scenarios involving finite-dimensional subspaces of $S(\mathcal{H}, \mathcal{G})$ to set the stage for the sampling discourse in the subsequent section. This forthcoming discussion will elaborate on employing a set of points within $\Omega \times \mathcal{V}$ to characterize these subspaces.

\section{Sampling}
In this section, we assume $\Omega$ to be a compact subset of $\mathbb{C}^{N}$ whose interior is nonempty and connected, and $\mathcal{H} = L^{2}\left( \Omega \right)$ is equipped with the usual Lebesgue measure \cite{Functional}. Furthermore, we assume that the eigenvector $\Psi^{\alpha}$ is piecewise smooth on $\Omega$, and each $f \in S\left( \mathcal{H} ,\mathcal {G} \right) \cong \mathcal{H}\otimes \mathbb{C}^{N}$ can be viewed as a function on  $\Omega \times \mathcal{V}$.

\subsection{Graph Fractional Sampling in Hilbert Space}
In this subsection, we propose a theory of graph fractional sampling in Hilbert space based on the GFRFT sampling \cite{GFRFTsampling}. Joint sampling in the Hilbert space and vertex domains can be challenging, especially in scenarios such as sensor networks, social networks, and radar networks. Different graph samples may be required at different time vertices.

Due to the inherent complexity of sampling in infinite-dimensional Hilbert spaces, it is often necessary to discretize or approximate the Hilbert space appropriately. A common approach is to represent the infinite-dimensional Hilbert space in a finite-dimensional form using an orthonormal basis. Subsequently, the sampling operator is applied to the finite-dimensional representation of the Hilbert space signal.

Suppose $f \in \text{span}(\overline{\Psi}^{\alpha} \otimes \overline{\Phi}^{\beta})$ is bandlimited with $\overline{\Psi}^{\alpha}$ and $\overline{\Phi}^{\beta}$ being finite subsets of $\Psi^{\alpha}$ and $\Phi^{\beta}$ respectively. Select a sampling set $\mathcal{W} \subset \Omega \times \mathcal{V}$ such that $f$ is uniquely determined by $\mathcal{W}$. Define the sampling graph signal $\tilde{f}(\mathcal{W}) \in \mathbb{C}^{|\mathcal{W}|}$, with $\left| \mathcal{W}\right|  \geq |  \overline{\Psi}^{\alpha  }|  \cdot | \overline{\Phi}^{\beta  } |$, such that $\tilde{f} = \mathbf{D}f$. The sampling matrix is defined as
\begin{equation}
	\mathrm{\mathbf{D}}_{i,j}  = \left\{ \begin{array}{rl}
		1, & j=\mathcal{W}_i, \\
		0, & \mathrm{otherwise}.
	\end{array} \right.
\end{equation}

We then recover $f$ from $\tilde{f}$ using the recovery operator $\mathbf{R}$, which is a linear mapping from $\mathbb{C}^{|\mathcal{W}|}$ to $\mathbb{C}^{|\Psi^{\alpha}| \cdot |\Phi^{\beta}|}$. Prior to this, we analyzed the bandwidth, establishing connections between $\mathcal{H}$, $\mathcal{G}$, and $S(\mathcal{H}, \mathcal{G})$. According to the projected bandwidth in Definition \ref{defn2} and Lemma \ref{lem2}, a signal $f$ is termed a synchronous bandlimited signal if the projected bandwidths satisfy $|\overline{\Psi}^{\alpha}| < |\Psi^{\alpha}|$ and $|\overline{\Phi}^{\beta}| < |\Phi^{\beta}|$. Clearly, the relationship between the projected and general bandwidths is given by $\max (|\overline{\Psi}^{\alpha}|, |\overline{\Phi}^{\beta}|) \leq K \leq |\overline{\Psi}^{\alpha}| \cdot |\overline{\Phi}^{\beta}|$. The condition for perfect reconstruction of $f $ from $\tilde{f}$ is detailed in Theorem \ref{thm3}.
\begin{thm}\label{thm3}
	For all $\alpha, \beta$-bandlimited graph signals $f$ with bandwidth $K$, where $K = |\overline{\Psi}^{\alpha}| \cdot |\overline{\Phi}^{\beta}|$ and $\mathrm{rank}\{\mathbf{D}(\overline{\Psi}^{\alpha} \otimes \overline{\Phi}^{\beta})\} = K$ (with $K < |\Psi^{\alpha}| \cdot |\Phi^{\beta}|$), perfect recovery of $f$ is achievable. Specifically, if $f \in \text{span}(\overline{\Psi}^{\alpha} \otimes \overline{\Phi}^{\beta})$, then $f$ can be perfectly recovered as $f = \mathbf{R} \tilde{f}$ by choosing the reconstruction operator $\mathbf{R} = (\overline{\Psi}^{\alpha} \otimes \overline{\Phi}^{\beta}) \{\mathbf{D}(\overline{\Psi}^{\alpha} \otimes \overline{\Phi}^{\beta})\}^{\dag}$. Here, $\overline{\Psi}^{\alpha} \otimes \overline{\Phi}^{\beta}$ represents a finite-dimensional approximation of $\Psi^{\alpha} \otimes \Phi^{\beta}$.
\end{thm}
\begin{proof}
	Let $\mathbf{P} = \{\mathbf{D}(\overline{\Psi}^{\alpha} \otimes \overline{\Phi}^{\beta})\}^{\dag}$. Given that $\mathrm{rank}\{\mathbf{D}(\overline{\Psi}^{\alpha} \otimes \overline{\Phi}^{\beta})\} = K$ and $\mathrm{rank}\{\mathbf{PD}(\overline{\Psi}^{\alpha} \otimes \overline{\Phi}^{\beta})\} = K$, it follows that $\mathrm{rank}(\mathbf{P}) = K$. Therefore, $\mathbf{R} \in \text{span}(\overline{\Psi}^{\alpha} \otimes \overline{\Phi}^{\beta})$. Moreover, since $\mathbf{RDRD} = (\overline{\Psi}^{\alpha} \otimes \overline{\Phi}^{\beta}) \mathbf{PD}(\overline{\Psi}^{\alpha} \otimes \overline{\Phi}^{\beta}) \mathbf{PD} = (\overline{\Psi}^{\alpha} \otimes \overline{\Phi}^{\beta}) \mathbf{PD} = \mathbf{RD}$, $\mathbf{RD}$ is a projection operator. Consequently, $\mathbf{R} \tilde{f}$ serves as an approximation of $f \in \text{span}(\overline{\Psi}^{\alpha} \otimes \overline{\Phi}^{\beta})$. When $f$ is in $\text{span}(\overline{\Psi}^{\alpha} \otimes \overline{\Phi}^{\beta})$, $\mathbf{R} \tilde{f} = f$, thereby achieving perfect recovery.
\end{proof}

When $|\mathcal{W}| < K$, it follows that $\mathrm{rank}\{\mathbf{PD}(\overline{\Psi}^{\alpha} \otimes \overline{\Phi}^{\beta})\} \leq \mathrm{rank}(\mathbf{P}) \leq |\mathcal{W}| < K$. Thus, $\mathbf{PD}(\overline{\Psi}^{\alpha} \otimes \overline{\Phi}^{\beta})$ cannot be an identity matrix, making perfect recovery of the original signal impossible. If $|\mathcal{S}| = |\mathcal{F}|$, for $\mathbf{PD}(\overline{\Psi}^{\alpha} \otimes \overline{\Phi}^{\beta})$ to be the identity matrix, $\mathbf{P}$ must be the inverse of $\mathbf{D}(\overline{\Psi}^{\alpha} \otimes \overline{\Phi}^{\beta})$. When $|\mathcal{W}| > K$, $\mathbf{P}$ acts as the pseudo-inverse of $\mathbf{D}(\overline{\Psi}^{\alpha} \otimes \overline{\Phi}^{\beta})$, and perfect recovery is achievable if $|\mathcal{W}| \geq K$. For simplicity, we focus on the case where $|\mathcal{W}| = K$.

\subsection{Optimal Sampling in the Hilbert Space}
Because we aim to choose the best set that minimizes the influence of noise, and there are multiple choices of $K = |\overline{\Psi}^{\alpha}| \cdot |\overline{\Phi}^{\beta}|$ linearly independent rows in $\overline{\Psi}^{\alpha} \otimes \overline{\Phi}^{\beta}$, the noise $n$ introduced by sampling is considered
\[
\tilde{f}=\mathbf{D}f+n.
\]
Therefore, the recovered signal $f_{\text{rec}}$ will be
\[
f_{\text{rec}} =\mathbf{R} \tilde{f}=\mathbf{RD}f+\mathbf{R} n
=f+\mathbf{R} n.
\]
By incorporating the original signal $f$, the bound of the recovery error $\epsilon$ is
\[
\begin{array}{rl}
	||f_{\text{rec}}-f||_2 &= ||\mathbf{R}n||_2 = ||(\overline{\Psi}^{\alpha} \otimes \overline{\Phi}^{\beta}) \{\mathbf{D}(\overline{\Psi}^{\alpha} \otimes \overline{\Phi}^{\beta})\}^{\dag}n||_2 \\
	&\leq ||\overline{\Psi}^{\alpha} \otimes \overline{\Phi}^{\beta}||_2 ||\{\mathbf{D}(\overline{\Psi}^{\alpha} \otimes \overline{\Phi}^{\beta})\}^{\dag}||_2||n||_2.
\end{array}
\]

Since $||\overline{\Psi}^{\alpha} \otimes \overline{\Phi}^{\beta}||_2$ and $||n||_2$ are fixed, we aim to minimize $||\{\mathbf{D}(\overline{\Psi}^{\alpha} \otimes \overline{\Phi}^{\beta})\}^{\dag}||_2$. For each qualified $\mathbf{D}$, we maximize the $\mathbf{D}(\overline{\Psi}^{\alpha} \otimes \overline{\Phi}^{\beta})$ minimum singular value $\sigma_{\min}$
\begin{equation}
	\mathbf{D}^{\text{opt}}=\arg\max_\mathbf{D}~\sigma_\mathrm{min}\{\mathbf{D}(\overline{\Psi}^{\alpha} \otimes \overline{\Phi}^{\beta})\}.\label{Psiopt}
\end{equation}

Because Eq. \eqref{Psiopt} achieves the maximum robustness to noise, it is called the optimal sampling operator. The optimization problem can be solved using the greedy algorithm \cite{GFTsampling,GFRFTsampling,GLCTsampling}, as shown in Algorithm \ref{alg1}.

% Alogrithm 1
\begin{algorithm}[H]
		\footnotesize
	\caption{\footnotesize{Optimal Sampling Operator via Greedy Algorithm}}\label{alg1}
	\begin{algorithmic}
		\STATE 
	\STATE \textbf{Input}
		\STATE \hspace{1cm}$\overline{\Psi}^{\alpha} \otimes \overline{\Phi}^{\beta}$: the finite-dimensional approximation of $\Psi^{\alpha} \otimes \Phi^{\beta}$
		\STATE \hspace{1cm}$|\mathcal{W}|$: the number of samples 
		\STATE \hspace{1cm}$|\Psi^{\alpha}| \cdot |\Phi^{\beta}|$: the number of rows in $\overline{\Psi}^{\alpha} \otimes \overline{\Phi}^{\beta}$
	\STATE \textbf{Output}
		\STATE \hspace{1cm}$\mathcal{W}$ sampling set 
	\STATE	\textbf{Function} 
		\STATE 
		\STATE \hspace{1cm}$\begin{array}{l}
			\textbf{for~}i = 1:|\mathcal{W}| \\
			\quad w=\arg\max_{j\in \{1:|\Psi^{\alpha}| \cdot |\Phi^{\beta}|\}-\mathcal{W}} \sigma_\mathrm{min} \left((\overline{\Psi}^{\alpha} \otimes \overline{\Phi}^{\beta})_{\mathcal{W}+\{j\}}\right) \\
			\quad \mathcal{W} = \mathcal{W} +\{w\} \\
			\textbf{end} \\
			\textbf{return}~ \mathcal{W}
		\end{array}$
	\end{algorithmic}
\end{algorithm}

Once the optimal sampling operator is determined, signal reconstruction can proceed. For all sampling experiments, the parameters $(\alpha, \beta)$ are treated as hyperparameters and selected via grid search \cite{JFRFT, JFRFTwiener}. Since HGFRFT is additive and continuous with respect to the transformation orders, we initially perform the search with larger step sizes to cover a broader range, followed by a finer search around the optimal values. The detailed procedure is outlined in Algorithm \ref{alg2} and applied in Section \ref{sec6.3.2}.

% Alogrithm 2
\begin{algorithm}[H]
		\footnotesize
	\caption{\footnotesize{Grid Search for Optimal $(\alpha, \beta)$ Parameters}}\label{alg2}
	\begin{algorithmic}
		\STATE \textbf{Input}
		\STATE \hspace{1cm} \(\alpha_{\text{range}}, \beta_{\text{range}}\): initial coarse ranges for \(\alpha\) and \(\beta\)
		\STATE \hspace{1cm} \(n\): noise vector, \(\tilde{f}\): noisy signal
		\STATE \textbf{Output}
		\STATE \hspace{1cm} \(\alpha^*, \beta^*\): optimal parameters
		\STATE \textbf{Initialize} \(\epsilon^* \leftarrow \infty\), \(\alpha^*, \beta^* \leftarrow 0, 0\)
		\STATE \textbf{for} \(\alpha \in \alpha_{\text{range}}\) \textbf{do}
		\STATE \hspace{2cm} \textbf{for} \(\beta \in \beta_{\text{range}}\) \textbf{do}
		\STATE \hspace{3cm}  \textbf{Compute error} \(\epsilon(\alpha, \beta) = \|\mathbf{R} n\|_2\)
		\STATE \hspace{3cm}  \textbf{if} \(\epsilon(\alpha, \beta) < \epsilon^*\) \textbf{then}
		\STATE \hspace{4cm} \(\epsilon^* \leftarrow \epsilon(\alpha, \beta)\)
		\STATE \hspace{4cm} \(\alpha^*, \beta^* \leftarrow \alpha, \beta\)
		\STATE \hspace{3cm} \textbf{end if}
		\STATE \hspace{2cm} \textbf{end for}
		\STATE \hspace{1cm} \textbf{end for}
		\STATE \textbf{Refine search around} \((\alpha^*, \beta^*)\) with smaller step sizes
		\STATE \textbf{Repeat the grid search for refined range}
		\STATE \textbf{Return} \(\alpha^*, \beta^*\)
	\end{algorithmic}
\end{algorithm}

\section{Numerical Results and Applications}
The HGFRFT, akin to the HGFT and the time-vertex Fourier transform, has shown its versatility across a spectrum of problem classes, as evidenced by various datasets. In this section, we leverage the HGFRFT for conducting simulation experiments on a product graph and to model the susceptible-exposed-infected-recovereds (SEIRS) epidemic in Europe \cite{JFT}. Additionally, we employ fractional domain signal processing techniques in Hilbert space to handle radar signals and digital images, comparing these results with those obtained through HGFT. The findings underscore the effectiveness of HGFRFT in signal denoising, sample recovery, and localization tasks. All experiments\footnote{Code Available in: \url{https://github.com/Zhangyubit/HGFRFT}} presented in this paper were conducted with the assistance of the GSP toolbox \cite{GSPBOX}.

\subsection{Simulation on Product Graph Sampling} \label{sec6.1}
In this section, we conduct simulations for both a 4-node ring graph within $\mathbb{C}^{N}$ and a 4-node path graph in $\mathcal{H}$. We choose $\alpha=0.7$ and $\beta=0.5$. For our experiments, we define a graph signal with a bandwidth of $2$ in both spaces. Here, we introduce a graph chirp signal, analogous to the time-frequency analysis signal $x(t) = e^{i(lt + kt^2)}$, where $l$ and $k$ are constants. In the Hilbert space-vertex domain, the signal $f$ is formulated as 
\begin{equation}
f = \mathbf{f}_1 + 0.5\mathbf{f}_2 + 2\mathbf{f}_3, \label{signal}
\end{equation}
where $\mathbf{f}_i$ denotes the $i$th column in the operator matrix $\Psi^{\alpha} \otimes \Phi^{\beta}$. Fig. \ref{fig2} presents the configurations of the original graphs. 
\begin{figure*}[ht]
	\begin{center}
		\begin{minipage}[t]{0.3\linewidth}
			\centering
			\includegraphics[width=\linewidth]{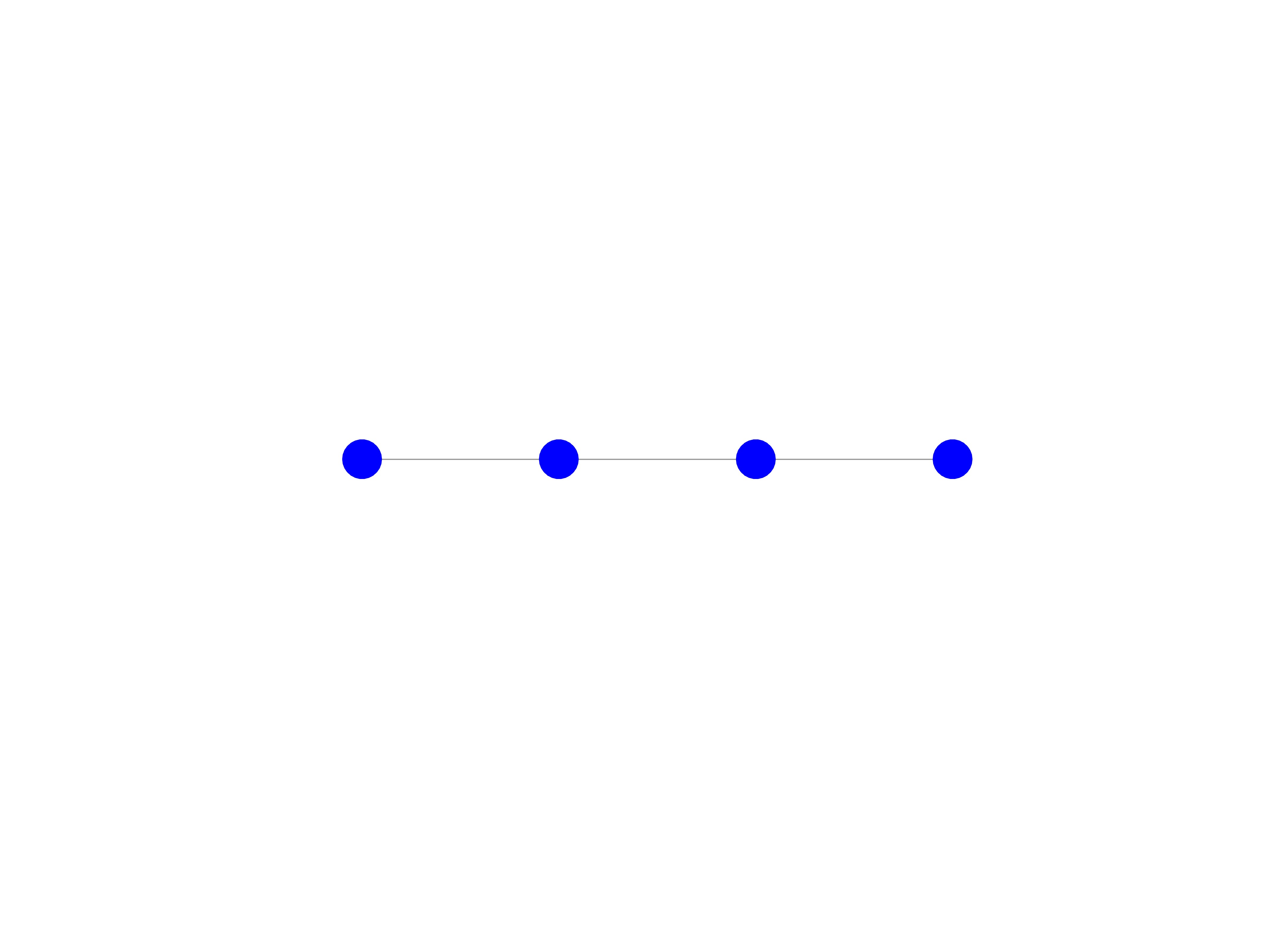}
			\parbox{1.5cm}{\tiny (a) Path graph.}
		\end{minipage}
		\begin{minipage}[t]{0.3\linewidth}
			\centering
			\includegraphics[width=\linewidth]{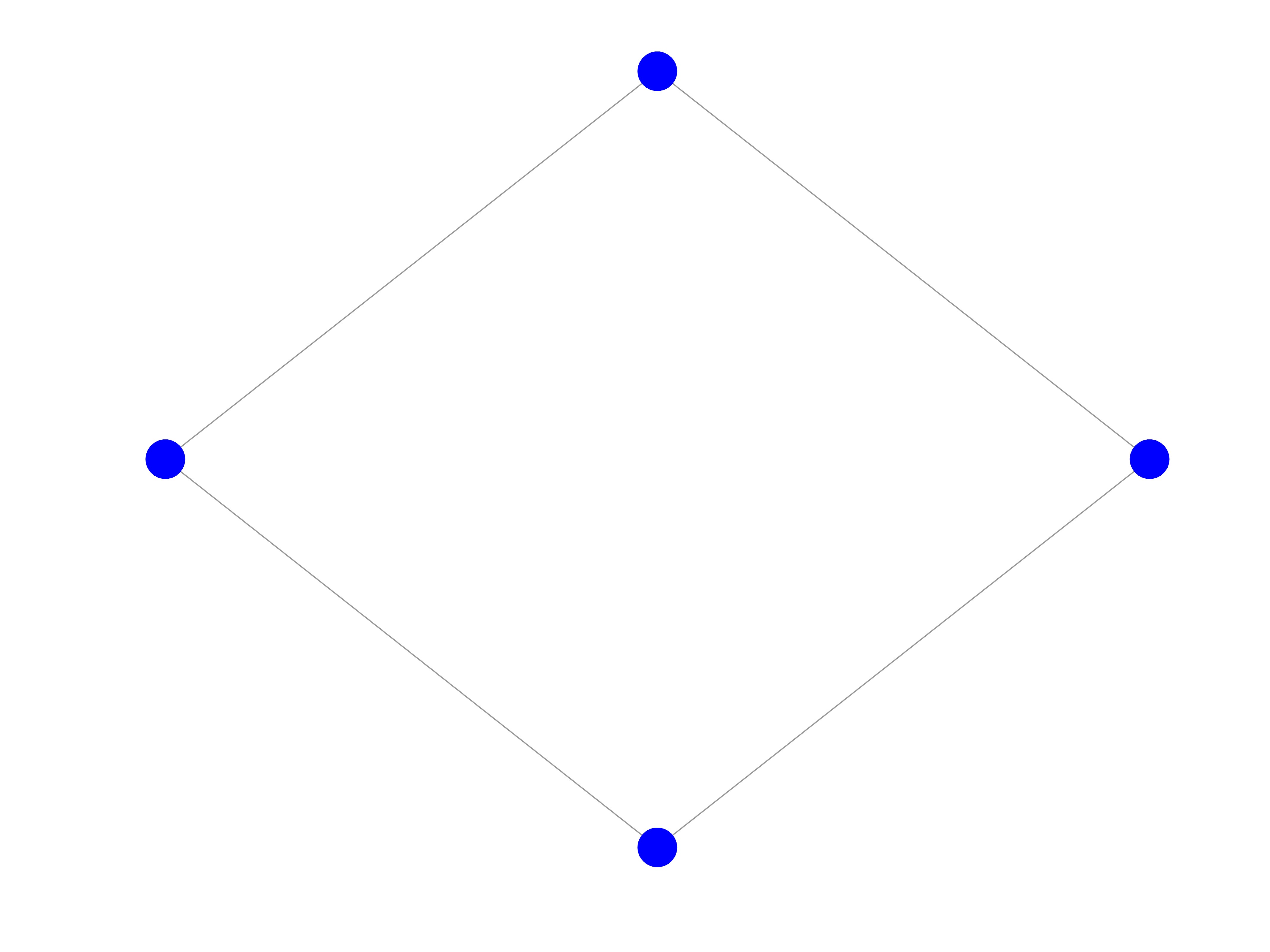}
			\parbox{1.5cm}{\tiny (b) Ring graph.}
		\end{minipage}
		\begin{minipage}[t]{0.3\linewidth}
			\centering
			\includegraphics[width=\linewidth]{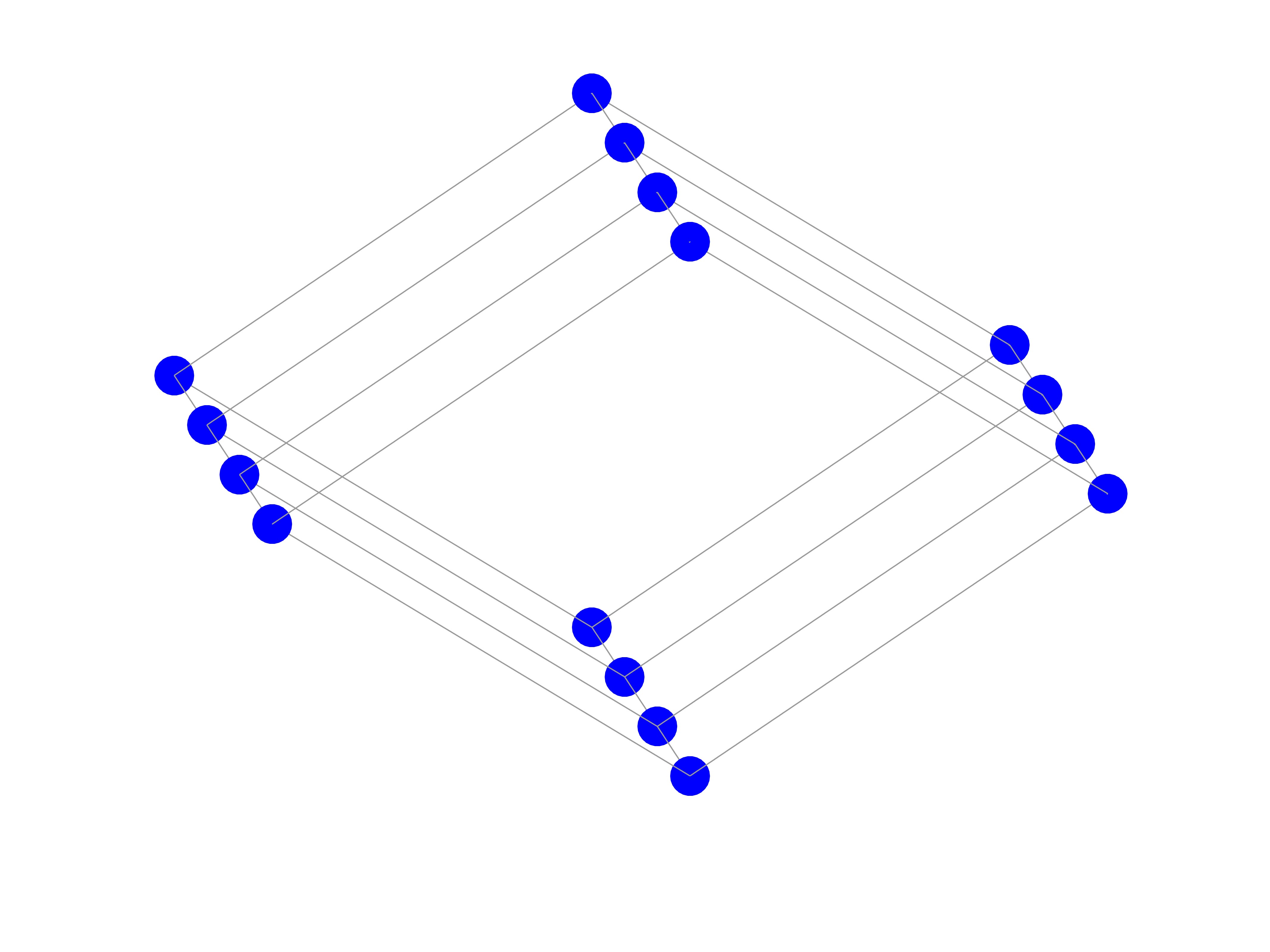}
			\parbox{1.5cm}{\tiny (c) Product graph.}
		\end{minipage}
	\end{center}
	\caption{Simulation graph: 4-node path graph, 4-node ring graph and 16-node product graph.}
	\vspace*{-3pt}
	\label{fig2}
\end{figure*}

Although signals on these graphs can exhibit complex structures, previous considerations have not addressed complex-valued signals, such as linearly chirped signals, allowing us to examine the sampling effects in both frameworks. Fig. \ref{fig3} illustrates graph signals in the fractional domain within conventional GSP. Panels (a) and (b) represent the decomposition of the signal as described in Eq. \eqref{signal}, corresponding to the operators $\Psi^{\alpha}$ and $\Phi^{\beta}$, respectively.
\begin{figure}[h!]
	\begin{center}
		\begin{minipage}[t]{0.45\linewidth}
			\centering
			\includegraphics[width=\linewidth]{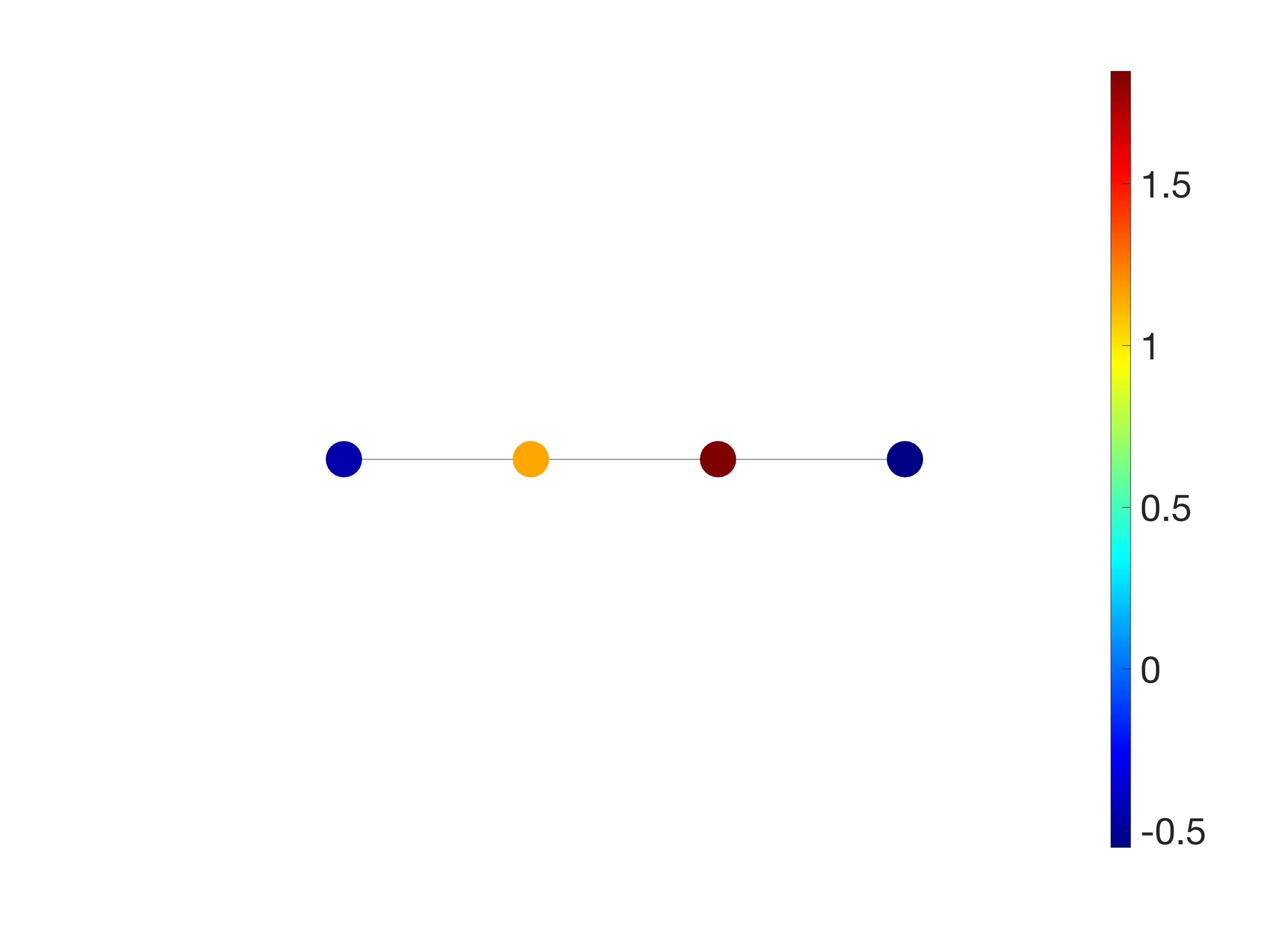}
			\parbox{3cm}{\tiny (a) The path graph of the signal is $\Psi^{\alpha}_1 + 0.5\Psi^{\alpha}_2 + 2\Psi^{\alpha}_3$.}
		\end{minipage}
		\begin{minipage}[t]{0.45\linewidth}
			\centering
			\includegraphics[width=\linewidth]{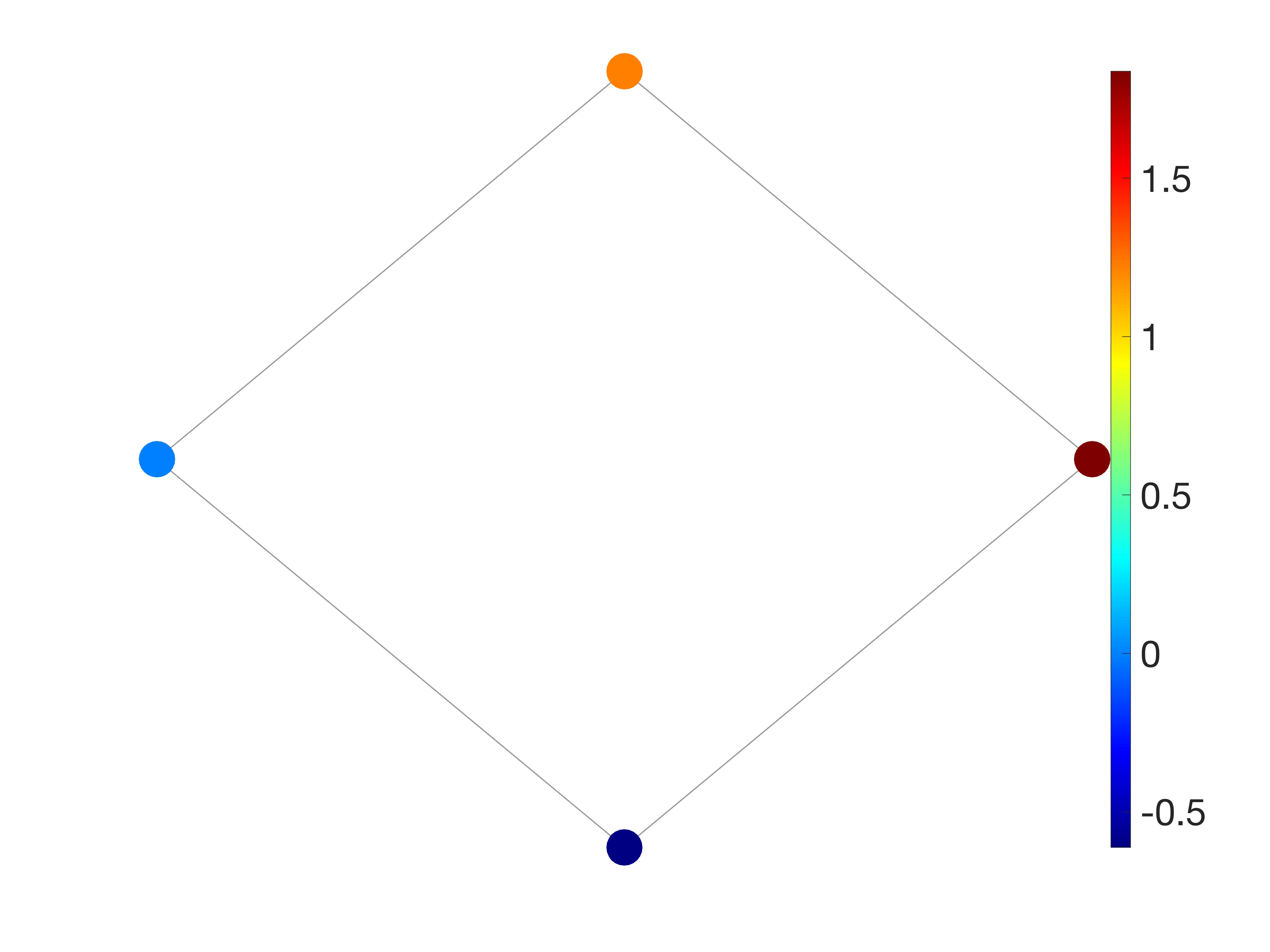}
			\parbox{3cm}{\tiny (b) The ring graph of the signal is $\Phi^{\beta}_1 + 0.5\Phi^{\beta}_2 + 2\Phi^{\beta}_3$.}
		\end{minipage}
	\end{center}
	\caption{The fractional graph signal in traditional GSP.}
	\vspace*{-3pt}
	\label{fig3}
\end{figure}

We employ two optimal sampling operators based on $\alpha=0.7$ and $\beta=0.5$, namely the HGFT and HGFRFT, to sample and recover the graph signal. The recovery accuracy is quantified using the error metric
\begin{equation}
	\epsilon = ||f_{\text{rec}} - f||_2. \label{error}
\end{equation} 
Fig. \ref{fig4} depicts the graph signals on the product graph, while Fig. \ref{fig5}(a) and Fig. \ref{fig5}(b) illustrate the sampling nodes using HGFT and HGFRFT, respectively. Signal recovery results are presented in Fig. \ref{fig6}(a), where HGFT sampling yields a recovery error of $\epsilon = 0.8361$, and Fig. \ref{fig6}(b), where HGFRFT achieves near-perfect recovery with $\epsilon = 5.8055 \times 10^{-17}$. The inability of HGFT to perfectly reconstruct the graph signal is attributed to the signal's generalized GFT domain bandwidth, which is not strictly bounded, despite being limited to $\alpha = 0.7$ and $\beta = 0.5$. Fig. \ref{fig7}(a) demonstrates that when varying both $\alpha$ and $\beta$, sampling with $\alpha=0.7$ leads to noticeably smaller recovery errors. Similarly, Fig. \ref{fig7}(b) reveals that once the optimal $\alpha = 0.7$ is determined, the recovery error is minimized when $\beta = 0.5$.
\begin{figure}[h!]
	\begin{center}
		\begin{minipage}[t]{0.45\linewidth}
			\centering
			\includegraphics[width=\linewidth]{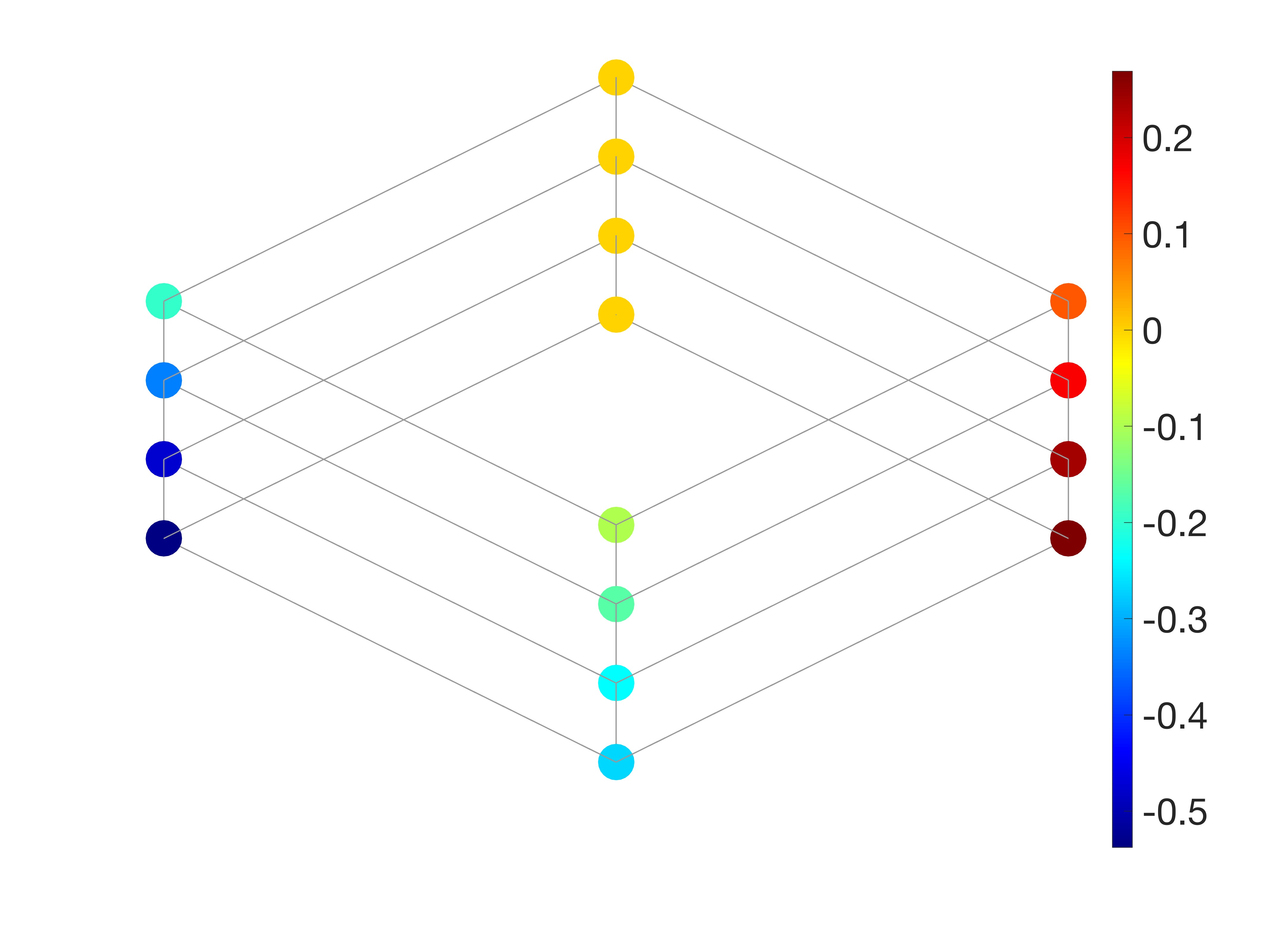}
			\parbox{2.5cm}{\tiny (a) The 1st basis vector $\mathbf{f}_1$.}
		\end{minipage}
		\begin{minipage}[t]{0.45\linewidth}
			\centering
			\includegraphics[width=\linewidth]{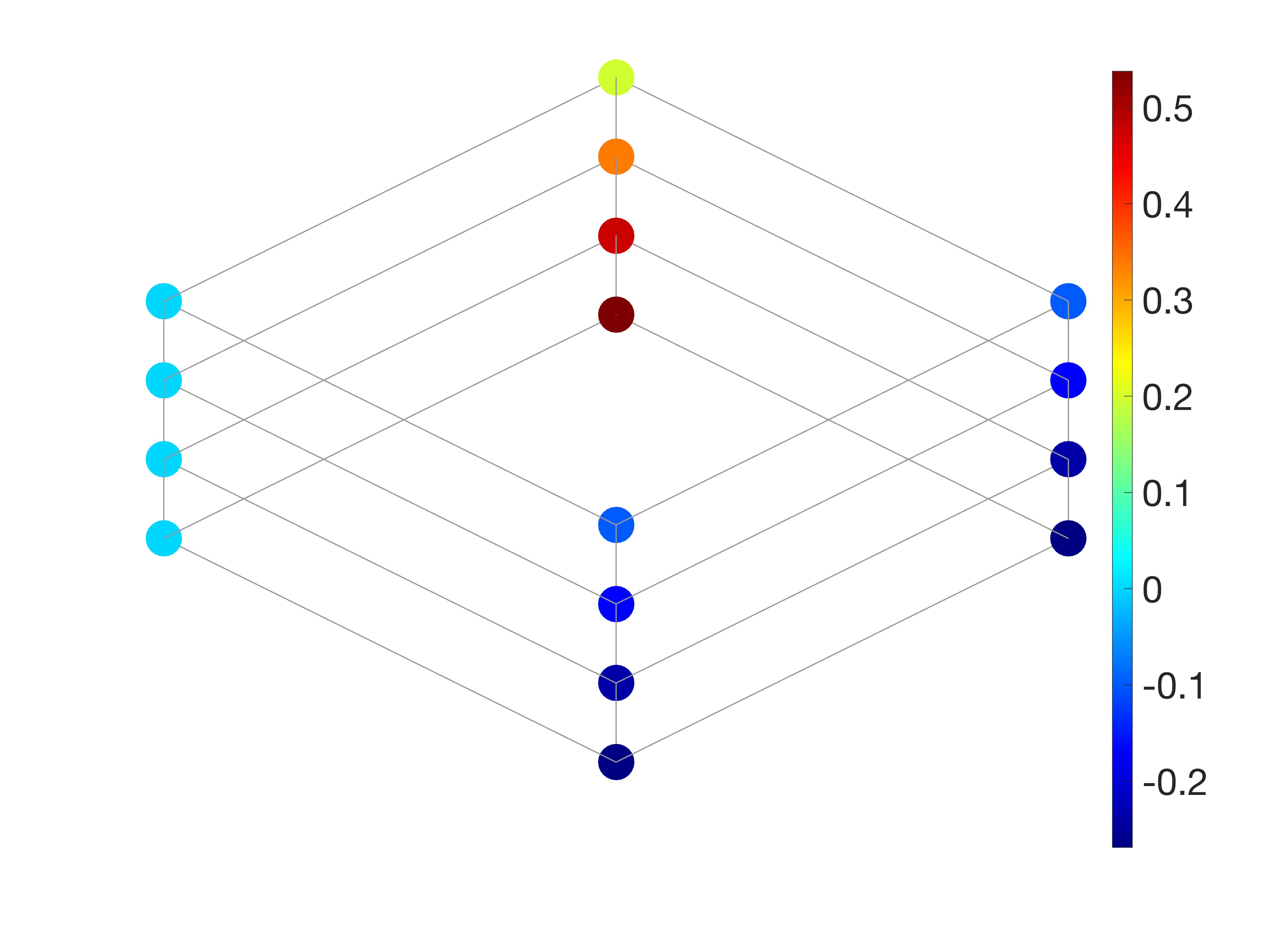}
			\parbox{2.5cm}{\tiny (b) The 2nd basis vector $\mathbf{f}_2$.}
		\end{minipage}
		\begin{minipage}[t]{0.45\linewidth}
			\centering
			\includegraphics[width=\linewidth]{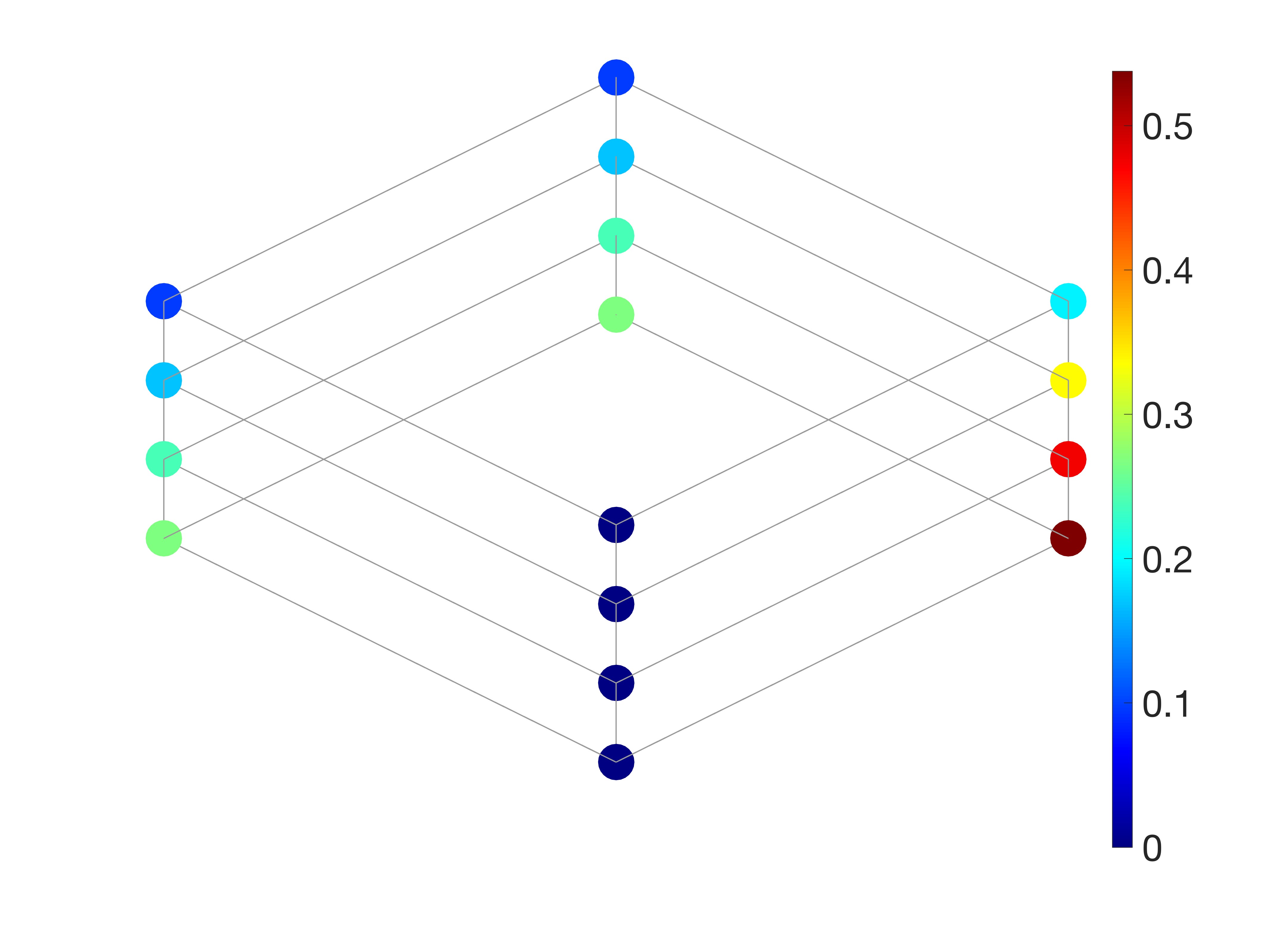}
			\parbox{2.5cm}{\tiny (c) The 3rd basis vector $\mathbf{f}_3$.}
		\end{minipage}
		\begin{minipage}[t]{0.45\linewidth}
			\centering
			\includegraphics[width=\linewidth]{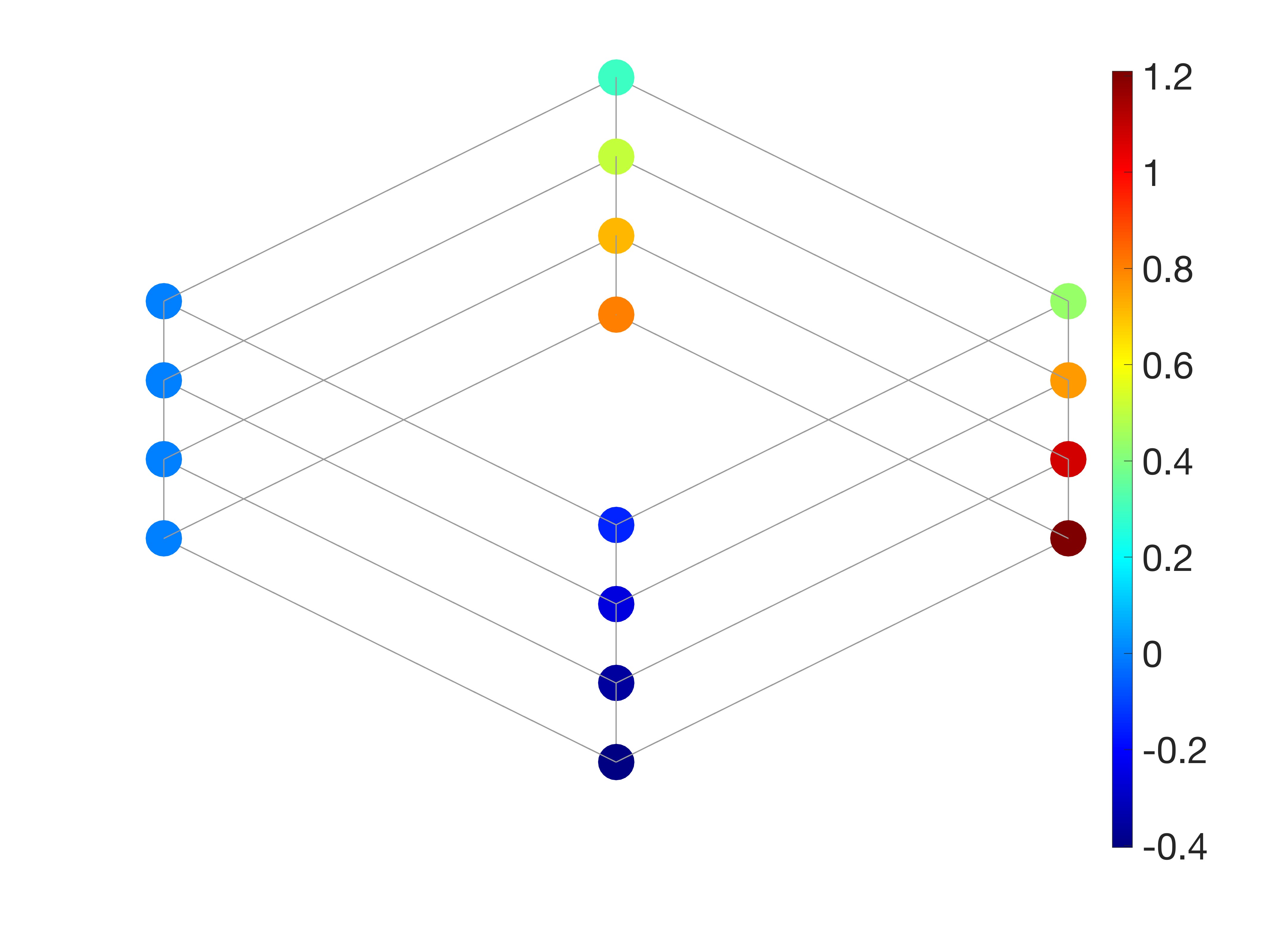}
			\parbox{2.4cm}{\tiny (d) Original graph signal.}
		\end{minipage}
	\end{center}
	\caption{Graph signals on the product graph.}
	\vspace*{-3pt}
	\label{fig4}
\end{figure}
\begin{figure}[h!]
	\begin{center}
		\begin{minipage}[t]{0.45\linewidth}
			\centering
			\includegraphics[width=\linewidth]{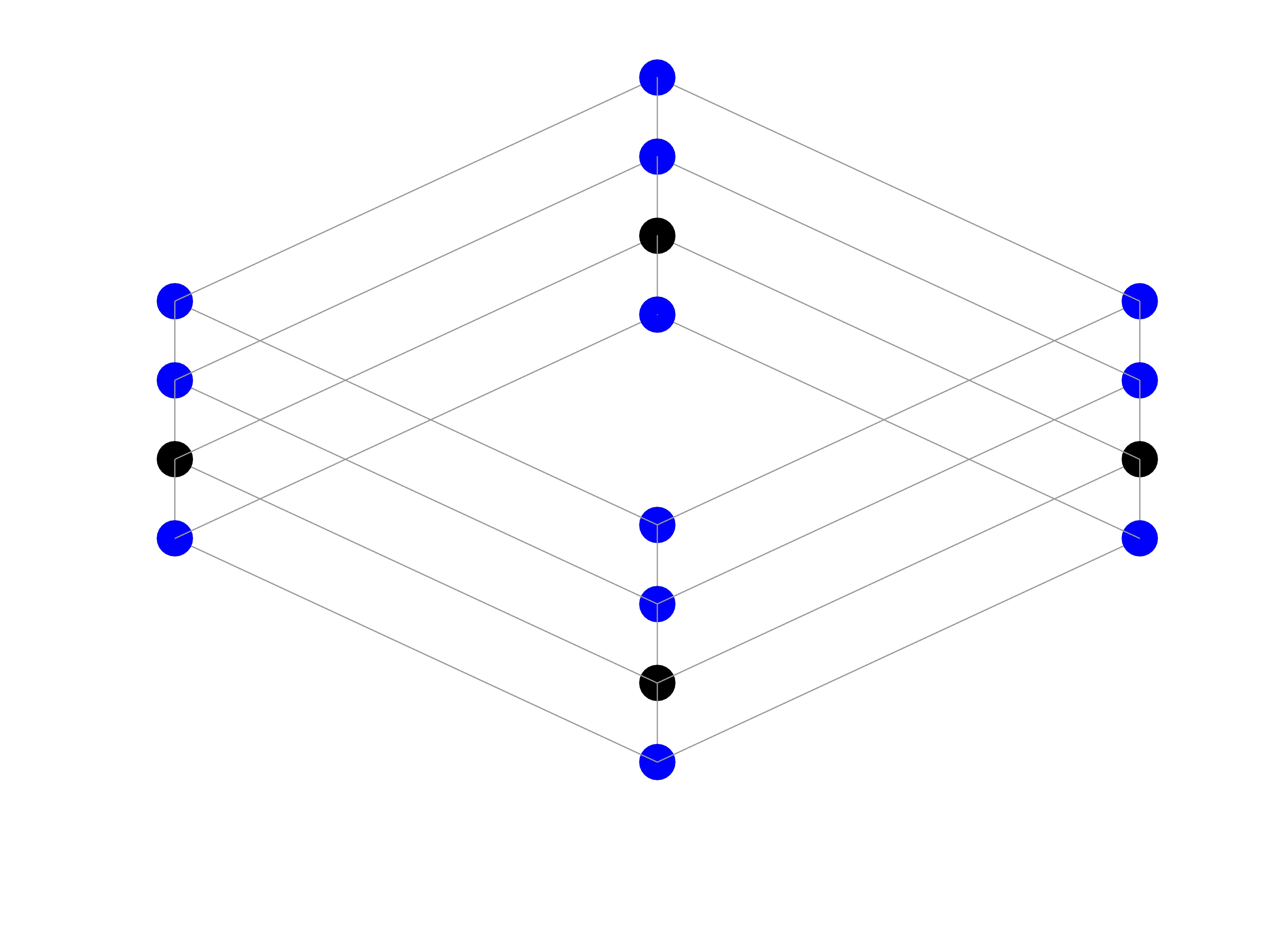}
			\parbox{2.5cm}{\tiny (a) Samples from HGFT sampling.}
		\end{minipage}
		\begin{minipage}[t]{0.45\linewidth}
			\centering
			\includegraphics[width=\linewidth]{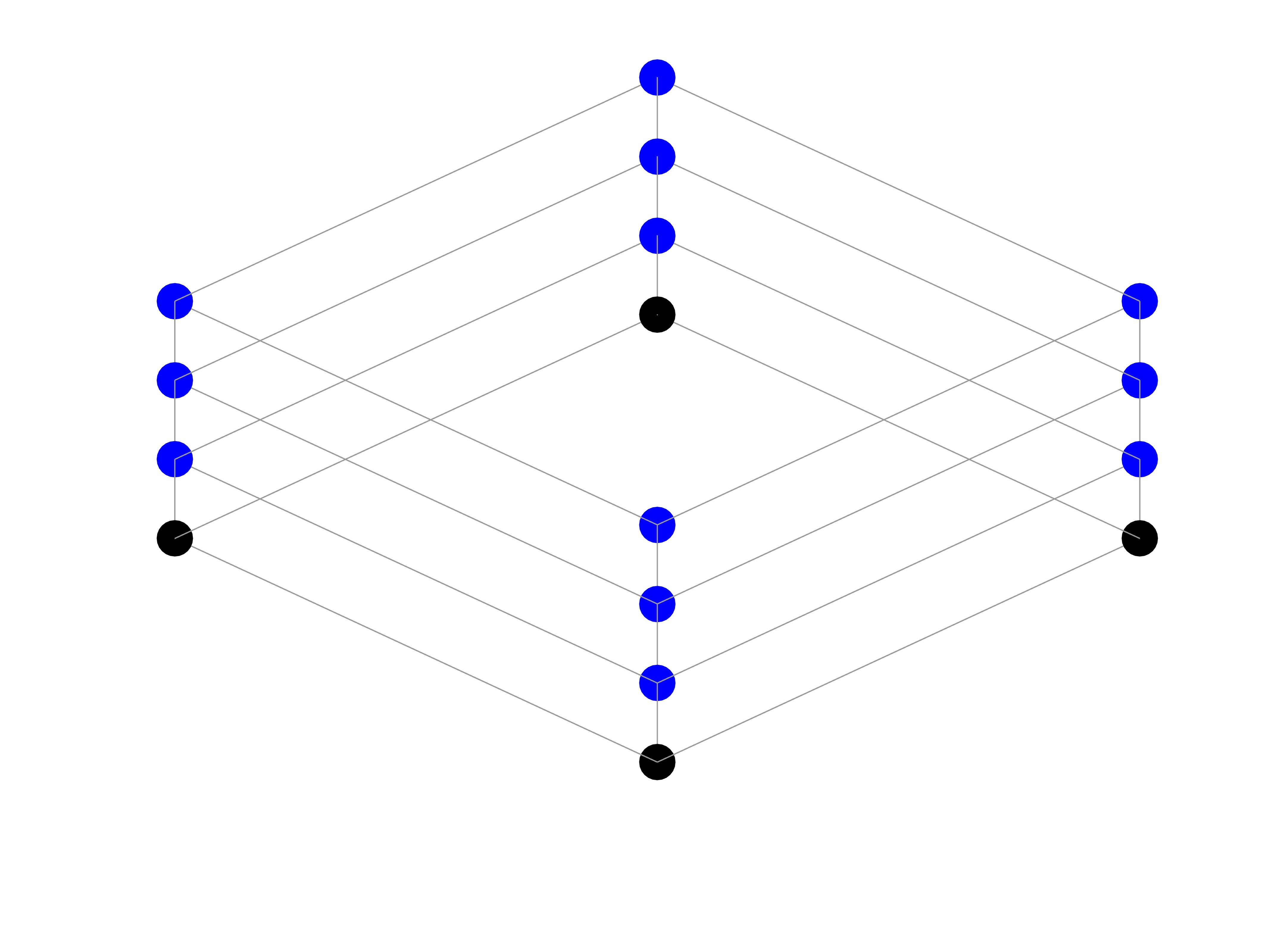}
			\parbox{2.7cm}{\tiny (b) Samples from HGFRFT sampling.}
		\end{minipage}
	\end{center}
	\caption{Sampling node locations (shown in black).}
	\vspace*{-3pt}
		\label{fig5}
\end{figure}

\begin{figure}[h!]
	\begin{center}
		\begin{minipage}[t]{0.45\linewidth}
			\centering
			\includegraphics[width=\linewidth]{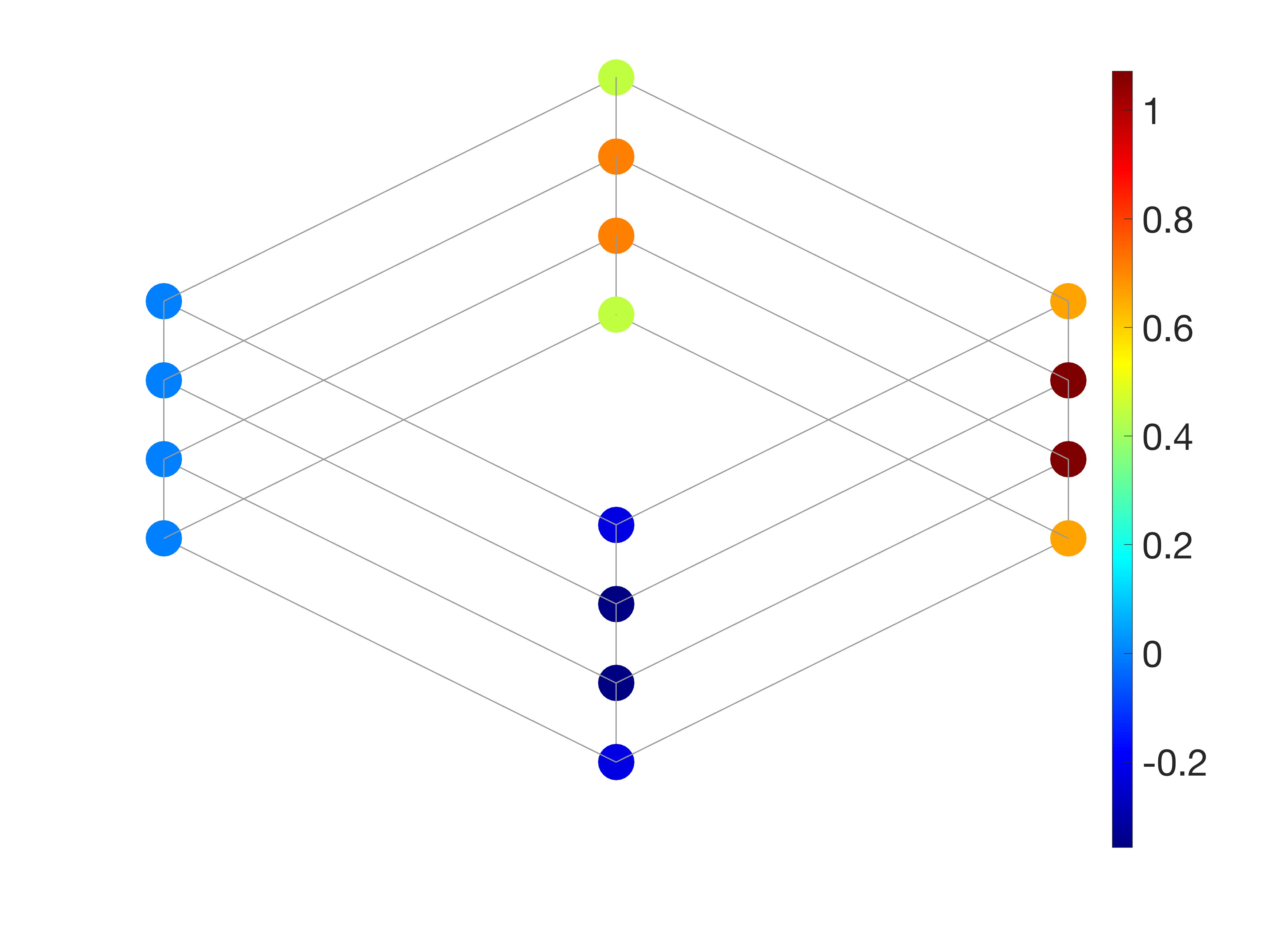}
			\parbox{3cm}{\tiny (a) Recovery from HGFT sampling.}
		\end{minipage}
		\begin{minipage}[t]{0.45\linewidth}
			\centering
			\includegraphics[width=\linewidth]{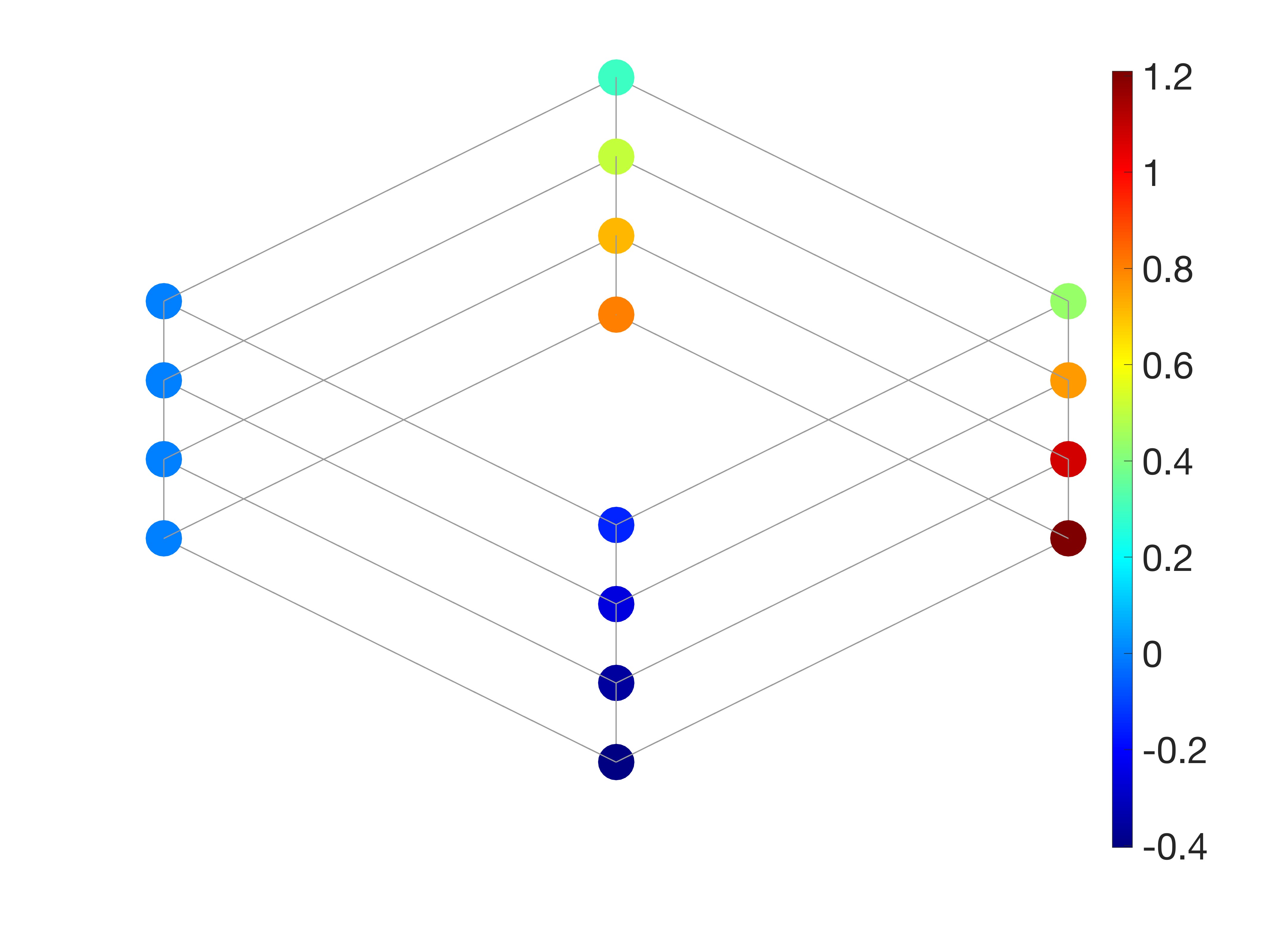}
			\parbox{3cm}{\tiny (b) Recovery from HGFRFT sampling.}
		\end{minipage}
	\end{center}
	\caption{Recovered signals on the product graph.}
	\vspace*{-3pt}
		\label{fig6}
\end{figure}

\begin{figure}[h!]
	\begin{center}
		\begin{minipage}[t]{0.45\linewidth}
			\centering
			\includegraphics[width=\linewidth]{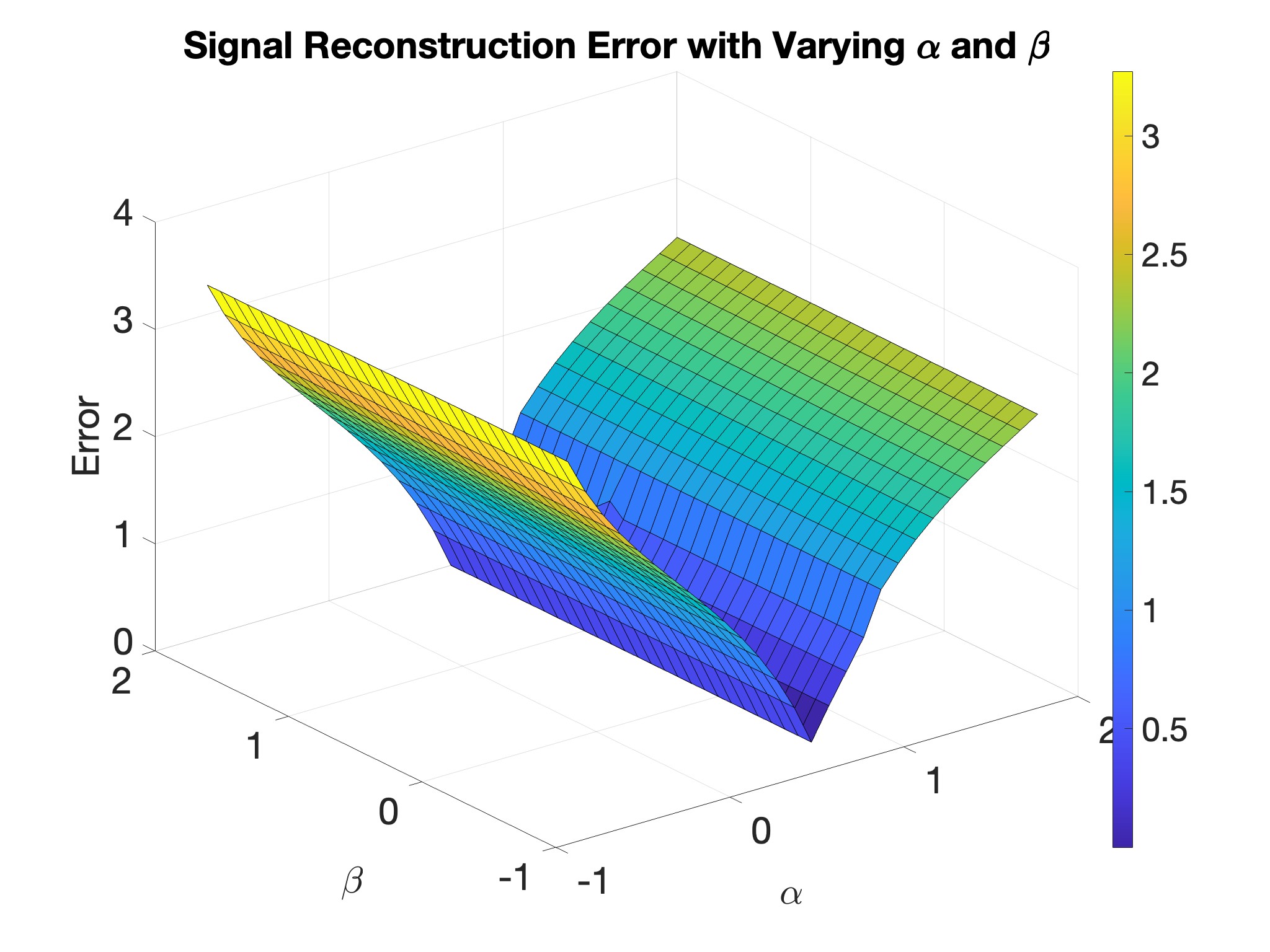}
			\parbox{2cm}{\tiny (a) $\epsilon$ vs. $\alpha$ and $\beta$.}
		\end{minipage}
		\begin{minipage}[t]{0.45\linewidth}
			\centering
			\includegraphics[width=\linewidth]{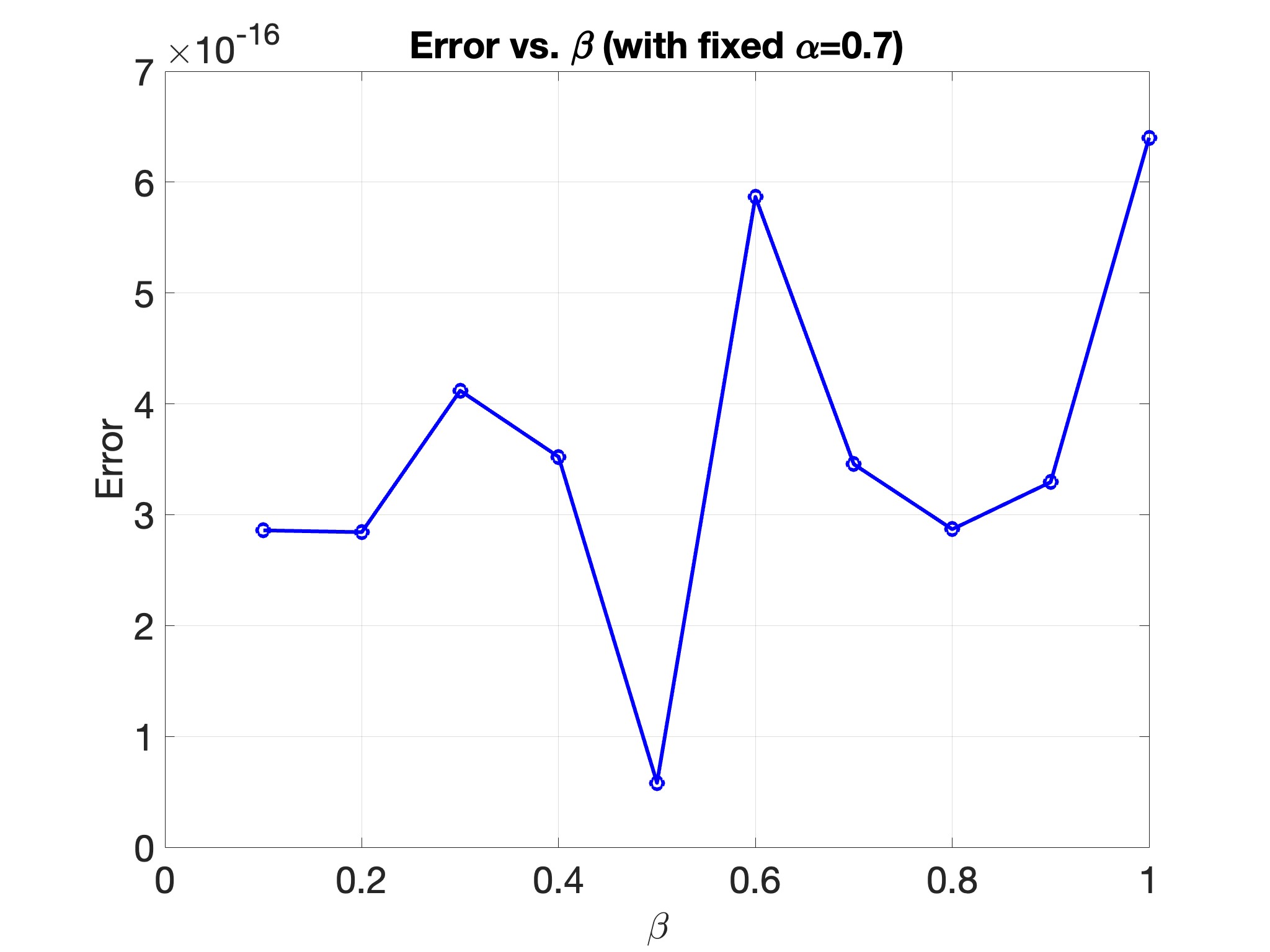}
			\parbox{2.5cm}{\tiny (b) $\epsilon$ vs. $\beta$ (with fixed $\alpha$=0.7).}
		\end{minipage}
	\end{center}
	\caption{Error value $\epsilon$ for different $\alpha$ and $\beta$.}
	\vspace*{-3pt}
	\label{fig7}
\end{figure}

\subsection{Graph Frequency Analysis on Epidemic Models}\label{sec6.2}
In this section, we illustrate how the HGFRFT can be applied to the complex dynamics of networks through epidemic models \cite{epidemics}, offering new methods and perspectives for the evolution of nonlinear, discrete, and uncertain models of infectious disease spread. In this graph, each node represents a city with a fixed population, and the edges are based on the terrestrial positions and flight connections between major European cities. We simulate the spread of an epidemic in $N = 704$ European cities using a model: the SEIRS model \cite{JFT}. The model is parameterized by the probability of infection transmission, the infectious period, the incubation period, and the immunity period.

\subsubsection{Construction and Simulation of The SEIRS Model}
This model explores the behavior of heat diffusion and wave equations in discrete settings by solving linear partial differential equations \cite{Diffusion}. In discrete setups, the heat diffusion equation is represented as $f_t - f_{t-1} = -s\mathcal{L} f_t$, where $f_t$ denotes the state at time $t$, $s$ is the heat diffusion rate indicating the speed of heat diffusion, and $\mathcal{L}$ is the graph Laplacian matrix that illustrates the connectivity between vertices in the graph. The solution to the heat equation can be obtained by applying $(\mathbf{I} – s\mathcal{L})^{t-1}$ to the initial condition $f_1$. When a signal undergoes the HGFRFT, the transformed signal is represented as
\[
\mathbf{Y}(l,k)=(\Phi^{-\beta}f_{1})(l)\cdot \Psi^{\alpha}(k,1), 
\]
where $\alpha$ and $\beta$ modulate the frequencies in the Hilbert space and graph, respectively. Here, $l$ indexes the graph frequencies, and $k$ indexes the Hilbert space frequencies. There is a unique structure in the Hilbert space-vertex spectral domain
\[
\hat{\mathbf{X}} \left( l,k\right) = \frac{1}{\sqrt{|\Psi^{\alpha}|}} \frac{a\left( \lambda_{l} ,\omega_{k}\right)^{\top} - 1}{a\left( \lambda_{l} ,\omega_{k}\right) - 1} \mathbf{Y}\left( l,k\right),
\]
where $a(\lambda_l, \omega_k) = (1 - s\lambda_l) e^{-j \omega_k}$, and $|\Psi^{\alpha}|$ denotes the dimension of the Hilbert space. This tends to preserve low-frequency information while filtering out high-frequency noise.

For the wave equation \cite{JFT}, it can be modeled through the discrete second-order differential equation $\mathbf{X}(\mathbf{B} - \mathrm{diag}(\mathbf{B1})) = s\mathcal{L}\mathbf{X}$, representing discrete waves propagating on the graph at a speed \(s > 0\). Hence, in the Hilbert space-vertex spectral domain, the solution is given by
\[
\hat{\mathbf{X}} \left( l,k\right)  =\sum_{t} \cos \left\{ t\cos^{-1} \left( 1-\frac{s\lambda_{l} }{2} \right)  \right\}  e^{-j\omega_{k} t}\cdot \mathbf{Y}\left( l,k\right),
\]
where $s < 4/\lambda_{\max}$ to maintain stability.

Using these equations, we derive the SEIRS model. Fig. \ref{fig8} shows the HGFRFT of the SEIRS model simulating an epidemic outbreak, illustrating the evolution of infection counts across the graph in the $(\omega, \lambda)$ plane, where brighter colors indicate higher energy levels. Figs. \ref{fig8}(a)-(j) demonstrate how the plane shifts with variations in the fractional pair $(\alpha, \beta)$. Lower $\beta$ values alter the graph's structure, concentrating energy around certain angular frequency axes, while $\beta = 1$ and reduced $\alpha$ increase energy across the image. The HGFRFT reveals spectral lines at regular intervals along the angular frequency axis, reflecting reinfection potential after temporary immunity. It also distinguishes between high and low infection probability scenarios, with higher outbreak probabilities corresponding to increased high-frequency energy, reflecting the impulsive nature of epidemics.
	\begin{figure*}[ht]
		\centering
		% 第一行
		\begin{minipage}[t]{0.18\linewidth}
			\centering
			\includegraphics[width=\linewidth]{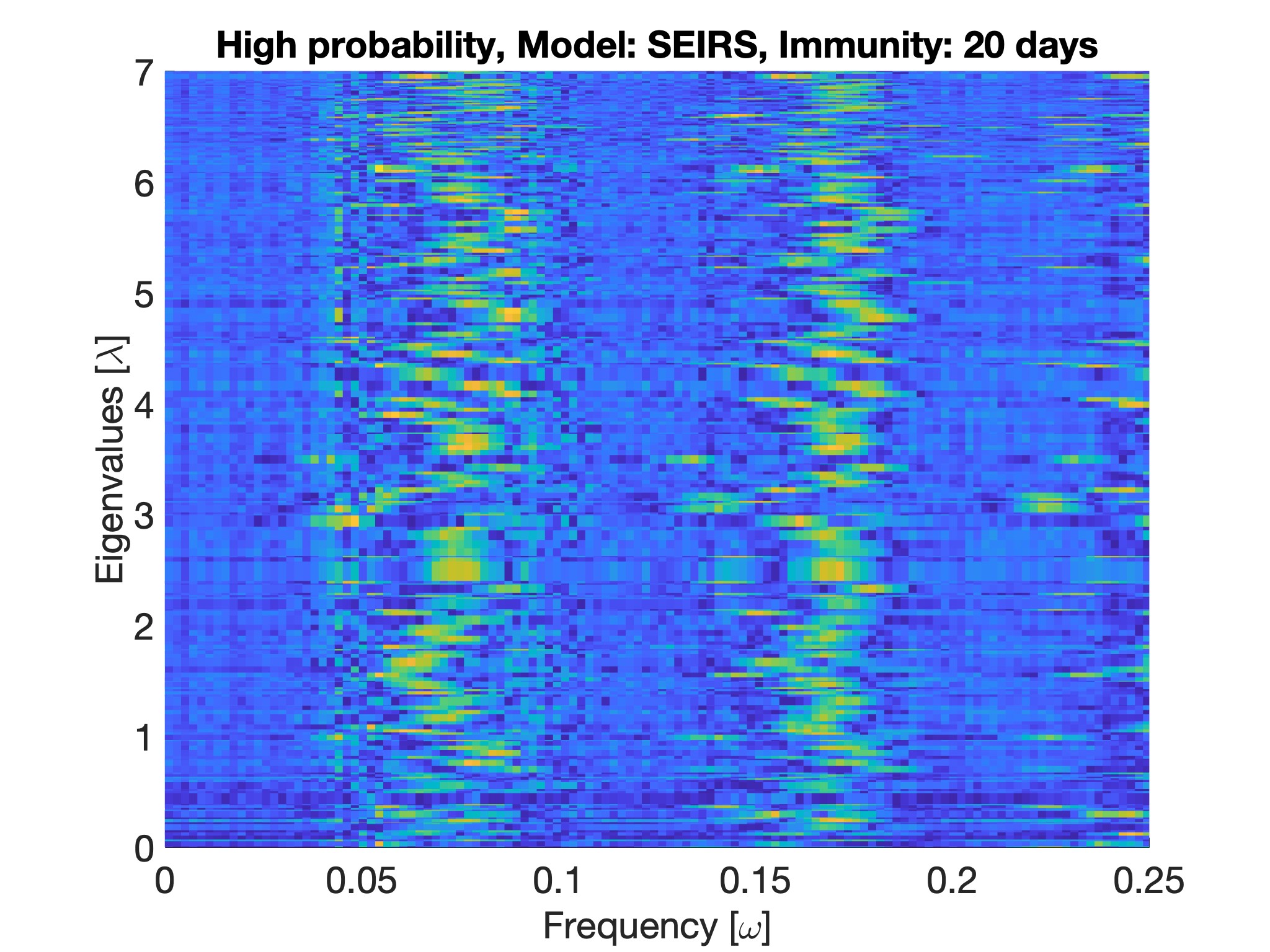}
			\vspace{-20pt} % 调整图片与标题之间的距离
			%\parbox{4cm}{\tiny (a) High probability of $\alpha=0.1, \beta=0.1$}
			\caption*{\tiny (a) High probability of $\alpha=0.1, \beta=0.1$.}
		\end{minipage}
		\hfill
		\begin{minipage}[t]{0.18\linewidth}
			\centering
			\includegraphics[width=\linewidth]{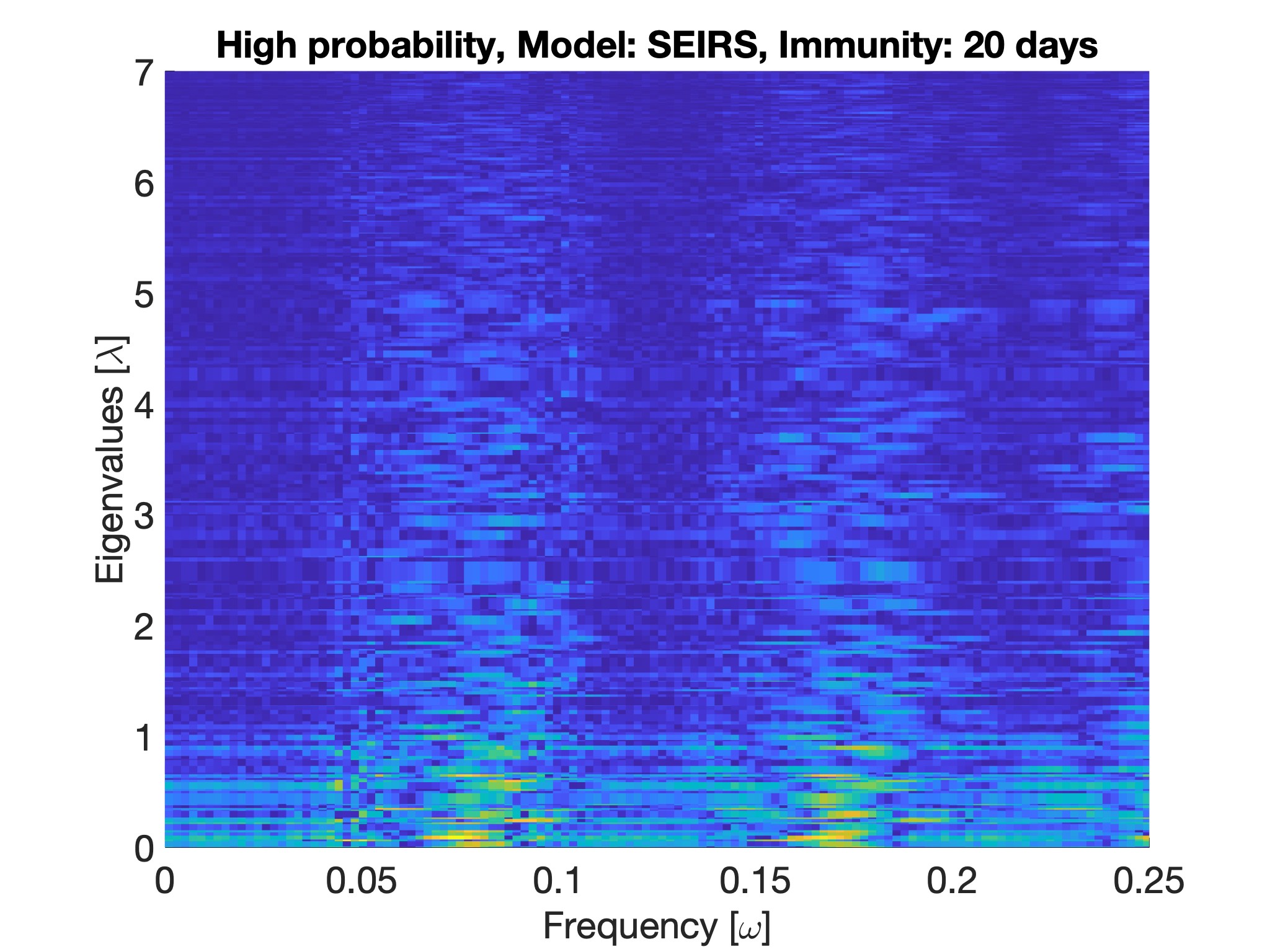}
			\vspace{-20pt} % 调整图片与标题之间的距离
			\caption*{\tiny (b) High probability of $\alpha=1, \beta=0.1$.}
		\end{minipage}
		\hfill
		\begin{minipage}[t]{0.18\linewidth}
			\centering
			\includegraphics[width=\linewidth]{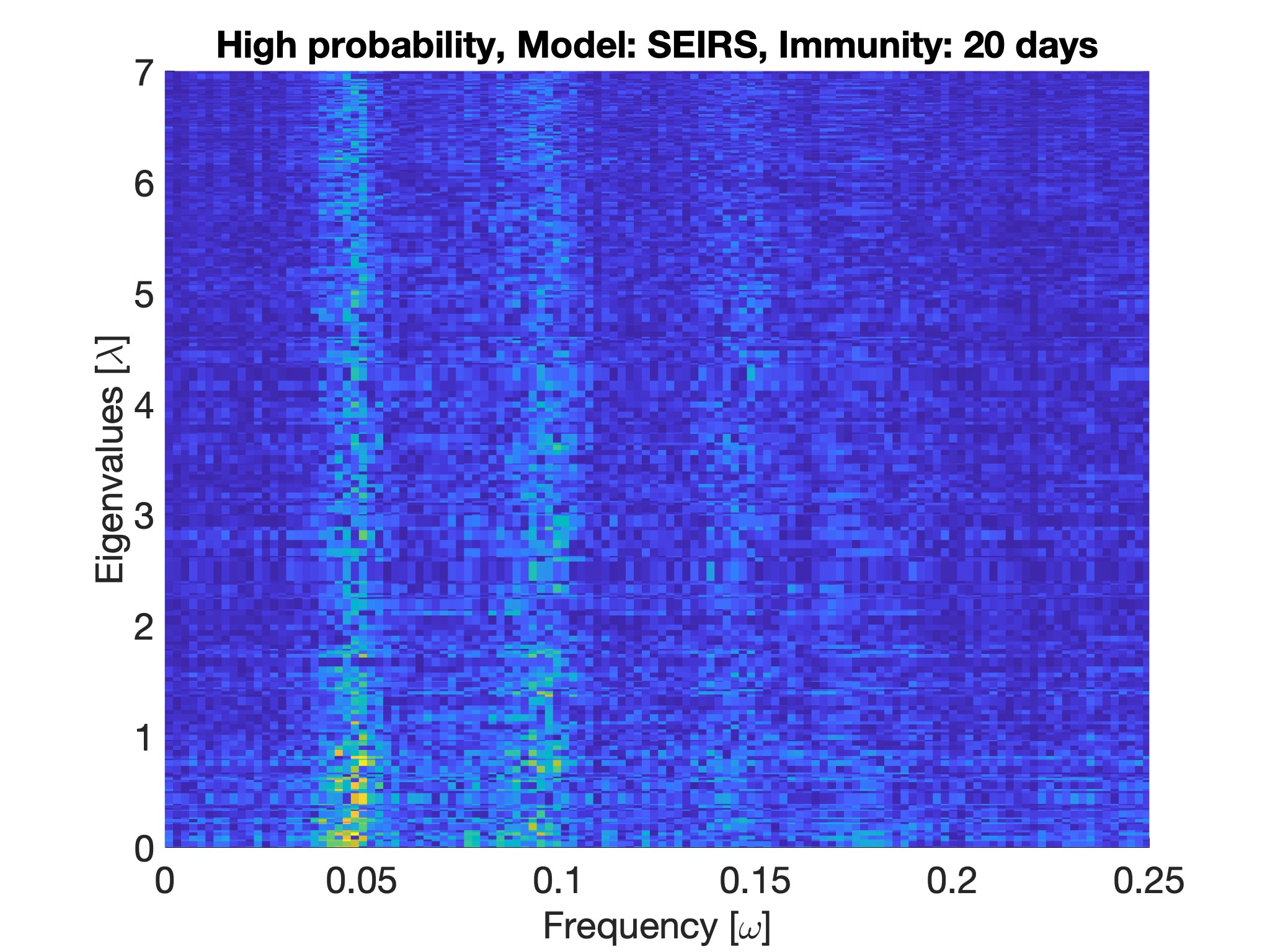}
			\vspace{-20pt} % 调整图片与标题之间的距离
			\caption*{\tiny (c) High probability of $\alpha=0.5, \beta=0.5$.}
		\end{minipage}
		\hfill
		\begin{minipage}[t]{0.18\linewidth}
			\centering
			\includegraphics[width=\linewidth]{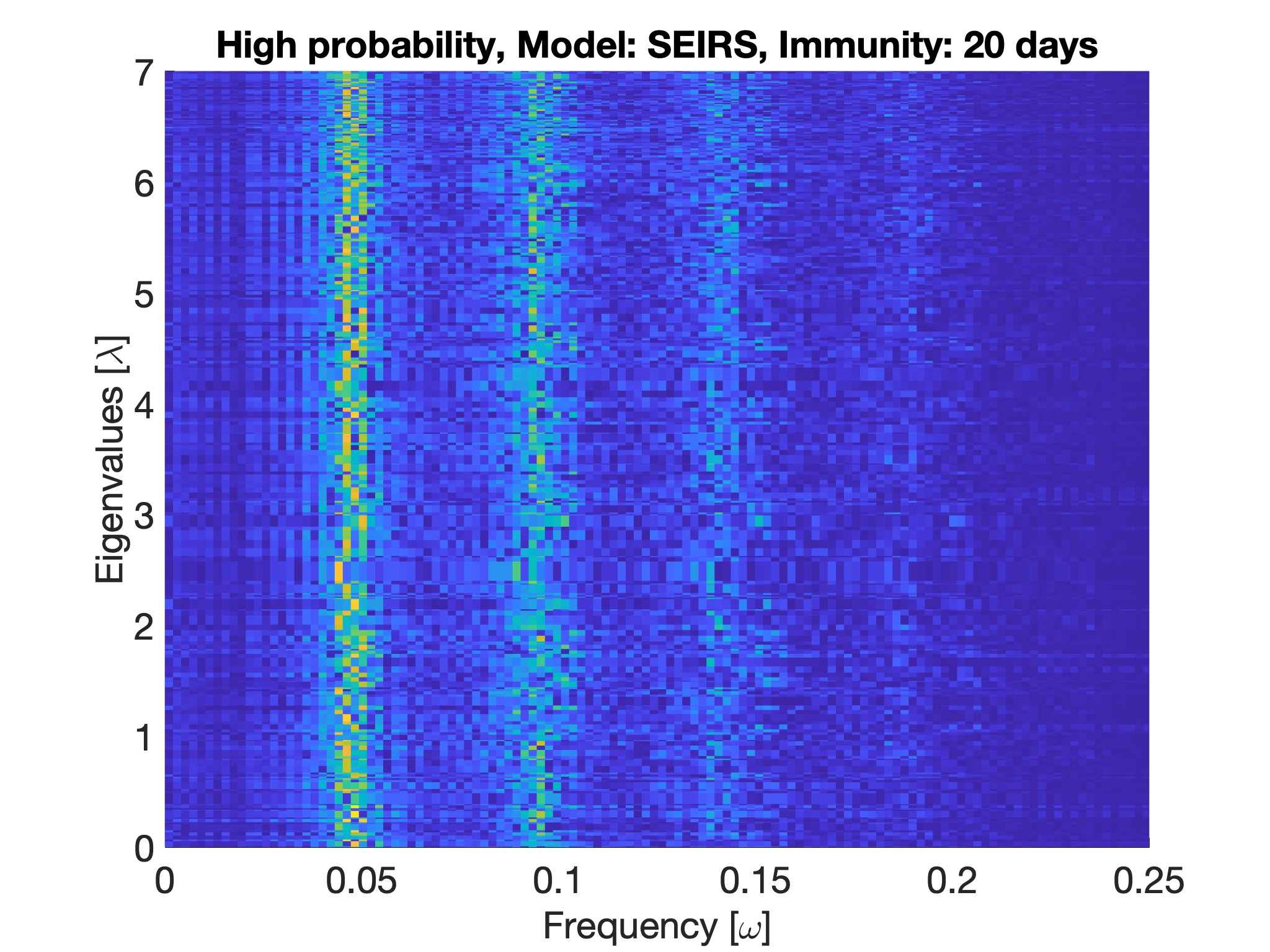}
			\vspace{-20pt} % 调整图片与标题之间的距离
			\caption*{\tiny (d) High probability of $\alpha=0.1, \beta=1$.}
		\end{minipage}
		\hfill
		\begin{minipage}[t]{0.18\linewidth}
			\centering
			\includegraphics[width=\linewidth]{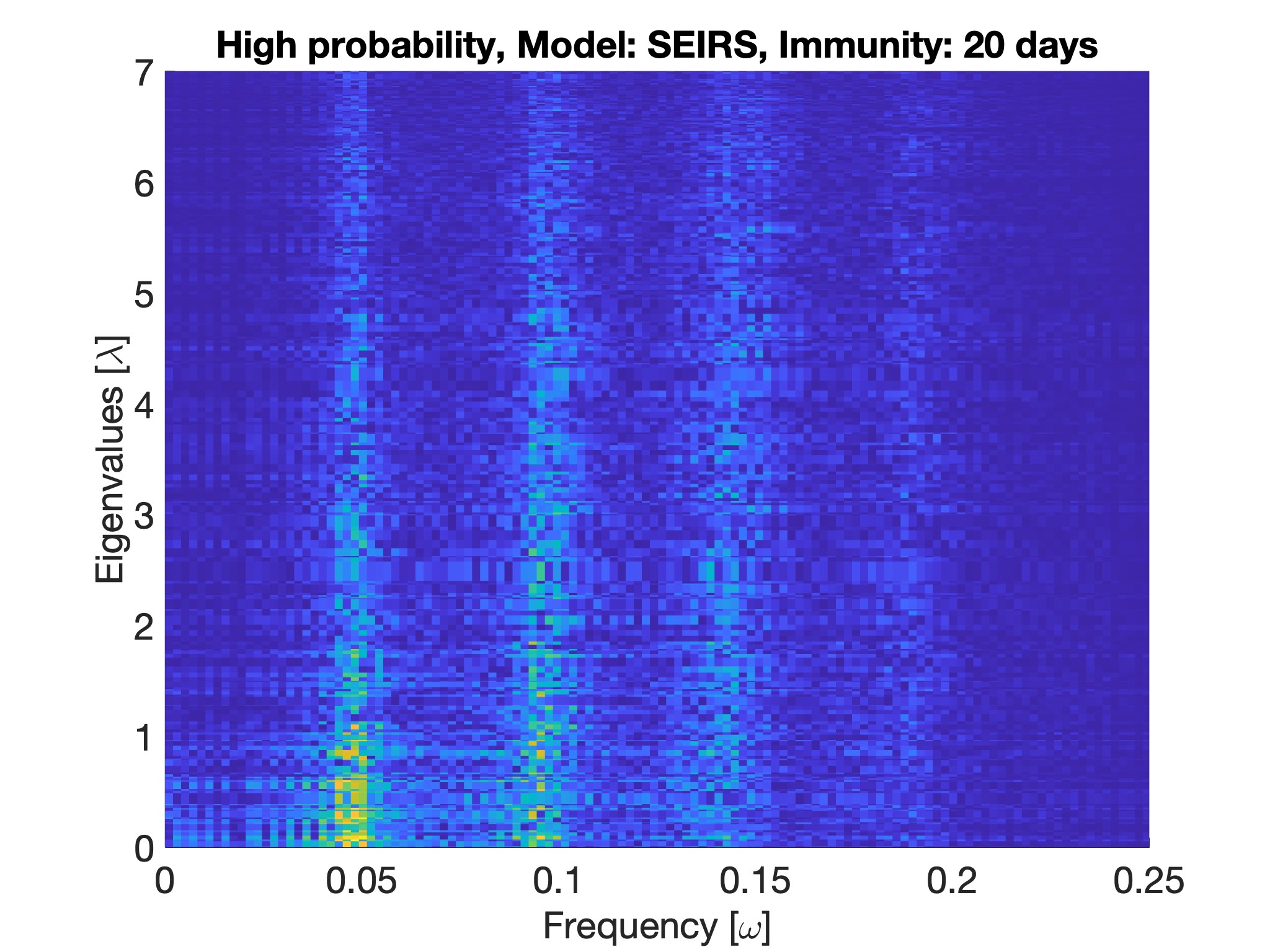}
			\vspace{-20pt} % 调整图片与标题之间的距离
			\caption*{\tiny (e) High probability of $\alpha=1, \beta=1$.}
		\end{minipage}
		
		%\vspace{4mm} % 在两行之间添加一些垂直空间
		
		% 第二行
		\begin{minipage}[t]{0.18\linewidth}
			\centering
			\includegraphics[width=\linewidth]{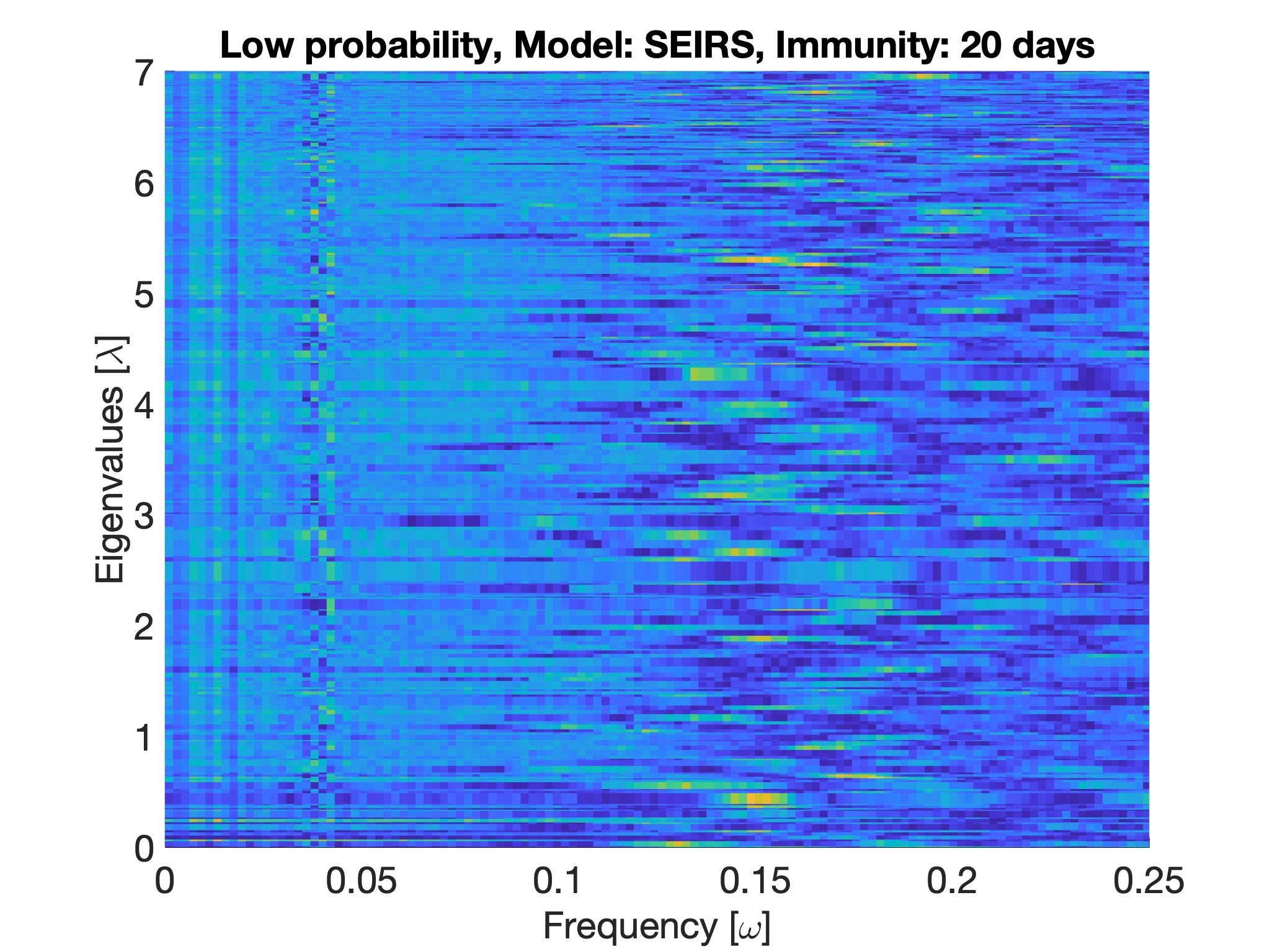}
			\vspace{-20pt} % 调整图片与标题之间的距离
			\caption*{\tiny (f) Low probability of $\alpha=0.1, \beta=0.1$.}
		\end{minipage}
		\hfill
		\begin{minipage}[t]{0.18\linewidth}
			\centering
			\includegraphics[width=\linewidth]{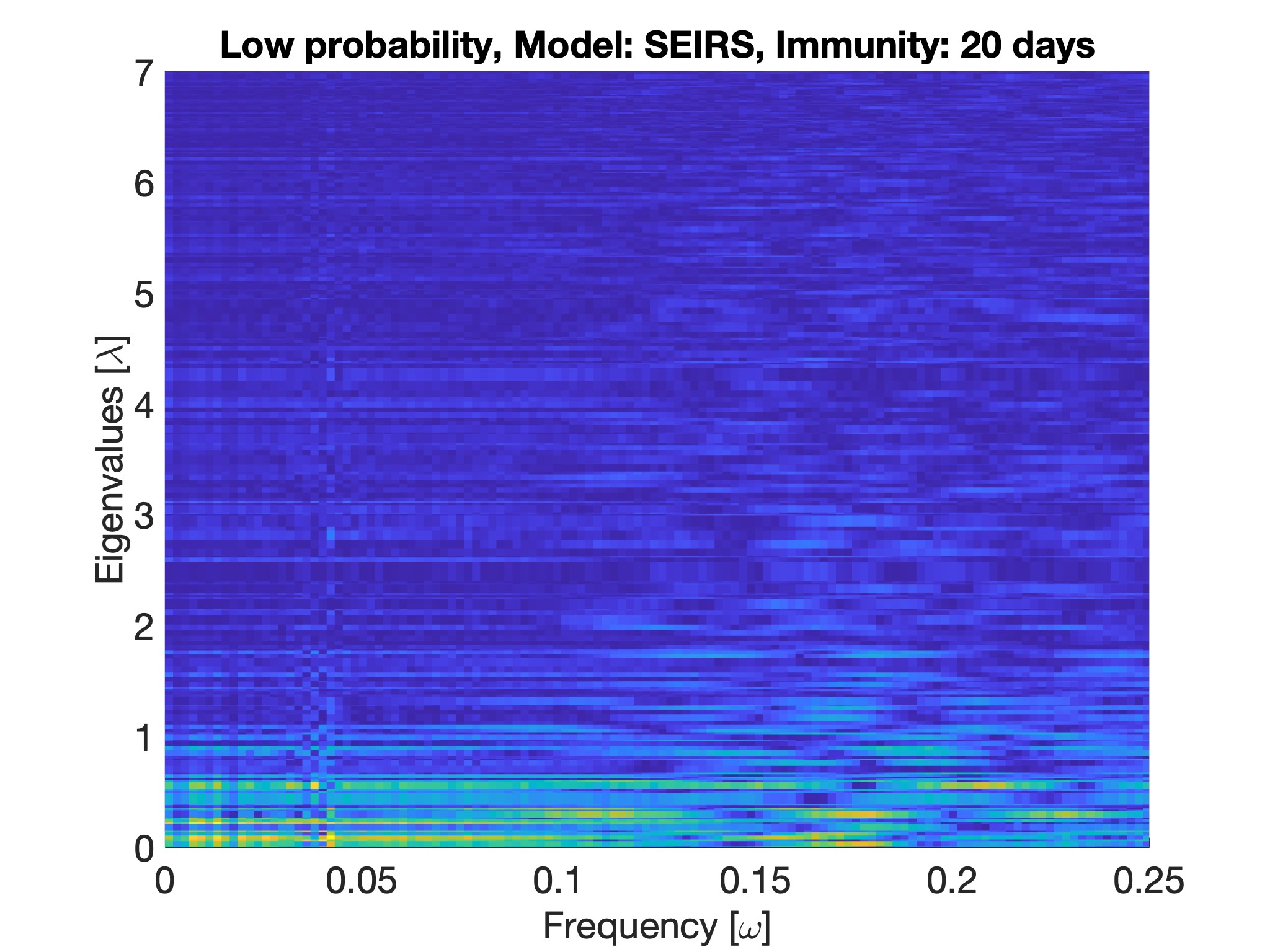}
			\vspace{-20pt} % 调整图片与标题之间的距离
			\caption*{\tiny (g) Low probability of $\alpha=1, \beta=0.1$.}
		\end{minipage}
		\hfill
		\begin{minipage}[t]{0.18\linewidth}
			\centering
			\includegraphics[width=\linewidth]{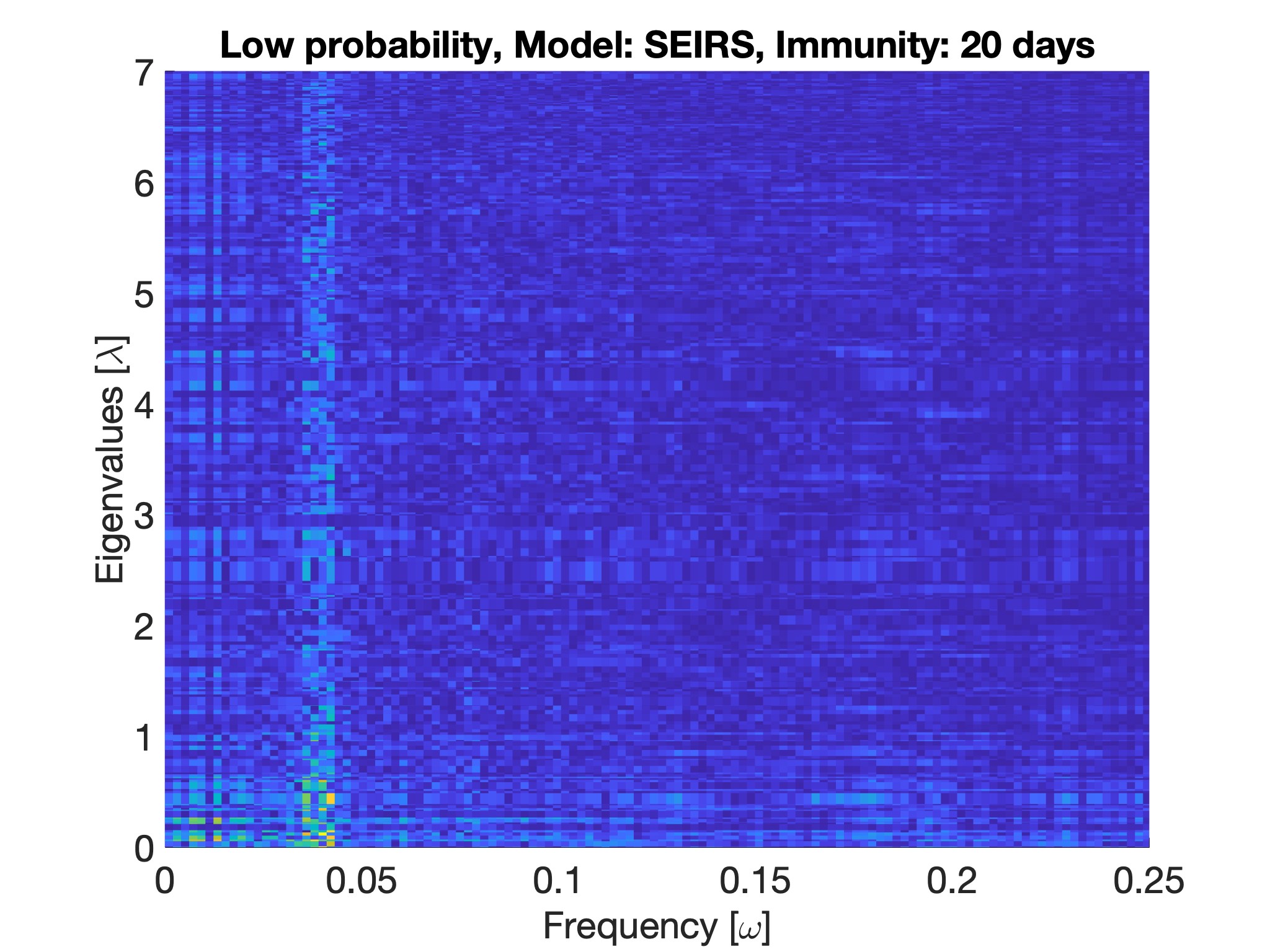}
			\vspace{-20pt} % 调整图片与标题之间的距离
			\caption*{\tiny (h) Low probability of $\alpha=0.5, \beta=0.5$.}
		\end{minipage}
		\hfill
		\begin{minipage}[t]{0.18\linewidth}
			\centering
			\includegraphics[width=\linewidth]{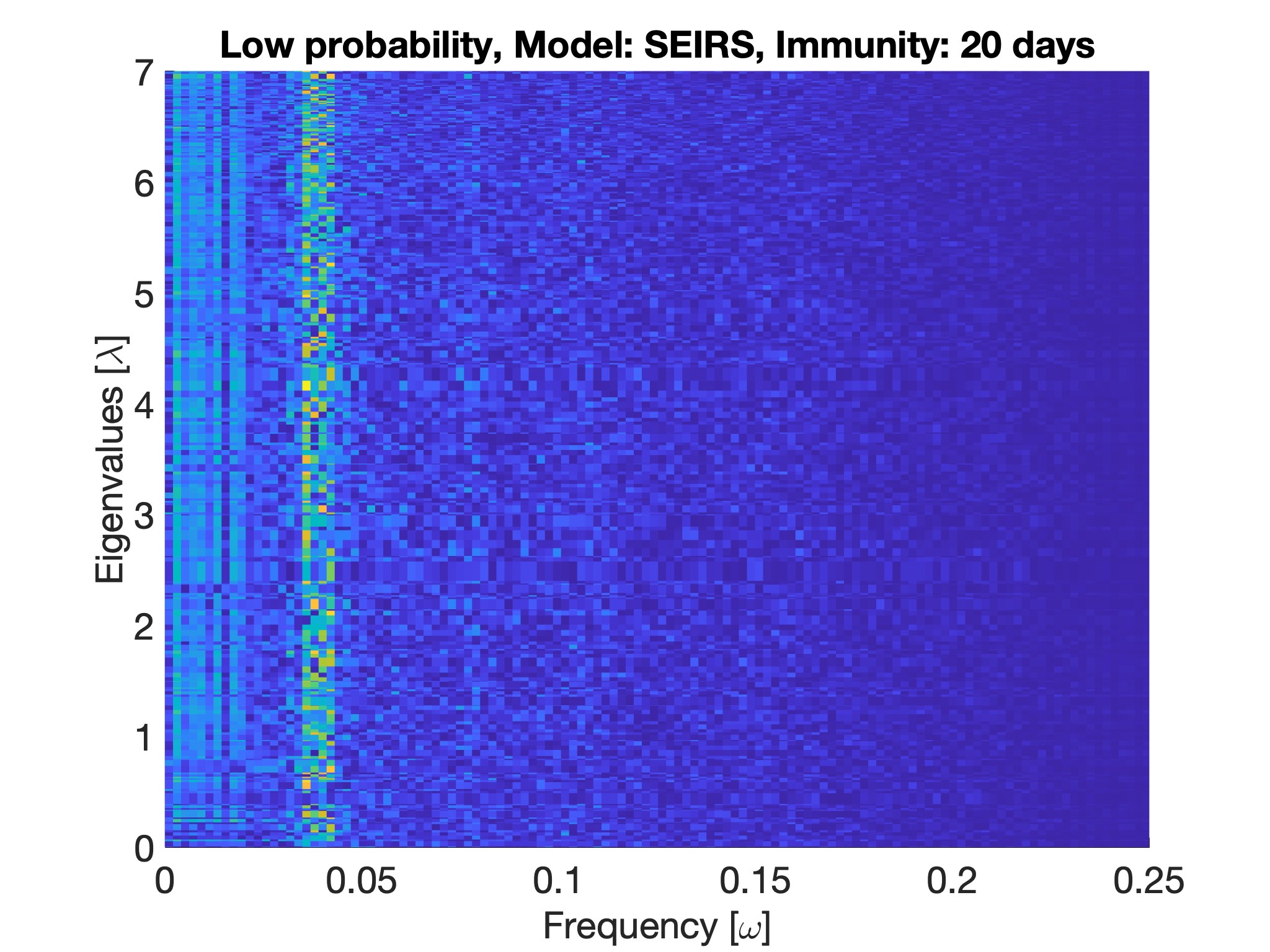}
			\vspace{-20pt} % 调整图片与标题之间的距离
			\caption*{\tiny (i) Low probability of $\alpha=0.1, \beta=1$.}
		\end{minipage}
		\hfill
		\begin{minipage}[t]{0.18\linewidth}
			\centering
			\includegraphics[width=\linewidth]{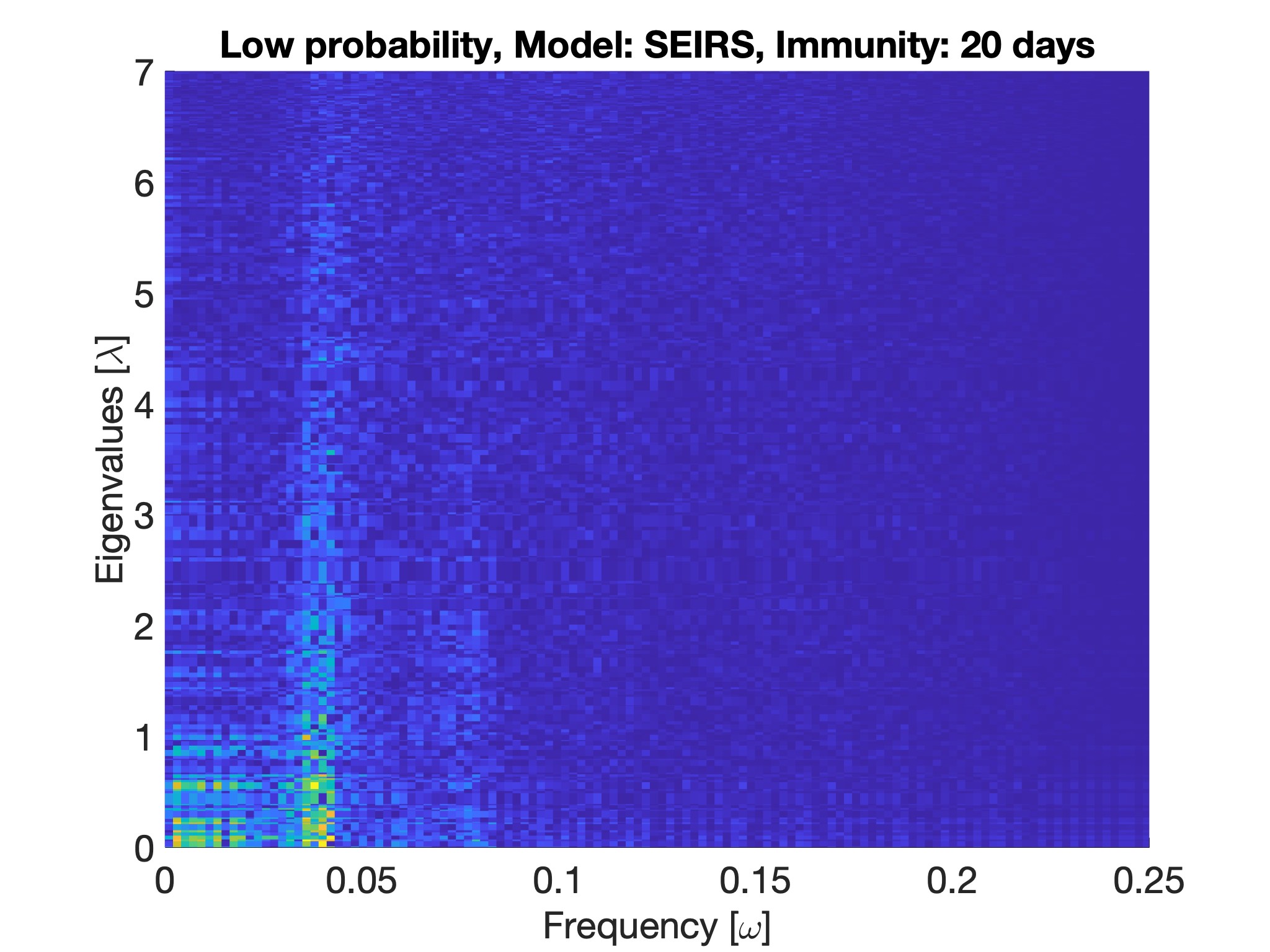}
			\vspace{-20pt} % 调整图片与标题之间的距离
			\caption*{\tiny (j) Low probability of $\alpha=1, \beta=1$.}
		\end{minipage}
		\caption{The HGFRFT of the simulated epidemic breakouts for the SEIRS model.}
		\label{fig8}
	\end{figure*}
	
\subsubsection{Comparison of Energy Compactness}
In this study, we selected energy compactness as the primary metric to compare the HGFRFT, the DFRFT ($\mathcal{H}^{\alpha}_f$-transform), and the GFRFT ($\mathcal{G}^{\beta}_f$-transform) because it effectively measures the concentration of signal energy in the transform domain. Energy compactness quantifies how well the signal energy is concentrated in a small number of transform coefficients, which is crucial for applications such as compression, denoising, and feature extraction. By calculating the normalized error after discarding transform coefficients below a certain percentile threshold, we can assess the efficiency of each transform in compressing signal energy. Higher energy compactness indicates that the signal energy is concentrated in fewer coefficients, leading to improved processing performance.

As shown in Figs. \ref{fig9}(a)–(d), based on the SEIRS model dataset for European infection numbers, the results demonstrate that as $\alpha$ decreases, the behavior of DFRFT gradually converges to that of HGFRFT; similarly, as $\beta$ decreases, GFRFT approaches HGFRFT. This analysis highlights the advantages of HGFRFT in terms of energy concentration and flexibility.
	\begin{figure}[h]
		\centering
		% 第一行
		\begin{minipage}[t]{0.47\linewidth}
			\centering
			\includegraphics[width=\linewidth]{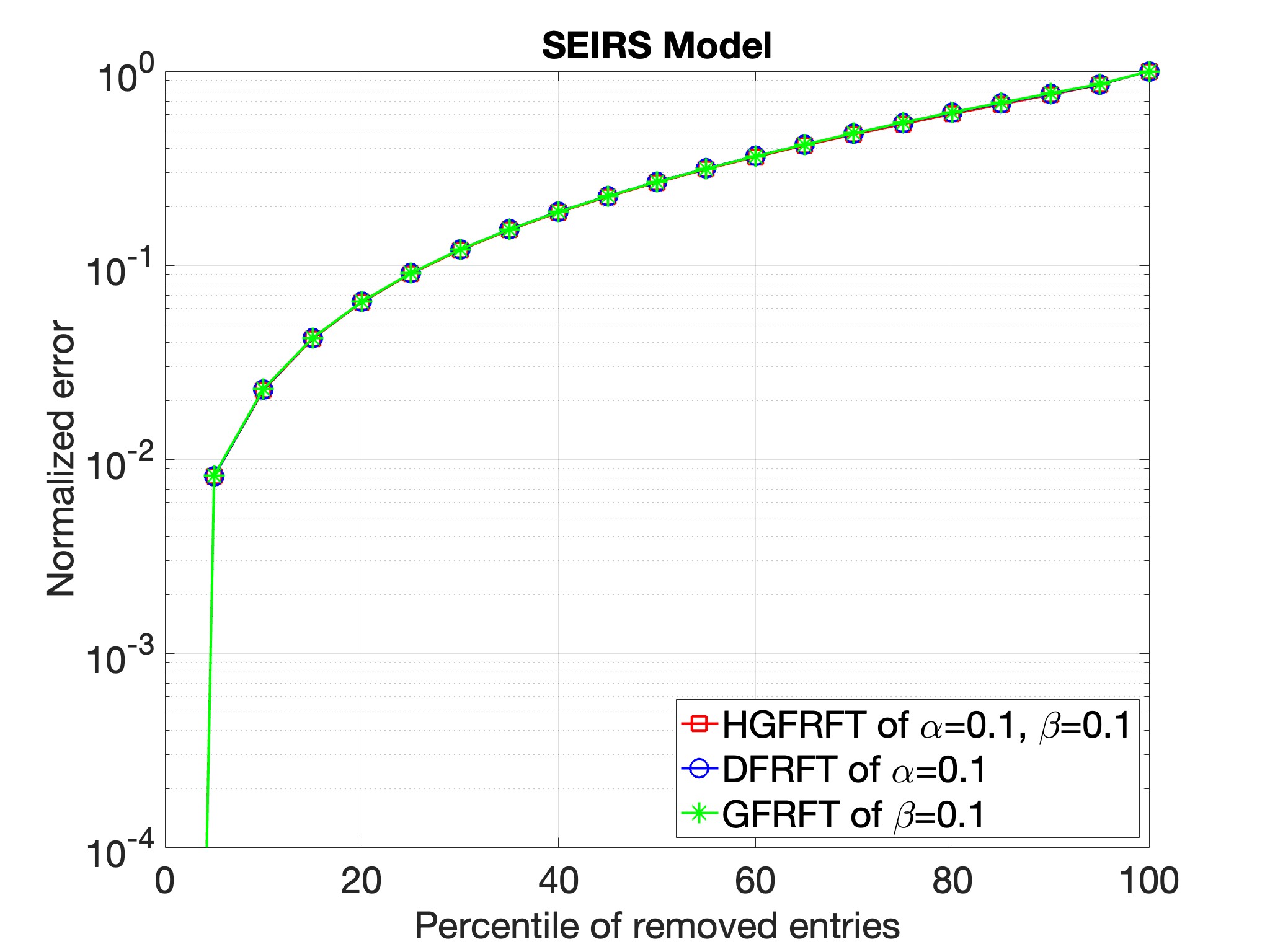}
			\vspace{-20pt} % 调整图片与标题之间的距离
			\caption*{\tiny (a) The epidemic model of $\alpha=0.1, \beta=0.1$.}
		\end{minipage}
		\begin{minipage}[t]{0.47\linewidth}
			\centering
			\includegraphics[width=\linewidth]{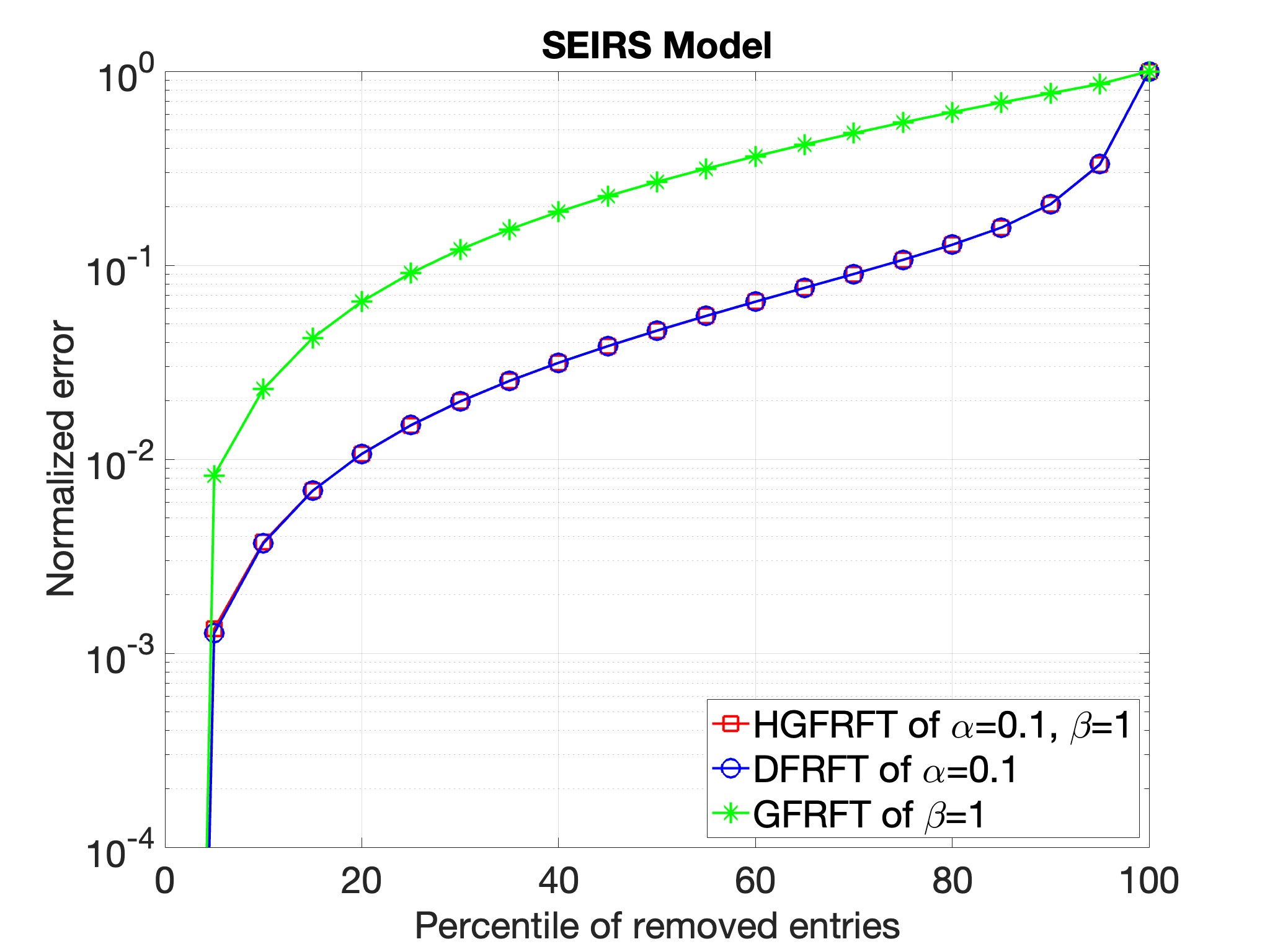}
			\vspace{-20pt} % 调整图片与标题之间的距离
			\caption*{\tiny (b) The epidemic model of $\alpha=0.1, \beta=1$.}
		\end{minipage}
		\begin{minipage}[t]{0.47\linewidth}
			\centering
			\includegraphics[width=\linewidth]{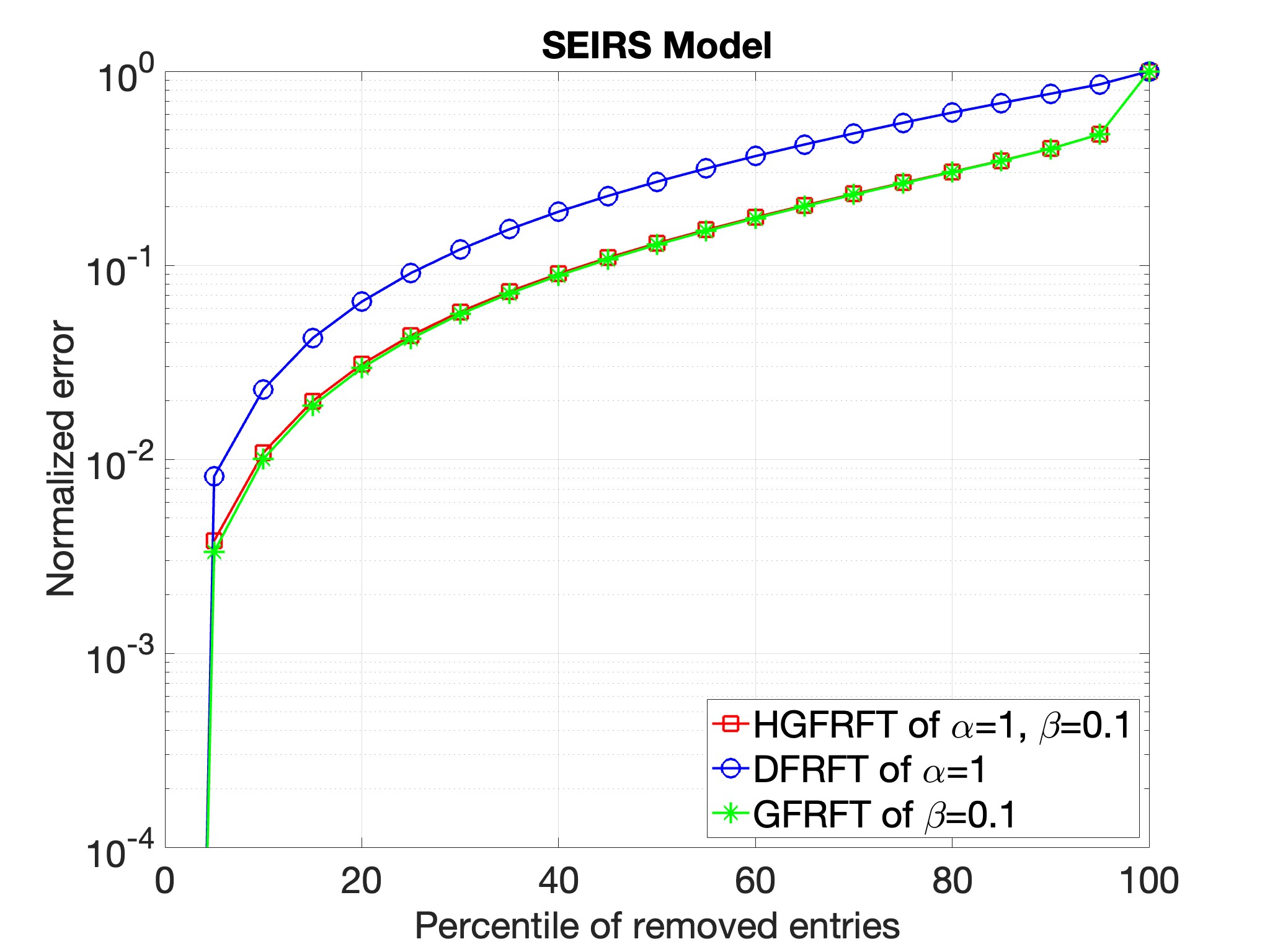}
			\vspace{-20pt} % 调整图片与标题之间的距离
			\caption*{\tiny (c) The epidemic model of $\alpha=1, \beta=0.1$.}
		\end{minipage}
		\begin{minipage}[t]{0.47\linewidth}
			\centering
			\includegraphics[width=\linewidth]{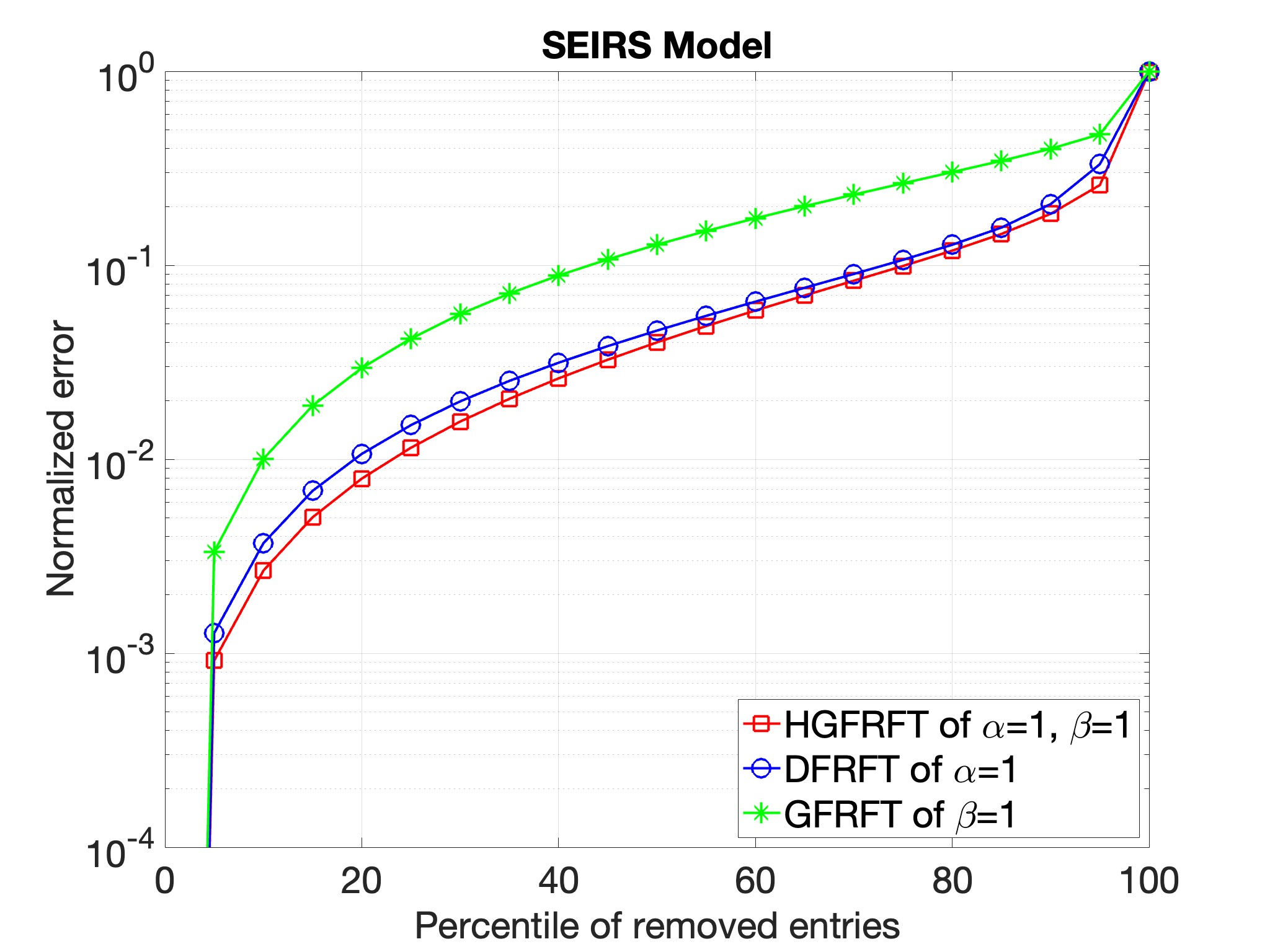}
			\vspace{-20pt} % 调整图片与标题之间的距离
			\caption*{\tiny (d) The epidemic model of $\alpha=1, \beta=1$.}
		\end{minipage}
		\caption{Compactness of different datasets.}
		\label{fig9}
	\end{figure}

\subsubsection{Comparison with The JFRFT}
%This study demonstrates the application of the HGFRFT and the JFRFT \cite{JFRFT} to the SEIRS model of epidemic propagation in networks in spectral analysis, leading to a refined evaluation of the signal propagation properties on the graph.
%
%%% 更改矩形框的FRFT 和 等距采样槽的DFRFT
%We define a step function $f(\cdot,v) \in L^{2}(\Omega)$ for each node $v \in \mathcal{V}$ and compute its HGFRFT and JFRFT. The HGFRFT applies the $\mathcal{H}^{\alpha}_{f}$-transform, utilizing the property that the FRFT of a rectangular function is a sinc function. In contrast, JFRFT partitions the observation period into uniform intervals, recording each node’s state at the start of each segment as the graph signal, and then applies the DFRFT. Notably, $f(\cdot, v)$ is not time-bandlimited for any $v \in \mathcal{V}$, implying that discrete sampling alone cannot fully reconstruct the original signal.
%
%The dark blue color in Fig. \ref{fig10} indicates that HGFRFT or JFRFT have lower magnitudes. The HGFRFT shows more pronounced spectral dispersion in high-energy regions than JFRFT, reflecting its use of the full time-series data at each node, whereas JFRFT samples discrete intervals. Furthermore, Figs. \ref{fig10}(a)-(c) illustrate how variations in $\alpha$ and $\beta$ enhance spectral energy visibility, akin to patterns seen in Fig. \ref{fig8}.
This study demonstrates the application of the HGFRFT and the JFRFT \cite{JFRFT} to the SEIRS model of epidemic propagation in networks in spectral analysis, leading to a refined evaluation of the signal propagation properties on the graph.

We define a step function $f(\cdot,v) \in L^{2}(\Omega)$ for each node $v \in \mathcal{V}$ and compute its HGFRFT and JFRFT. First, the HGFRFT is applied via the $\mathcal{H}^{\alpha}_{f}$-transform, leveraging the FRFT properties of the rectangular function, where the FRFT acts as an interpolation between the time and frequency domains representations. As the order $\alpha$ of the FRFT varies from $0 \to 1$, the rectangular function gradually transitions into a $\mathrm{sinc}$ function, indicating a smooth transition between the time and frequency domains. In contrast, the JFRFT divides the observation period into uniform intervals, recording the state of each node at the beginning of each time slot as the graph signal, followed by the application of DFRFT to these discrete signals. Finally, both methods are subjected to GFRFT, resulting in the joint spectrum of the transformed signals. Notably, $f(\cdot, v)$ is not time-bandlimited for any $v \in \mathcal{V}$, implying that discrete sampling alone cannot fully reconstruct the original signal.

The dark blue color in Fig. \ref{fig10} indicates that HGFRFT or JFRFT have lower magnitudes. The HGFRFT shows more pronounced spectral dispersion in high-energy regions than JFRFT, reflecting its use of the full time-series data at each node, whereas JFRFT samples discrete intervals. Furthermore, Figs. \ref{fig10}(a)-(c) illustrate how variations in $\alpha$ and $\beta$ enhance spectral energy visibility, akin to patterns seen in Fig. \ref{fig8}.
\begin{figure}[h!]
	\begin{center}
		\begin{minipage}[t]{0.32\linewidth}
			\centering
			\includegraphics[width=\linewidth]{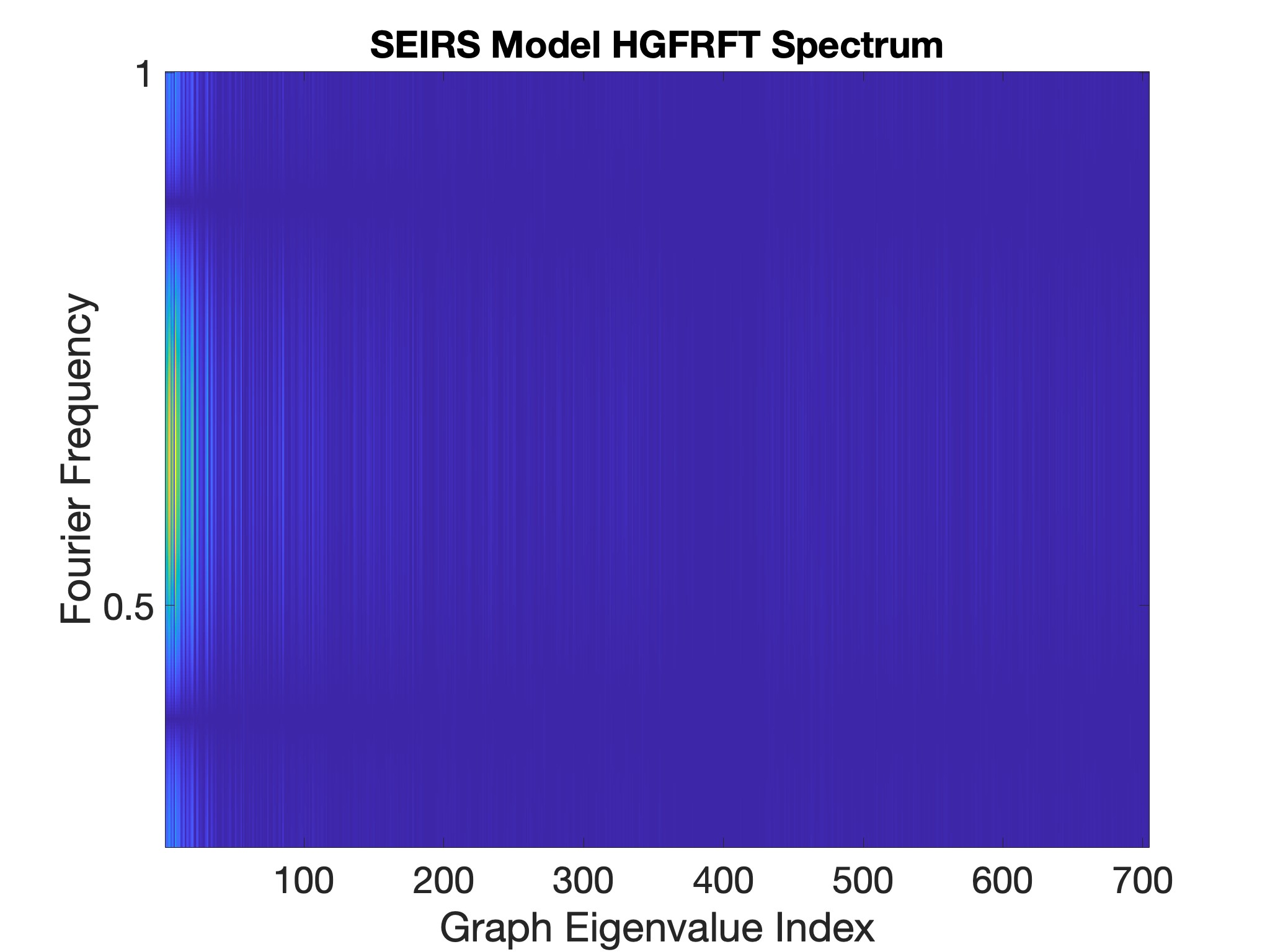}
			\parbox{2.6cm}{\tiny (a) The HGFRFT of $\alpha=1, \beta=1$.}
		\end{minipage}
		\begin{minipage}[t]{0.32\linewidth}
			\centering
			\includegraphics[width=\linewidth]{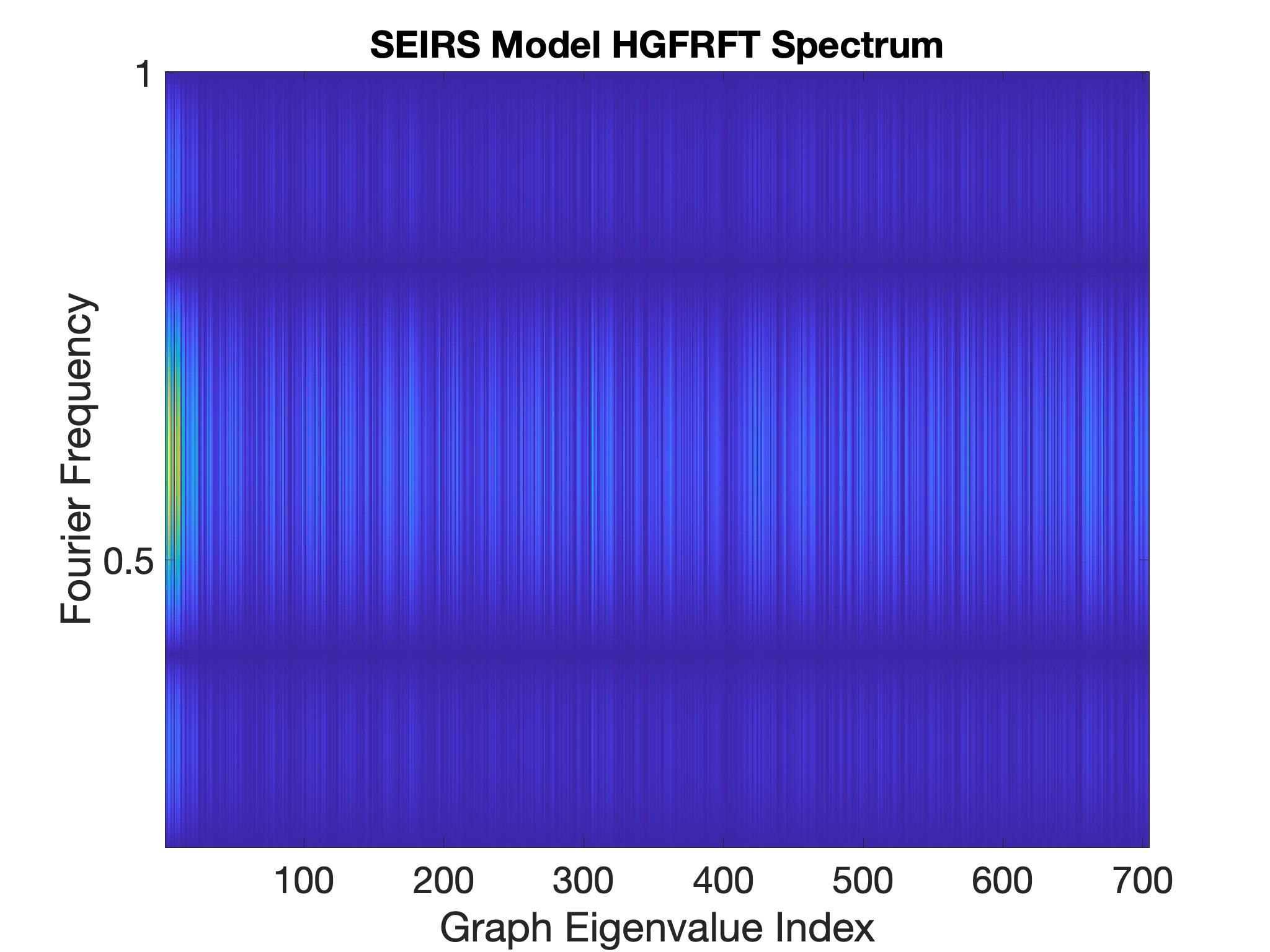}
			\parbox{2.6cm}{\tiny (b) The HGFRFT of $\alpha=1, \beta=0.5$.}
		\end{minipage}
		\begin{minipage}[t]{0.32\linewidth}
			\centering
			\includegraphics[width=\linewidth]{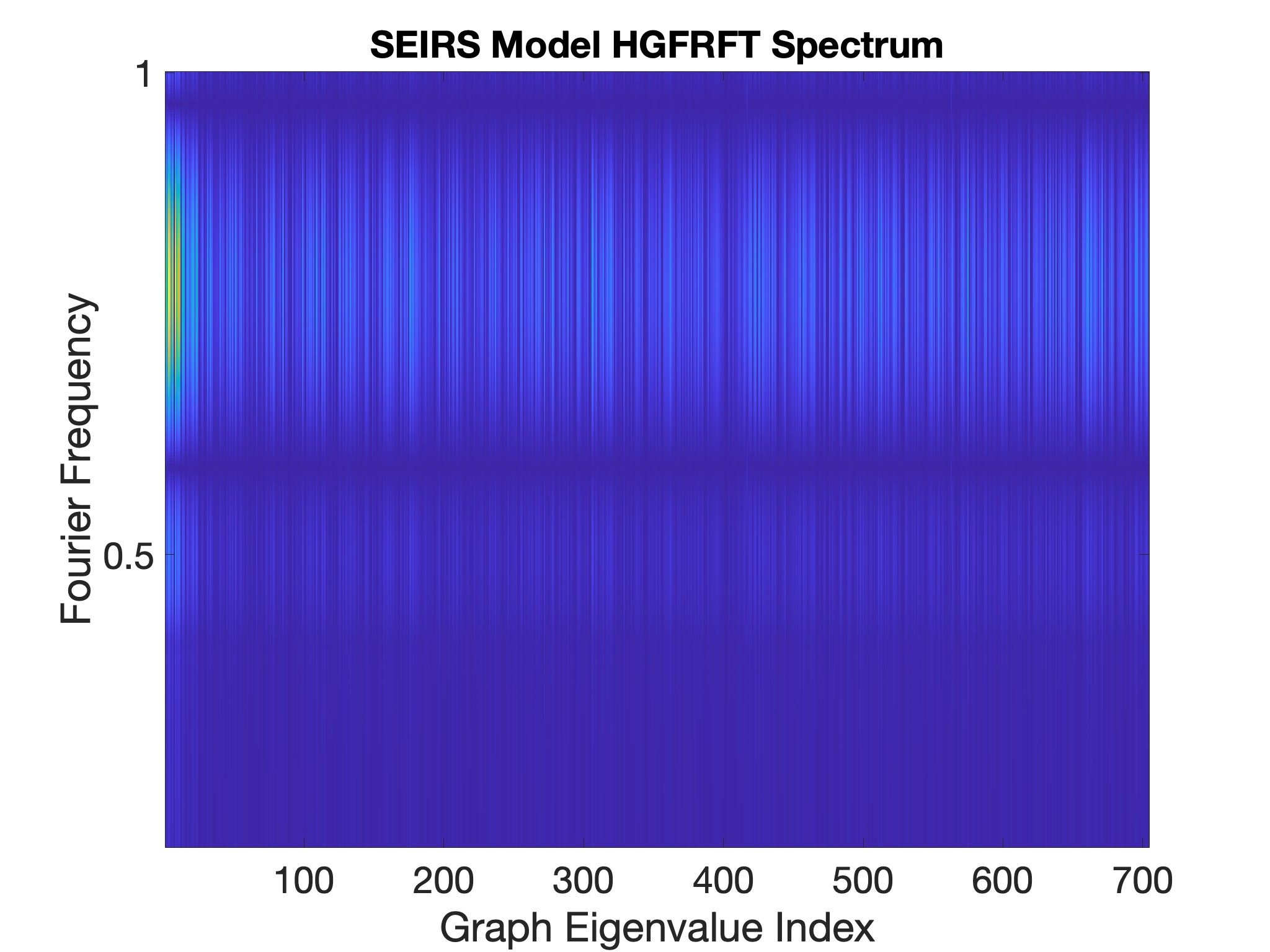}
			\parbox{2.6cm}{\tiny (c) The HGFRFT of $\alpha=0.7, \beta=0.5$.}
		\end{minipage}
		\begin{minipage}[t]{0.32\linewidth}
			\centering
			\includegraphics[width=\linewidth]{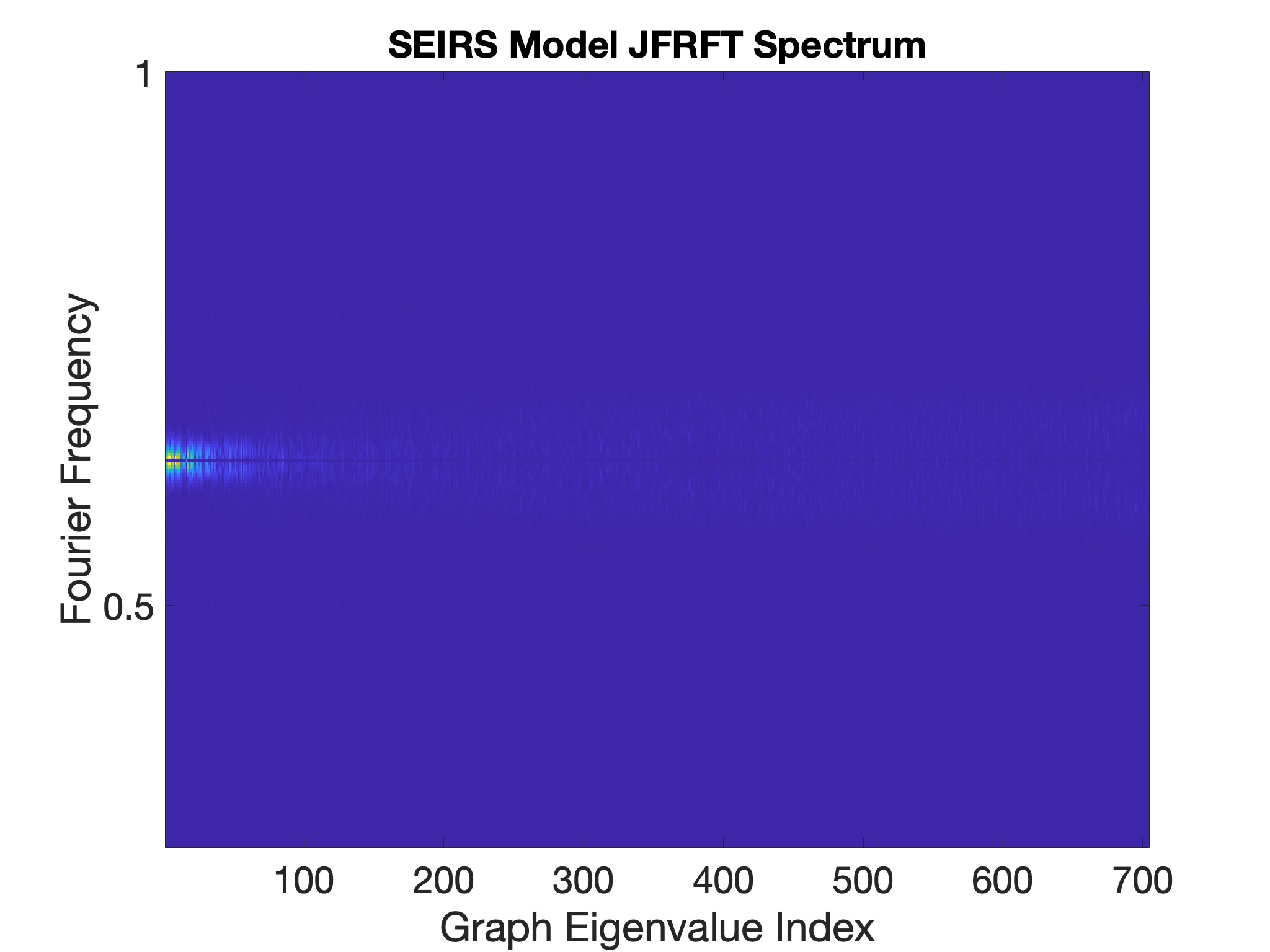}
			\parbox{2.6cm}{\tiny (d) The JFRFT of $\alpha=1, \beta=1$.}
		\end{minipage}
			\begin{minipage}[t]{0.32\linewidth}
			\centering
			\includegraphics[width=\linewidth]{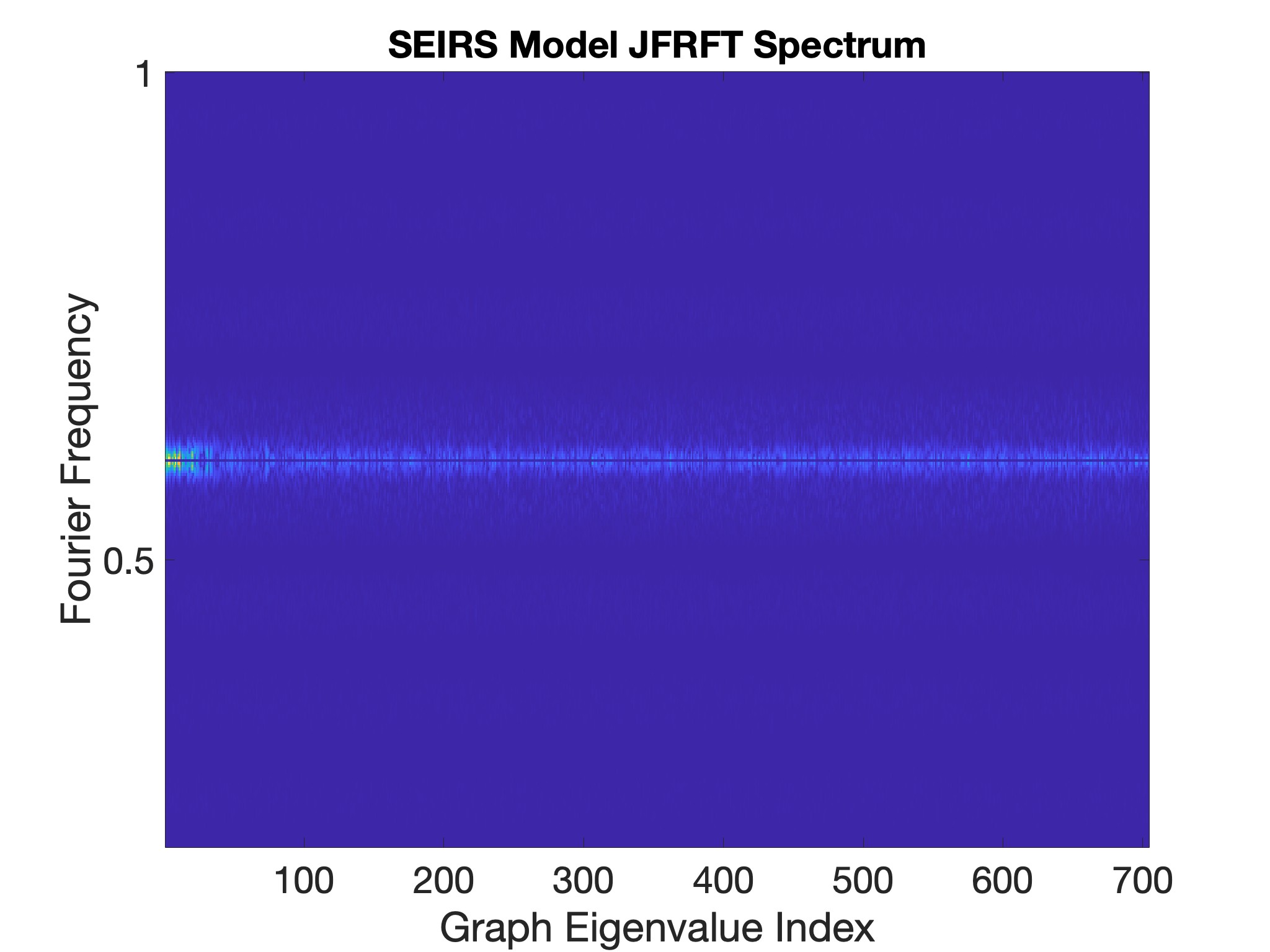}
			\parbox{2.6cm}{\tiny (e) The JFRFT of $\alpha=1, \beta=0.5$.}
		\end{minipage}
			\begin{minipage}[t]{0.32\linewidth}
			\centering
			\includegraphics[width=\linewidth]{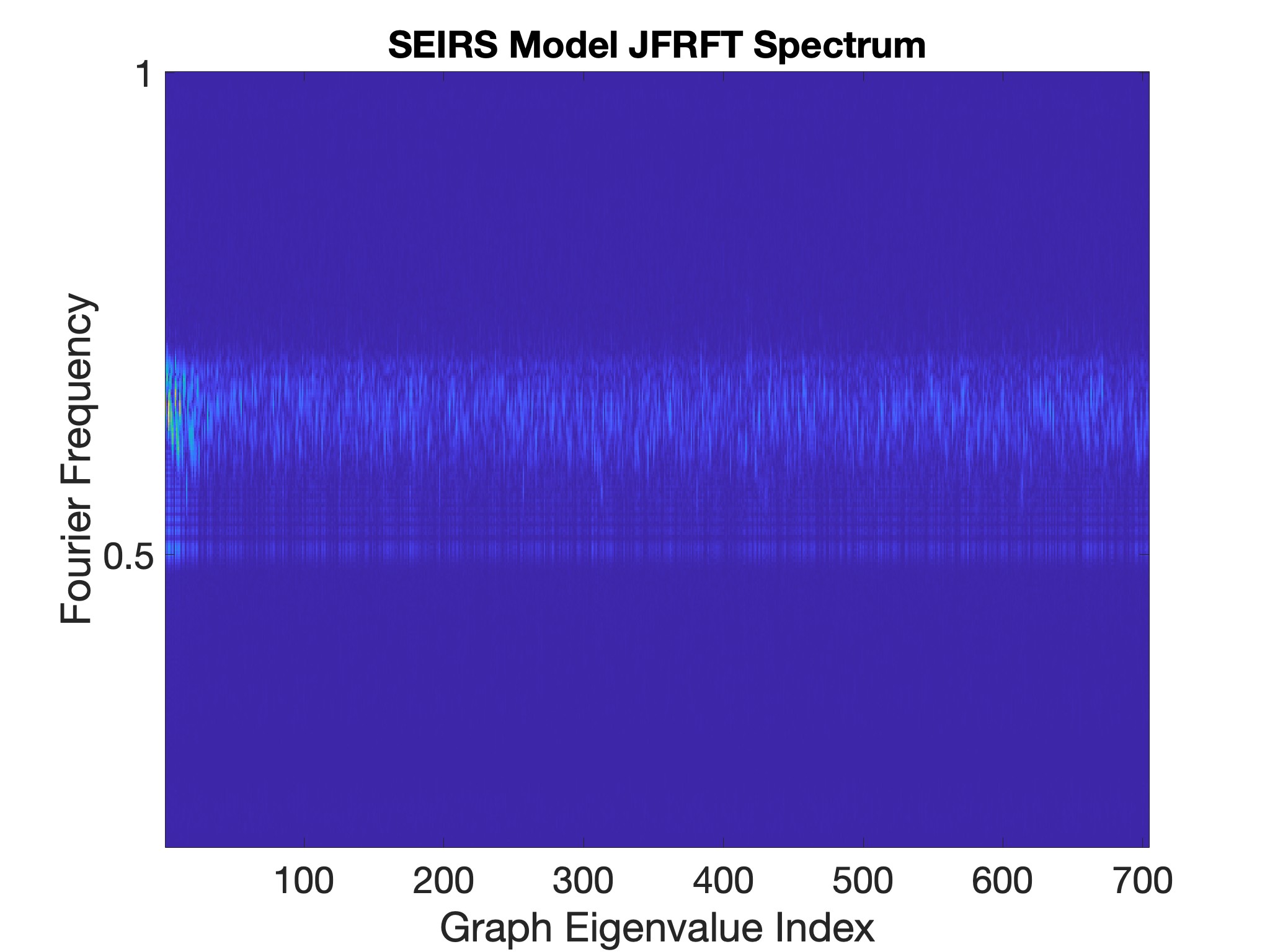}
			\parbox{2.6cm}{\tiny (f) The JFRFT of $\alpha=0.7, \beta=0.5$.}
		\end{minipage}
	\end{center}
	\caption{The HGFRFT and the JFRFT spectrum of the SEIRS model.}
	\vspace*{-3pt}
	\label{fig10}
\end{figure}

\subsection{Radar Signal Processing on Graphs}\label{sec6.3}
In radar signal processing, analyzing and handling bandlimited signals is crucial in various fields such as weather monitoring, seismic wave analysis, and communication systems. Traditional signal processing approaches rely on time and frequency domain analysis, but in high-dimensional data environments, radar signals are initially modeled in an infinite-dimensional Hilbert space \cite{Radar,Hilbertchirp}. This infinite-dimensional representation captures the continuity and theoretical completeness of signals.

However, practical limitations imposed by physical devices and computational demands prevent direct manipulation of infinite-dimensional signals. Thus, these signals are projected onto a finite-dimensional Hilbert space for processing \cite{HGFT,Hilbert}. In Hilbert space, a signal can be represented as an orthogonal basis expansion: $f(t) = \sum_{n=1}^{\infty} c_n \psi_n(t)$, where $\psi_n(t)$ is a set of orthonormal basis functions, and $c_n$ are the projection coefficients. According to the Nyquist sampling theorem \cite{Shannon}, the signal is mapped from the infinite-dimensional Hilbert space $\mathcal{H}$ to a finite-dimensional subspace $ \mathcal{H}_M$, represented as a vector of length $M$.

\subsubsection{Spectral Representation of Radar Signals on Graphs}
In this study, each radar signal is sampled at 200 points, producing a 200-dimensional vector. This discretization effectively maps the infinite-dimensional signal into a finite-dimensional vector, retaining its key time-frequency characteristics. By this process, we can analyze and manipulate the signals within a finite-dimensional Hilbert space.

Since Alaska and Hawaii are located too far from other states, we only model the radar signals on a graph $\mathcal{G}$ representing the geographical layout of the 48 contiguous U.S. states, with each node corresponding to one state \cite{USA}. The graph is constructed by connecting neighboring states with undirected edges of unit weight, indicating direct geographical adjacency. Each state’s radar signal is a time series (chirp signal), and the signal on each node is defined by
\[
x_i(t) = \exp\left(2j \pi \left(f_{0,i} t + 0.5 \mu_i t^2 \right)\right),
\]
where $f_{0,i}$ is the initial frequency of the $i$-th node and $\mu_i$ is the chirp rate. Each signal consists of 200 sample points. The initial frequency is set at $f_0 = 50$ Hz, with an initial bandwidth of $B = 150$ Hz and a time duration of 0.2 seconds. For each node, the initial frequency and bandwidth evolve as $f_{0,i} = f_0 + 5i$ and $B_i = B + 10i$, respectively.

The chirp signals at some vertices are transformed using Eq. \eqref{Halpha}. Figs. \ref{fig11}(a) and (b) show the spectrum of the signals at vertices 16 and 48, respectively. Figs. \ref{fig11}(c) and (d) display the spectrum of these signals after applying the $\mathcal{H}_{f}$-transform. Finally, Figs. \ref{fig11}(e) and (f) present the spectrum of the signals from (a) and (b) after applying the optimal $\mathcal{H}^{\alpha}_{f}$-transform, which effectively processes chirp signals by providing clearer distinctions in frequency variations and achieving higher frequency resolution.
\begin{figure}[h!]
	\begin{center}
		\begin{minipage}[t]{0.47\linewidth}
			\centering
			\includegraphics[width=\linewidth]{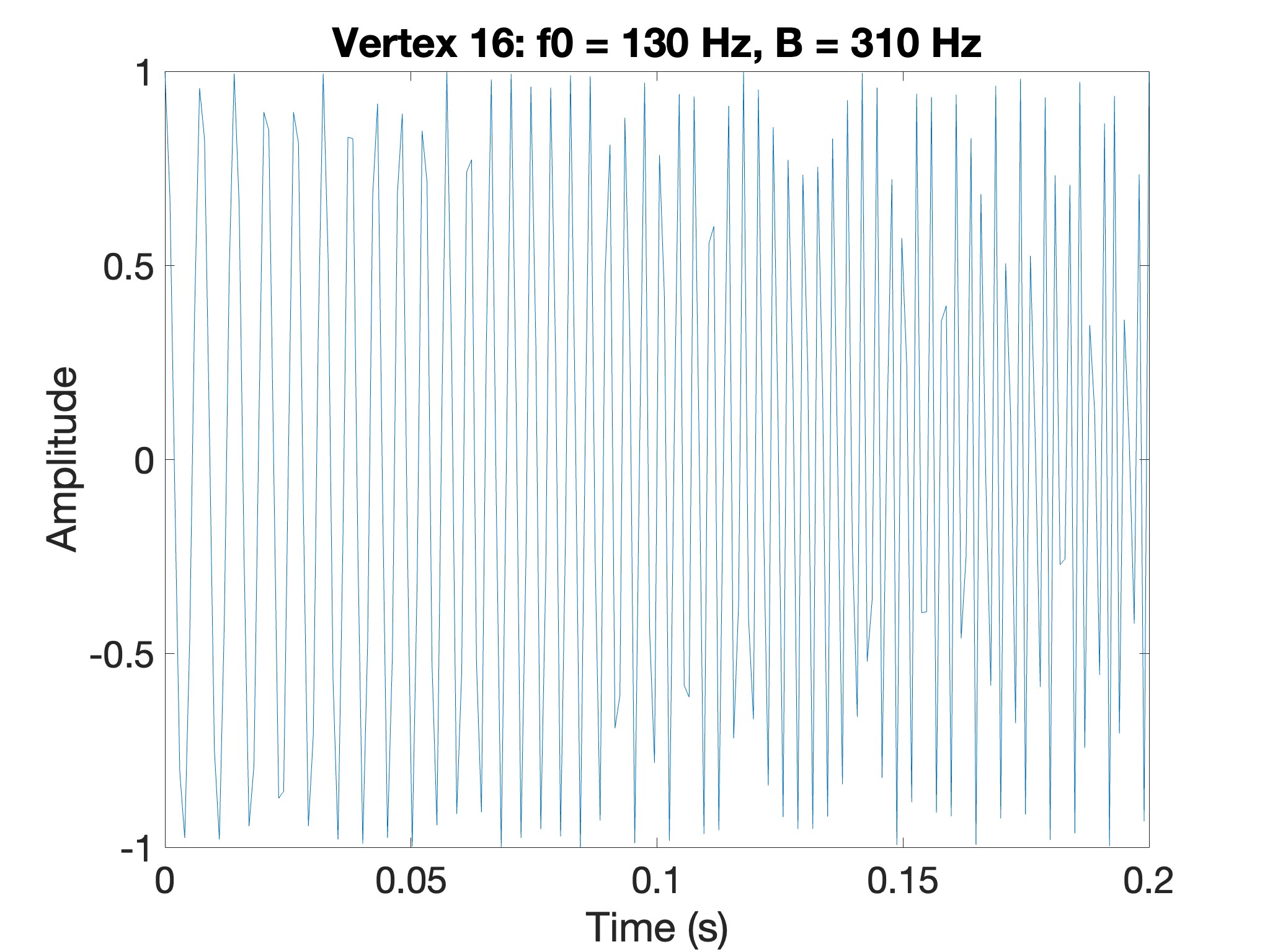}
			\parbox{4cm}{\tiny (a) Signal at vertex 16, $f_0=130$ Hz, $B=310$ Hz.}
		\end{minipage}
		\begin{minipage}[t]{0.47\linewidth}
			\centering
			\includegraphics[width=\linewidth]{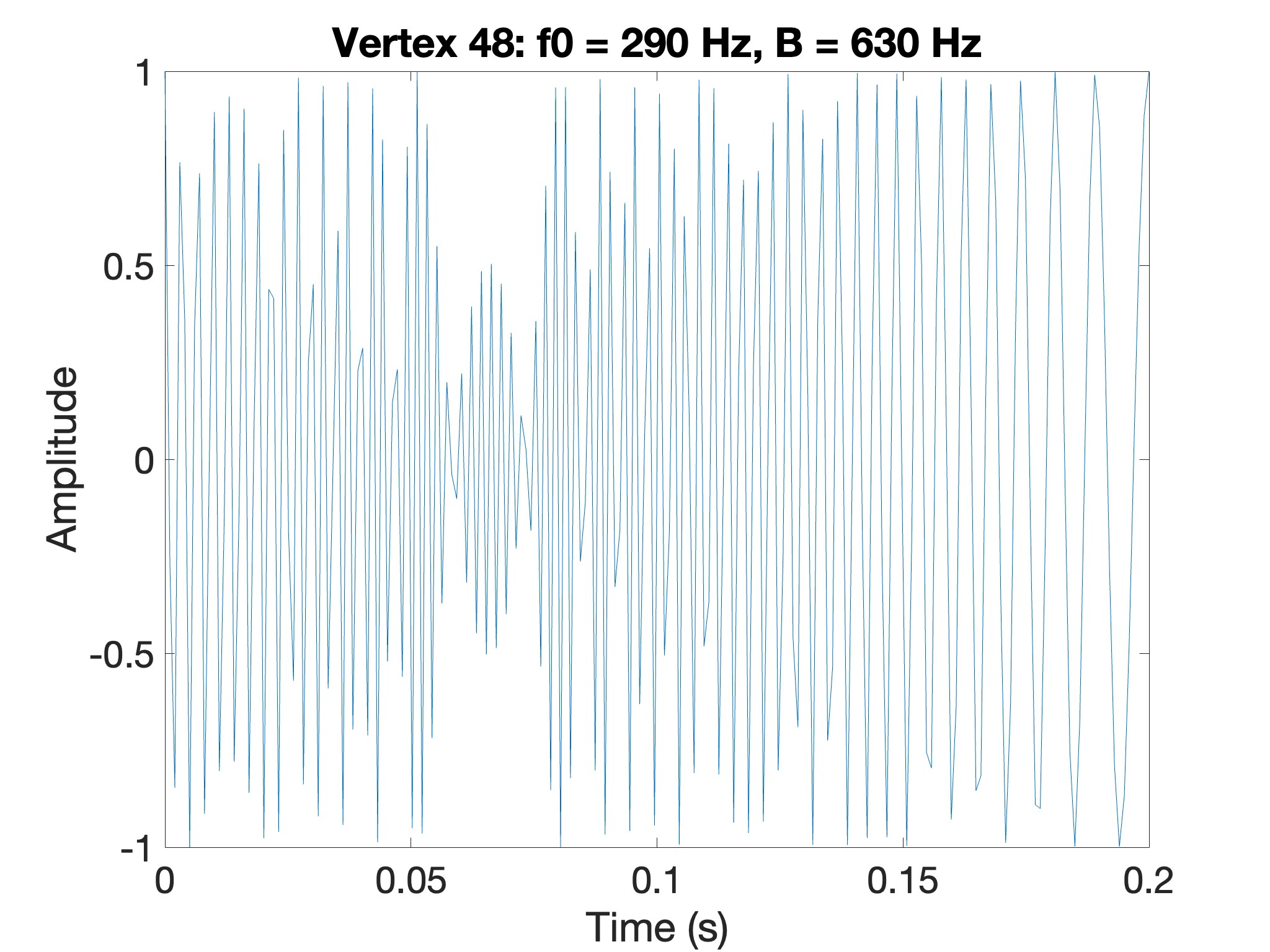}
			\parbox{4cm}{\tiny (b) Signal at vertex 48, $f_0=290$ Hz, $B=630$ Hz.}
		\end{minipage}
		\begin{minipage}[t]{0.47\linewidth}
			\centering
			\includegraphics[width=\linewidth]{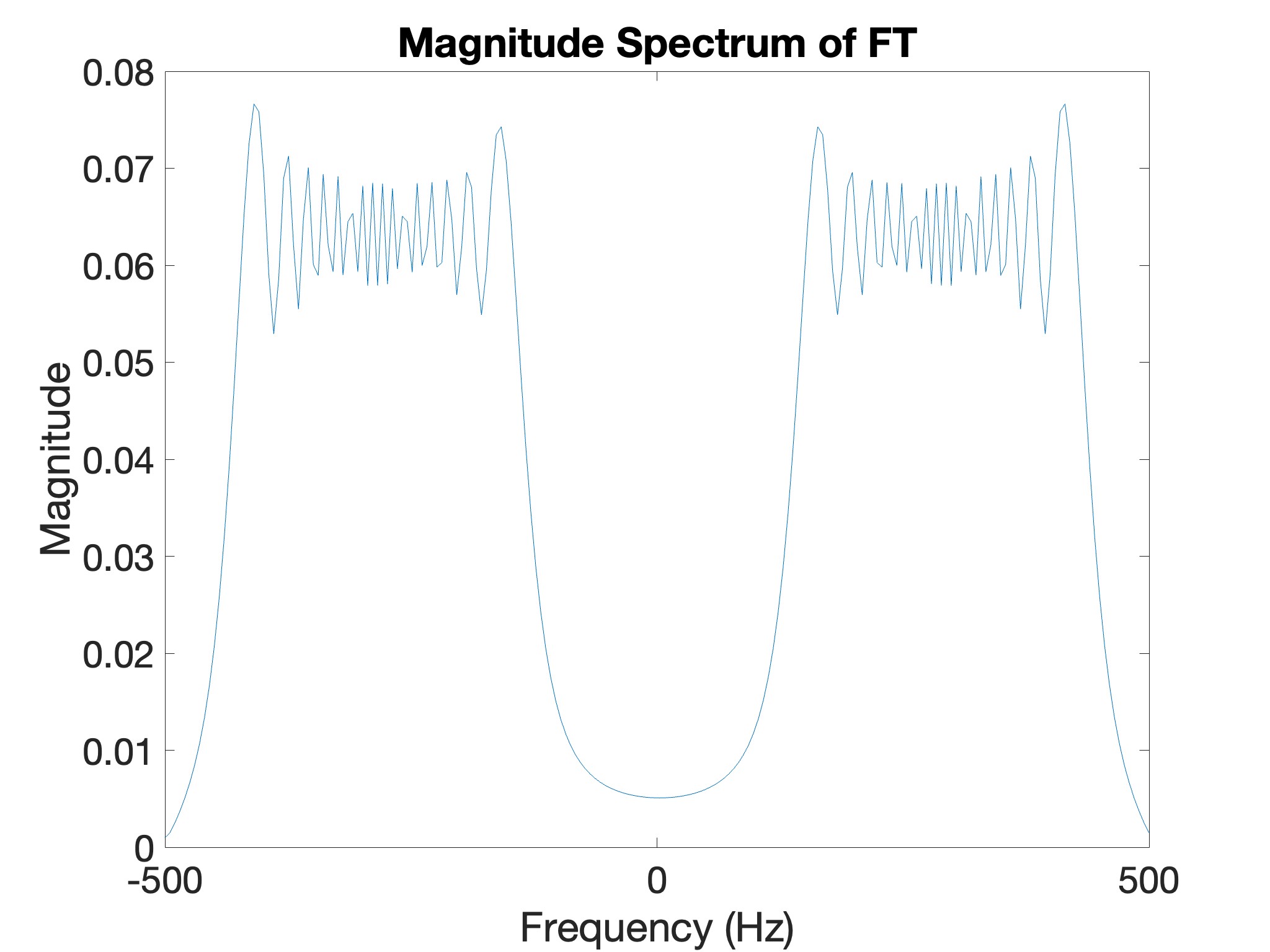}
			\parbox{4cm}{\tiny (c) Spectrum of the signal at vertex 16 after $\mathcal{H}_f$-transform.}
		\end{minipage}
		\begin{minipage}[t]{0.47\linewidth}
			\centering
			\includegraphics[width=\linewidth]{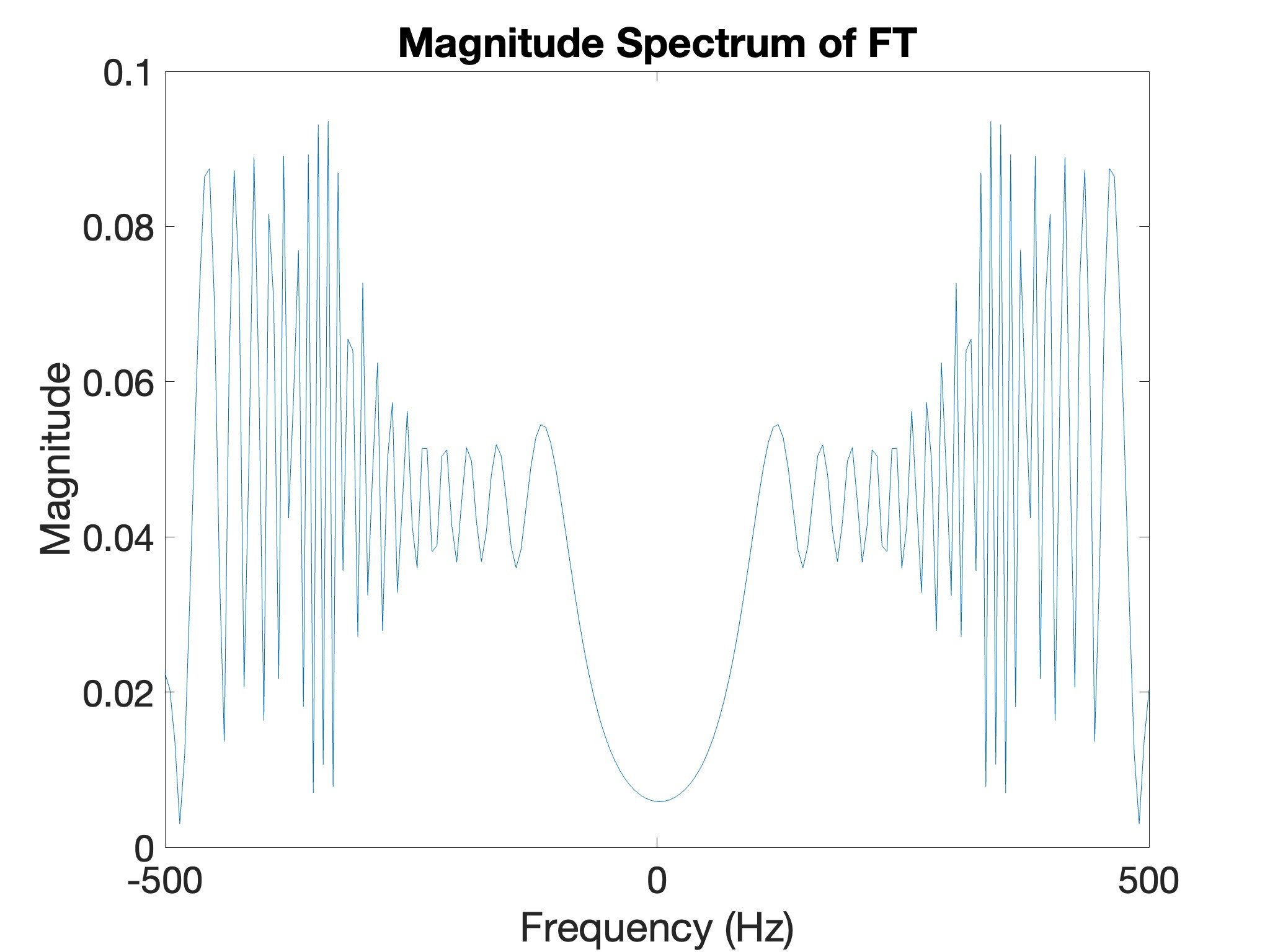}
			\parbox{4cm}{\tiny (d) Spectrum of the signal at vertex 48 after $\mathcal{H}_f$-transform.}
		\end{minipage}
		\begin{minipage}[t]{0.47\linewidth}
			\centering
			\includegraphics[width=\linewidth]{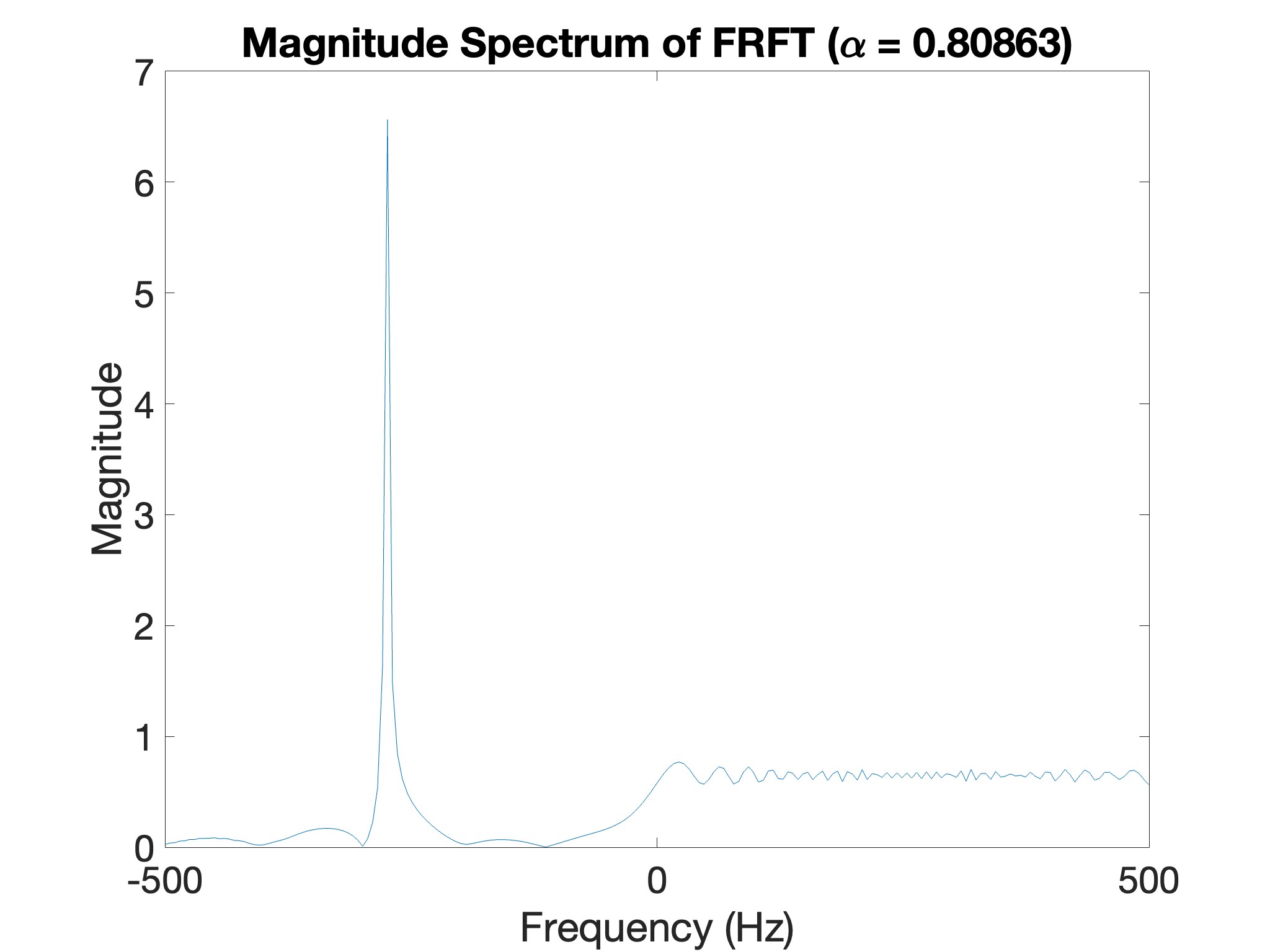}
			\parbox{4cm}{\tiny (e) Spectrum of the signal at vertex 16 after $\mathcal{H}^{\alpha}_f$-transform with $\alpha=0.81$.}
		\end{minipage}
		\begin{minipage}[t]{0.47\linewidth}
			\centering
			\includegraphics[width=\linewidth]{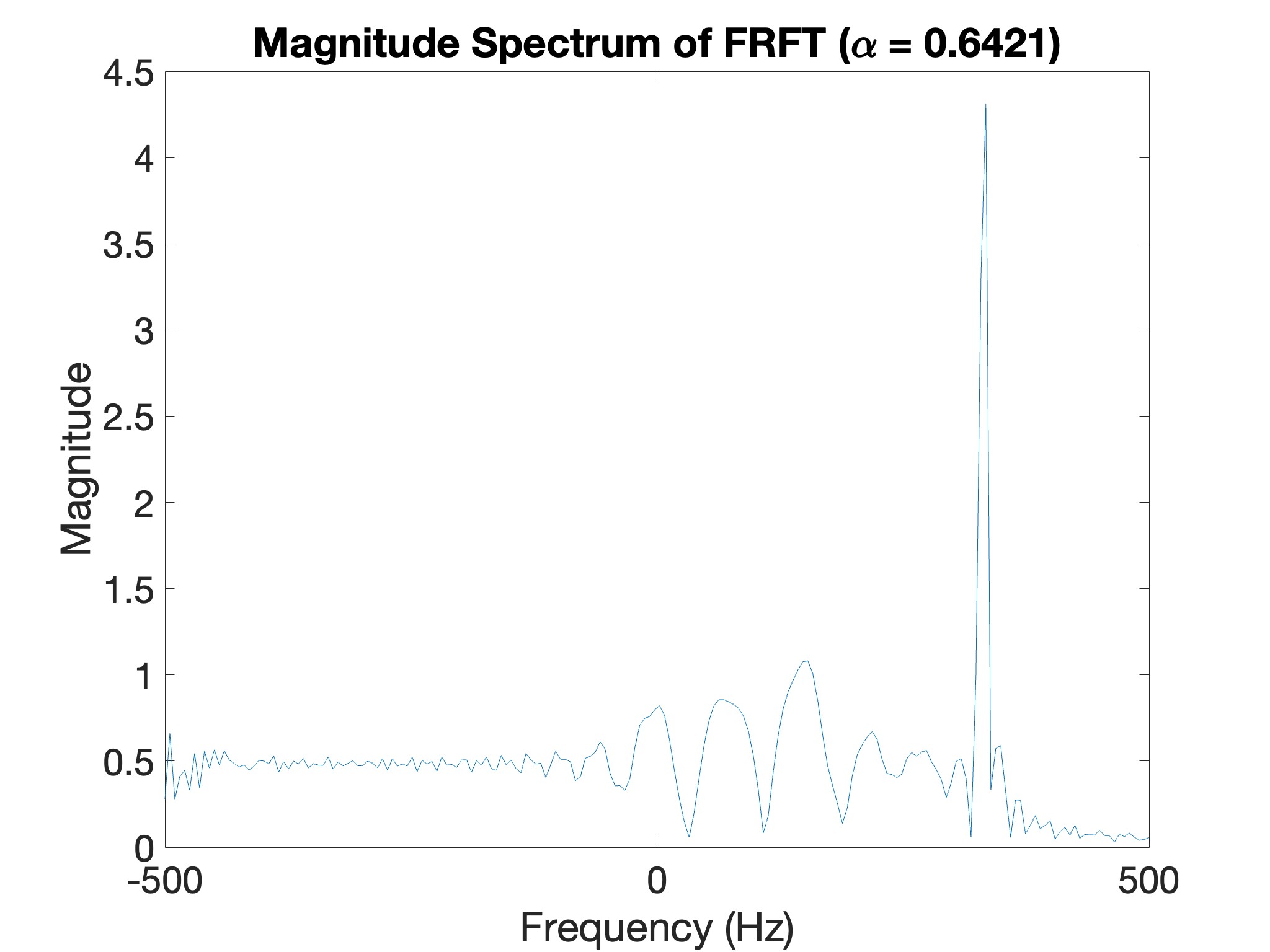}
			\parbox{4cm}{\tiny (f) Spectrum of the signal at vertex 48 after $\mathcal{H}^{\alpha}_f$-transform with $\alpha=0.64$.}
		\end{minipage}
	\end{center}
	\caption{Spectral representation of radar chirp signals at certain vertices.}
	\vspace*{-3pt}
	\label{fig11}
\end{figure}

Then, the graph signals at specific time points are transformed using Eq. \eqref{Gbeta}. Figs. \ref{fig12}(a) and (b) show the spectrum of the signals at times 1 and 20, respectively. Figs. \ref{fig12}(c) and (d) present the graph spectrum after applying the $\mathcal{G}_{f}$-transform. Finally, Figs. \ref{fig12}(e) and (f) display the graph spectrum of the signals in (a) and (b) after applying the $\mathcal{G}^{\beta}_{f}$-transform with $\beta = 0.5$. By adjusting the parameter $\beta$, we can obtain smoothness patterns from different perspectives.
\begin{figure}[h!]
	\begin{center}
		\begin{minipage}[t]{0.47\linewidth}
			\centering
			\includegraphics[width=\linewidth]{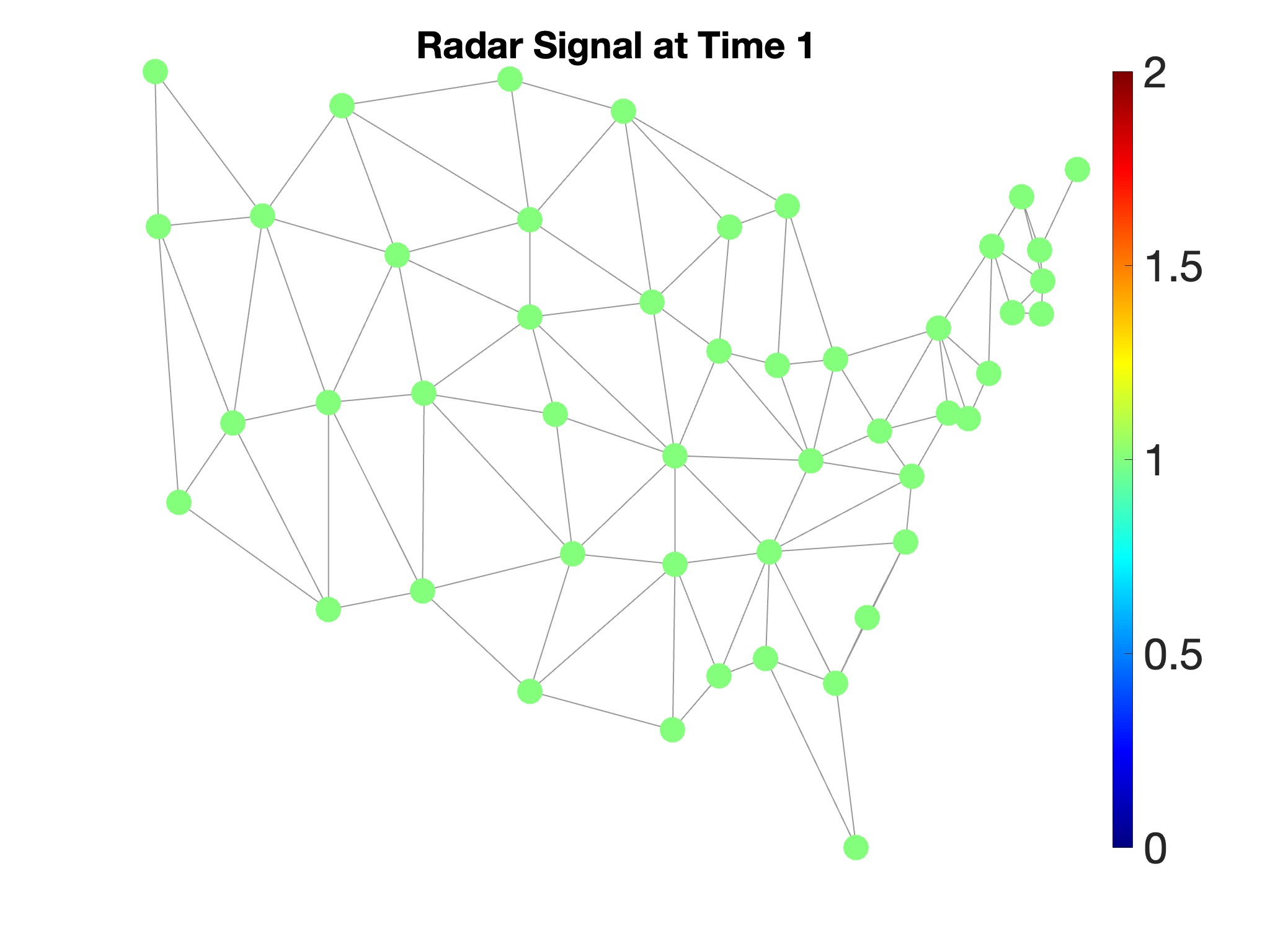}
			\parbox{2cm}{\tiny (a) Graph signal at time 1.}
		\end{minipage}
		\begin{minipage}[t]{0.47\linewidth}
			\centering
			\includegraphics[width=\linewidth]{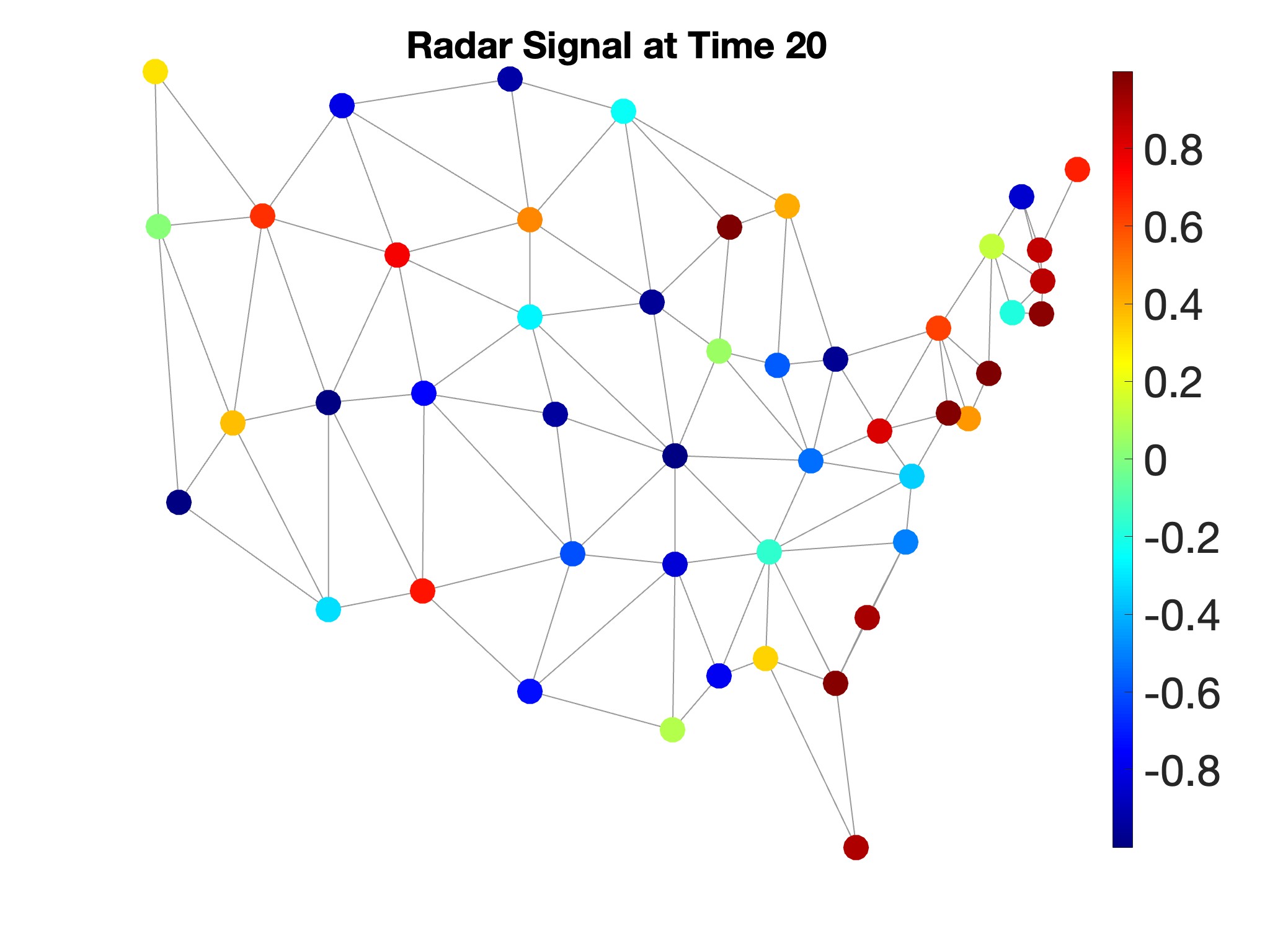}
			\parbox{2cm}{\tiny (b) Graph signal at time 20.}
		\end{minipage}
			\begin{minipage}[t]{0.47\linewidth}
			\centering
			\includegraphics[width=\linewidth]{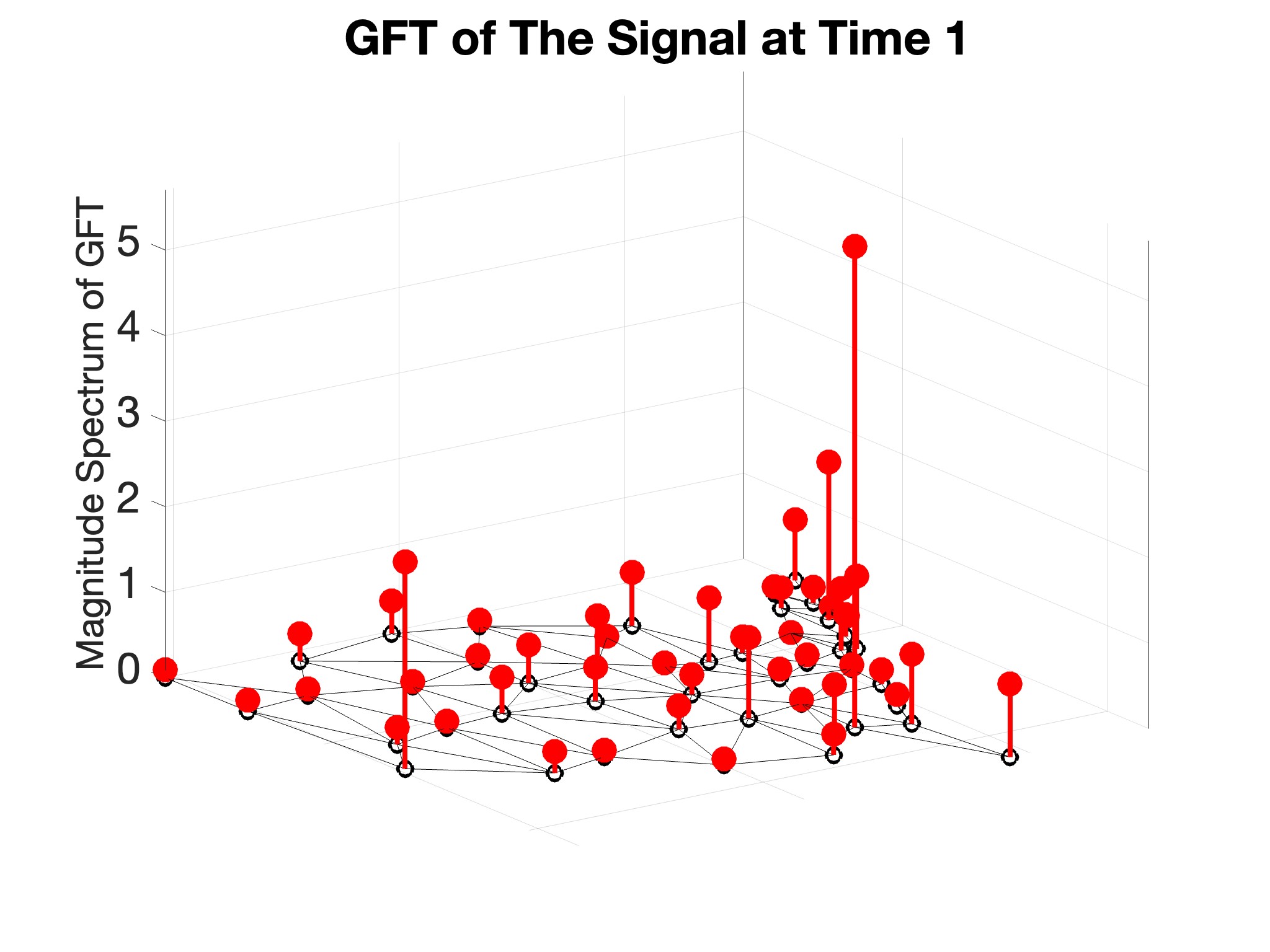}
			\parbox{4cm}{\tiny (d) Graph spectrum of the signal at time 1 after $\mathcal{G}_f$-transform.}
		\end{minipage}	
		\begin{minipage}[t]{0.47\linewidth}
		\centering
		\includegraphics[width=\linewidth]{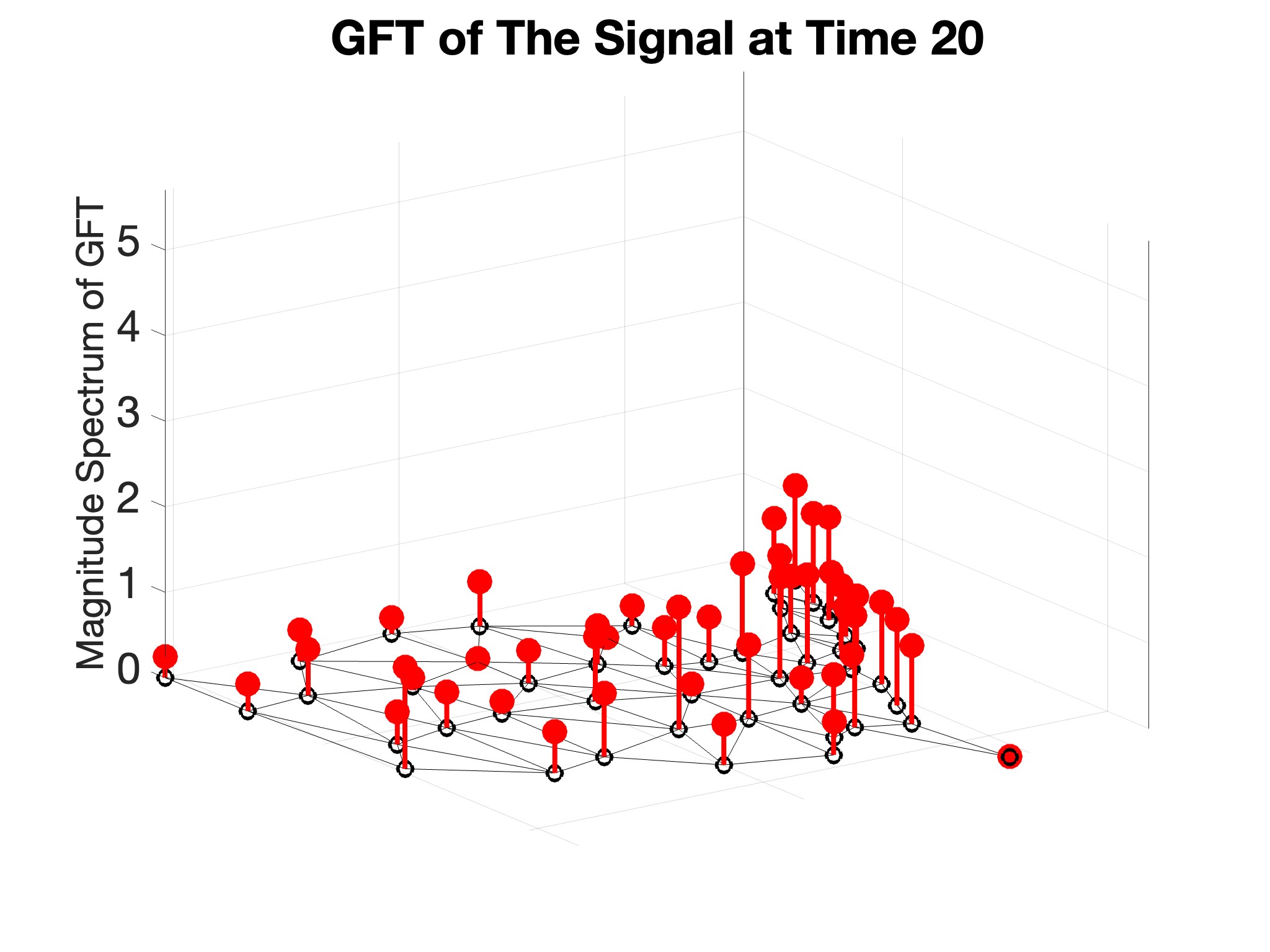}
		\parbox{4cm}{\tiny (e) Graph spectrum of the signal at time 20 after $\mathcal{G}_f$-transform.}
		\end{minipage}	
		\begin{minipage}[t]{0.47\linewidth}
		\centering
		\includegraphics[width=\linewidth]{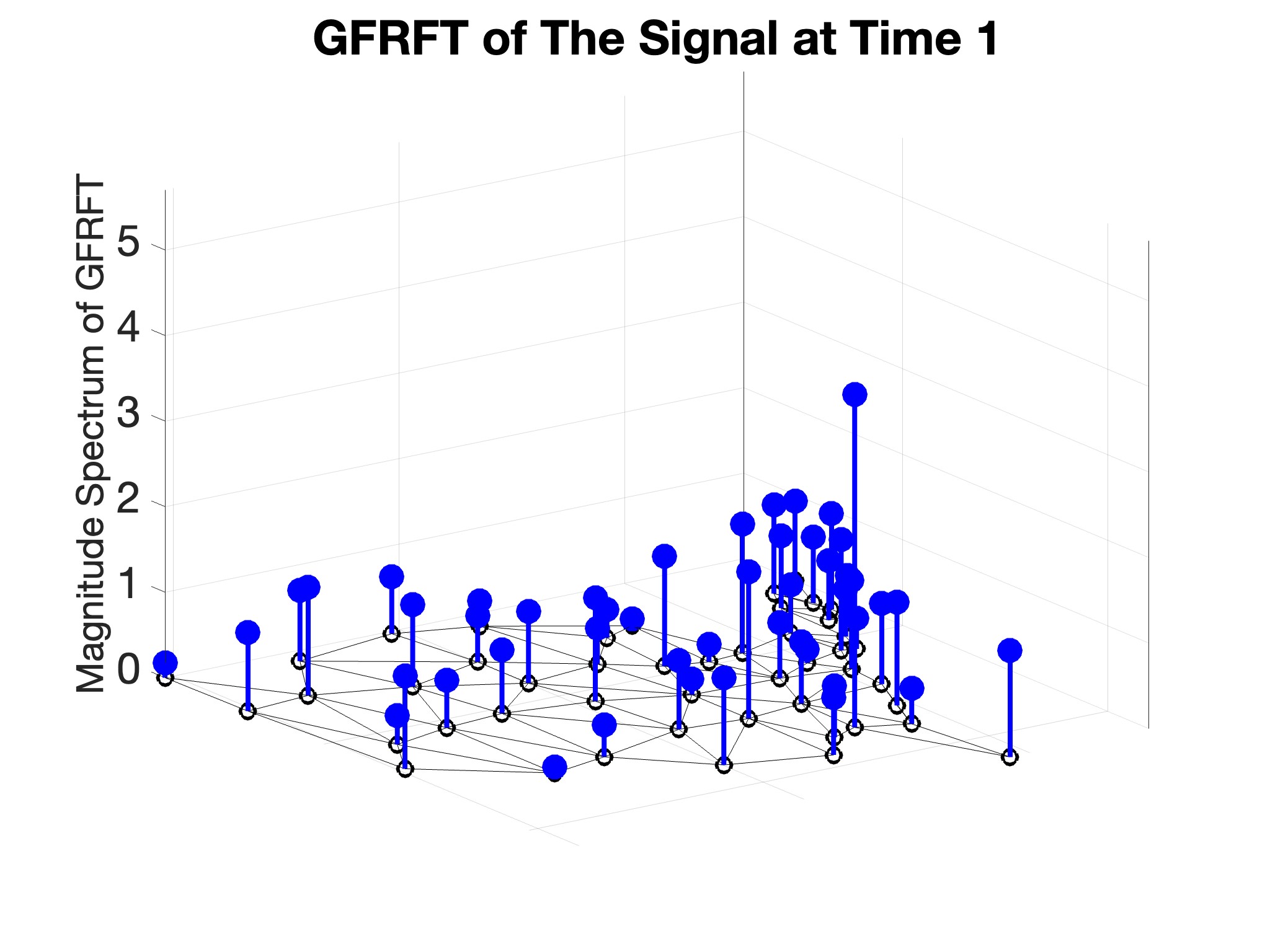}
		\parbox{4cm}{\tiny (g) Graph spectrum of the signal at time 1 after $\mathcal{G}^{\beta}_f$-transform with $\beta=0.5$.}
		\end{minipage}	
		\begin{minipage}[t]{0.47\linewidth}
		\centering
		\includegraphics[width=\linewidth]{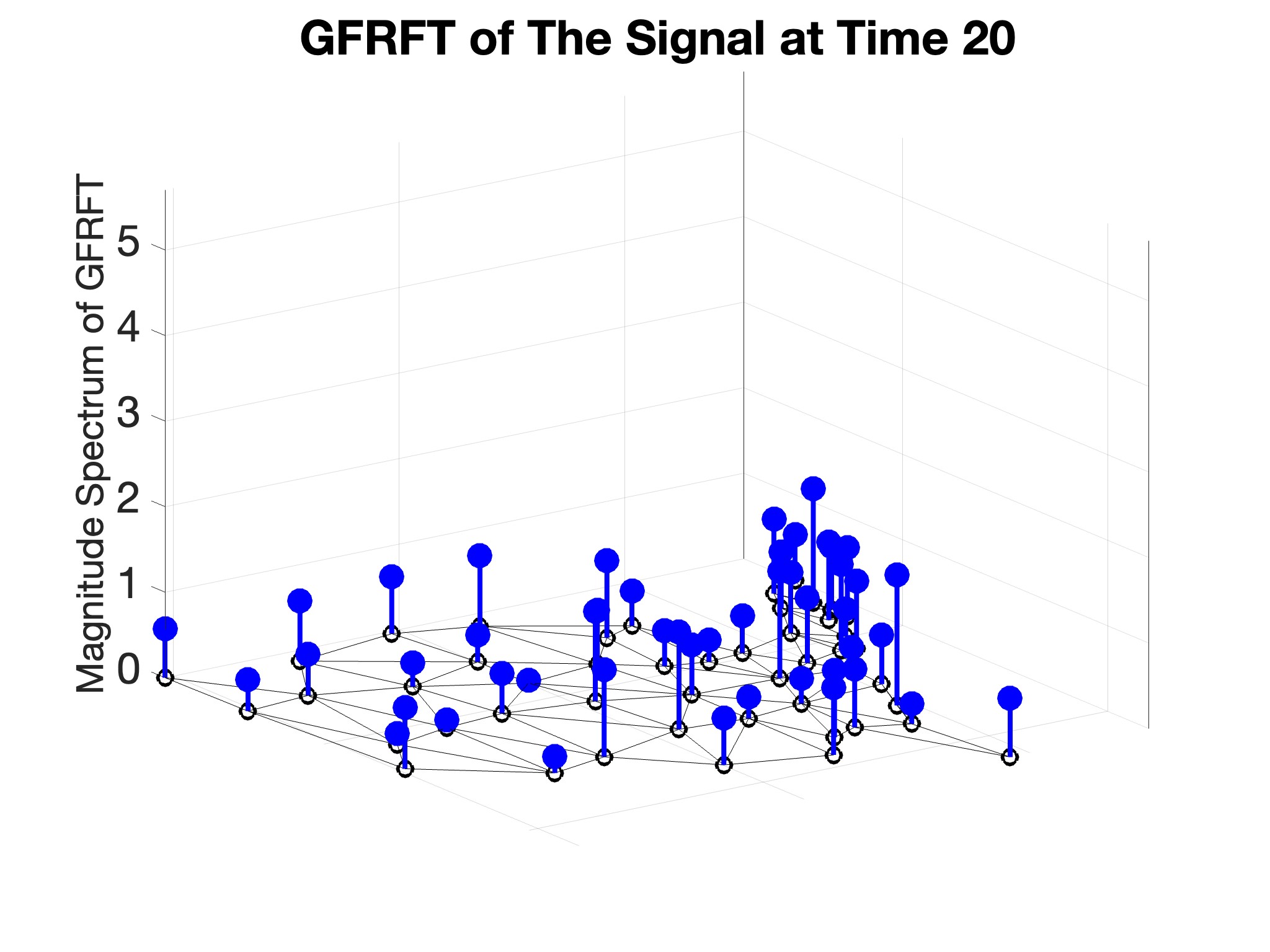}
		\parbox{4cm}{\tiny (h) Graph spectrum of the signal at time 20 after $\mathcal{G}^{\beta}_f$-transform with $\beta=0.5$.}
		\end{minipage}	
	\end{center}
	\caption{Spectral representation of graph signals at specific time steps for the 48 contiguous U.S. states.}
	\vspace*{-3pt}
	\label{fig12}
\end{figure}

Additionally, Fig. \ref{fig13} shows the Hilbert space–vertex signal representation. Panel (a) shows the spectral representation of the chirp signals on the graph, while panel (b) presents the transformed spectrum obtained via the HGFRFT with parameters $(0.1, 1)$. It is evident that the spectrum becomes more concentrated after transformation.
\begin{figure}[h!]
	\begin{center}
		\begin{minipage}[t]{0.47\linewidth}
			\centering
			\includegraphics[width=\linewidth]{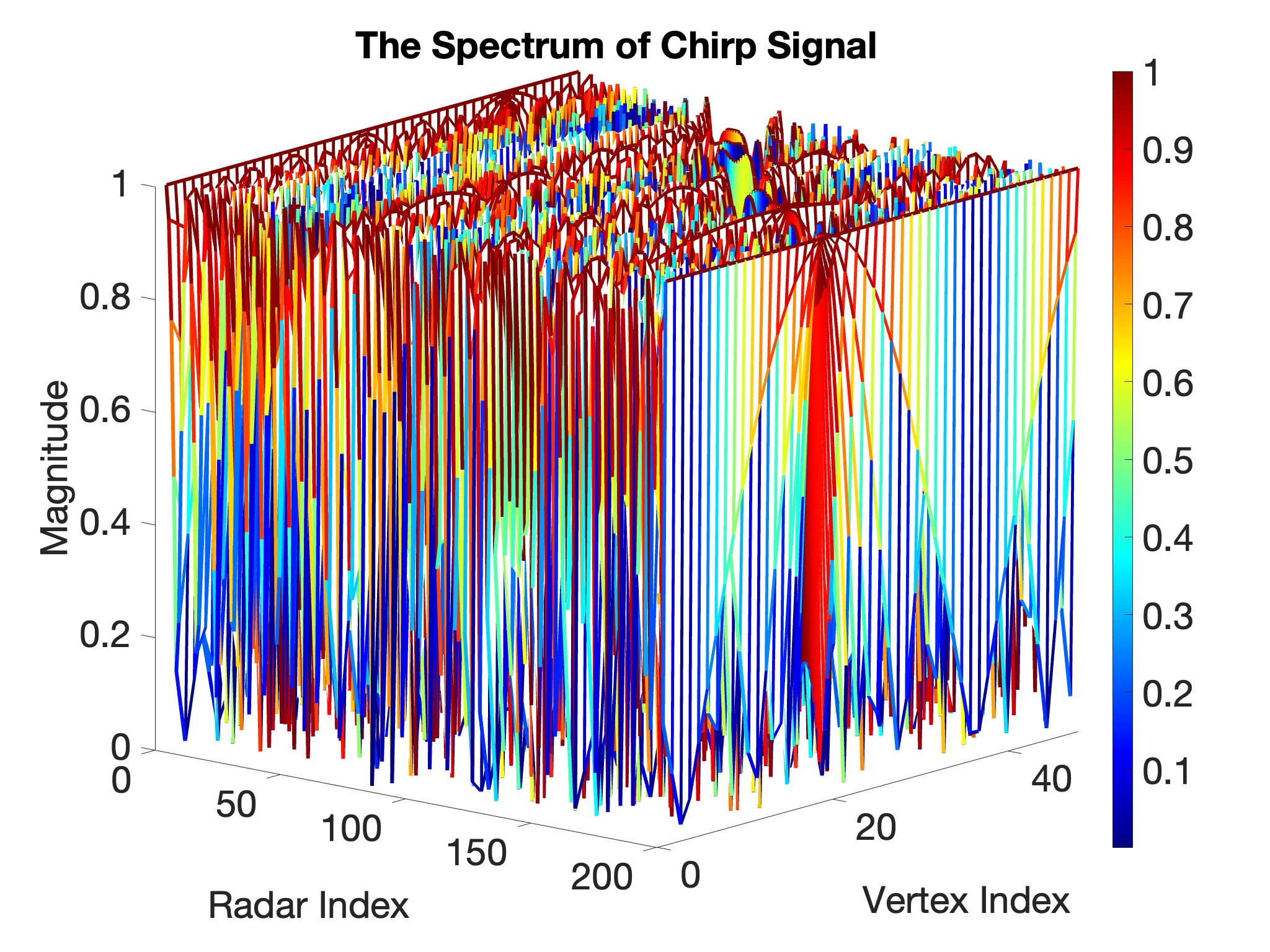}
			\parbox{4cm}{\tiny (a) Spectrum of the chirp signals on the graph.}
		\end{minipage}
		\begin{minipage}[t]{0.47\linewidth}
			\centering
			\includegraphics[width=\linewidth]{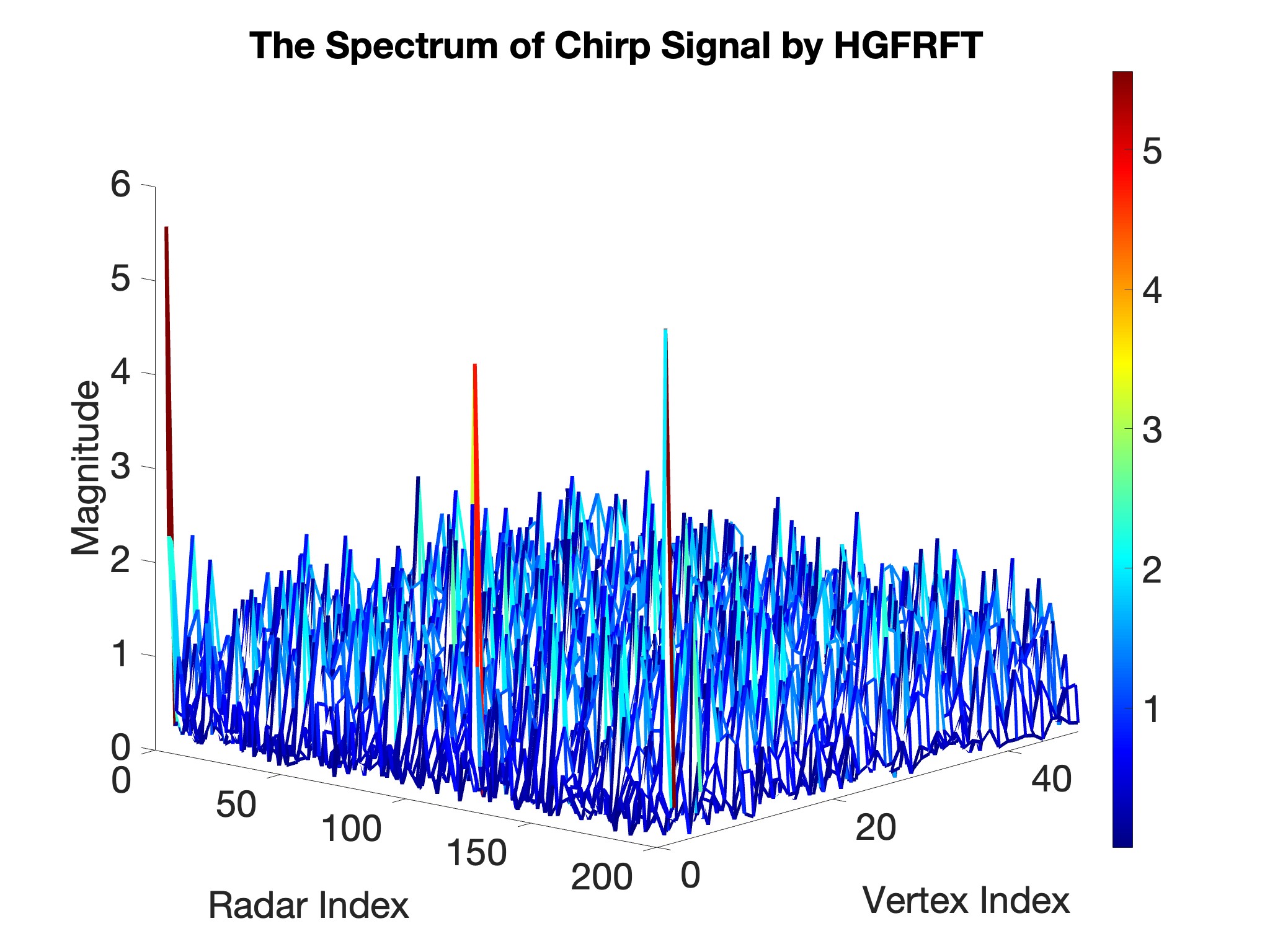}
			\parbox{4cm}{\tiny (b) Graph spectrum of the signals after HGFRFT with $\alpha=0.1$, $\beta=1$.}
		\end{minipage}
	\end{center}
	\caption{Spectral representation of the Hilbert space-vertex signal.}
	\vspace*{-3pt}
	\label{fig13}
\end{figure}

\subsubsection{Sampling and Reconstruction of Radar Signals on Graphs}\label{sec6.3.2}
In this study, we reconstruct the entire signal using fewer sampling nodes. Specifically, we employed chirp signals along with temperature, precipitation, and sunshine data\footnote{\url{https://www.currentresults.com/Weather/US/weather-averages-index.php}} collected from U.S. states as radar signals. By applying Lemma \ref{lem2}, the signals are bandlimited to ensure they lie within $\text{span}(\psi \otimes \phi)$. The HGFRFT is then applied to these bandlimited signals. Fig. \ref{fig14} presents the transformed spectra: (a) shows the HGFRFT spectrum of the bandlimited chirp signal, while (b) is a representative of real data, showing the HGFRFT spectrum of temperature data for the four seasons across the 48 contiguous U.S. states.
\begin{figure}[h!]
	\begin{center}
		\begin{minipage}[t]{0.47\linewidth}
			\centering
			\includegraphics[width=\linewidth]{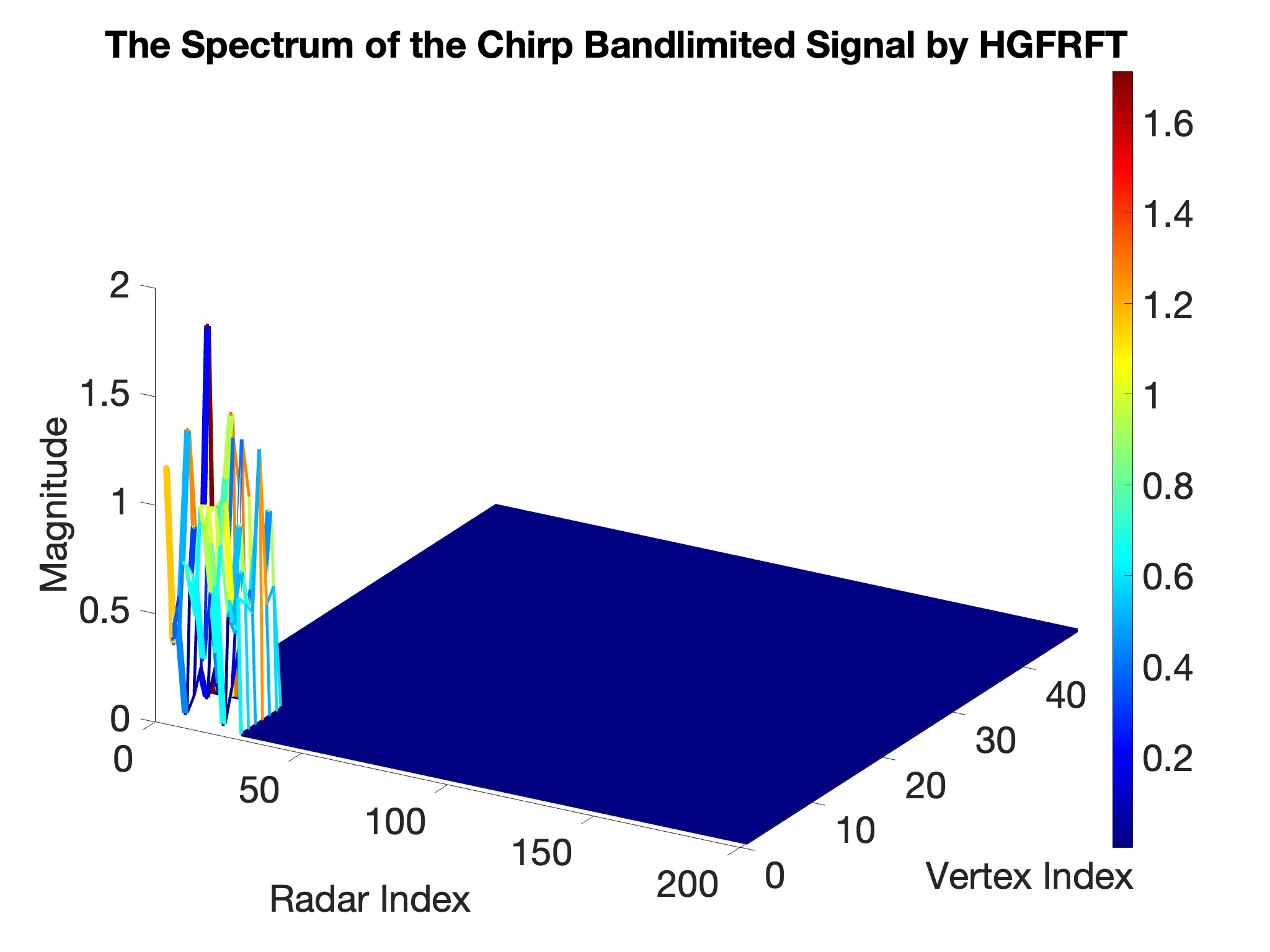}
			\parbox{4cm}{\tiny (a) Graph spectrum of the chirp bandlimited signals after HGFRFT.}
		\end{minipage}
		\begin{minipage}[t]{0.47\linewidth}
			\centering
			\includegraphics[width=\linewidth]{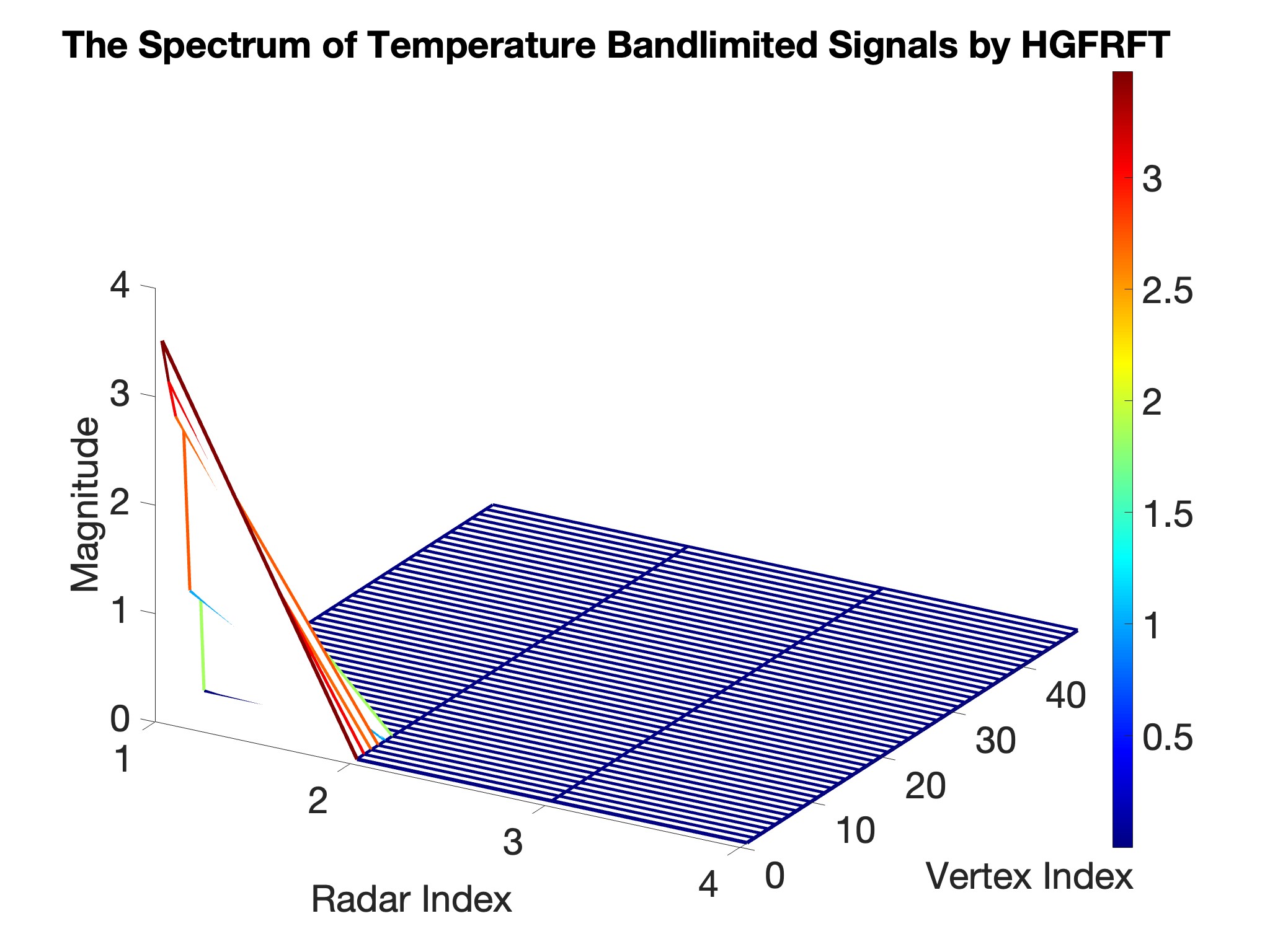}
			\parbox{4cm}{\tiny (b) Graph spectrum of the temperature bandlimited signals after HGFRFT.}
		\end{minipage}
%		\begin{minipage}[t]{0.47\linewidth}
%			\centering
%			\includegraphics[width=\linewidth]{HGFRFT_precipitation.jpg}
%			\parbox{4cm}{\tiny (d) Graph spectrum of the precipitation bandlimited signals after HGFRFT.}
%		\end{minipage}	
%		\begin{minipage}[t]{0.47\linewidth}
%			\centering
%			\includegraphics[width=\linewidth]{HGFRFT_sunshine.jpg}
%			\parbox{4cm}{\tiny (e) Graph spectrum of the sunshine bandlimited signals after HGFRFT.}
%		\end{minipage}	
	\end{center}
	\caption{Spectrum of radar bandlimited signal after HGFRFT.}
	\vspace*{-3pt}
	\label{fig14}
\end{figure}

We set the number of sampling nodes equal to the bandwidth, which is $156$ for the chirp signal and $6$ for others. Using Algorithm \ref{alg1}, the optimal sampling set is determined, and the signal is reconstructed based on Theorem \ref{thm3}. The error is computed using Eq. \eqref{error}. Table \ref{tab2} presents the optimal parameters $(\alpha, \beta)$ and the corresponding minimum reconstruction errors for different radar signals. Although the reconstruction error for HGFT is already very small, optimized HGFRFT parameters still offer advantages, even when the signal lies within the HGFT bandlimited ($f \in \text{span}(\psi \otimes \phi)$). While the improvement may be limited, the optimized HGFRFT leads to better recovery, as shown in Table \ref{tab2}. More importantly, when the signal is within the HGFRFT bandlimited, the gap between HGFRFT and HGFT becomes more significant (cf. Section \ref{sec6.1}). The introduction of HGFRFT provides greater flexibility, where even minor error improvements can have a significant impact in practical applications, particularly for different signal types and noisy environments.
\begin{table*}[t]
	\centering
	\caption{Optimal Parameters and Minimum Reconstruction Error for Different Radar Signals}\label{tab2}
		\begin{tabular}{lcccc}
			\toprule
			Radar Signals & HGFT $(\alpha, \beta)$&HGFT Reconstruction Error & HGFRFT Optimal $(\alpha, \beta)$ & HGFRFT Reconstruction Error  \\
			\midrule
			Chirp Signals     & $(1, 1)$ & $2.3171 \times 10^{-12}$ & $(2.35, 1.01)$ & $1.8780 \times 10^{-12}$ \\
			Temperature Signals & $(1, 1)$ & $7.6174 \times 10^{-14}$ & $(1.32, -0.50)$ & $5.9533 \times 10^{-14}$ \\
			Precipitation Signals  & $(1, 1)$ & $2.2661 \times 10^{-13}$ & $(0.67, 1.81)$ & $1.4334 \times 10^{-13}$\\
			Sunshine Signals   & $(1, 1)$ & $7.6416 \times 10^{-14}$ & $(0.77, 1.91)$ & $4.9694 \times 10^{-14}$ \\
			\bottomrule
	\end{tabular}
\end{table*}

\subsection{Digital Image Denoising}\label{sec6.4}
In signal processing theory, signals are typically modeled as functions within an infinite-dimensional Hilbert space. As demonstrated in Example \ref{examp4} of Section \ref{sec3.1}, this framework provides an infinite-dimensional representation of signals in the complex domain, enabling precise and continuous analysis and transformation.

To apply these theoretical concepts to practical computations, we first approximate the continuous signal space using an $N \times N$ grid. Based on the promising experimental results in \cite{HGFT}, we employ Chebyshev polynomials in instead of the DFRFT to approximate the basis of continuous functions within the infinite-dimensional Hilbert space $L^{2}([-1,1])$. Specifically, we select 100th-order Chebyshev polynomials within the interval $[-1, 1]$, leveraging their orthogonality. These polynomials are organized in a matrix, with each row representing a polynomial of a specific order.

For signal construction, we utilize two digit images (e.g., “0” and “6”) and define an interpolation function to generate signals transitioning between these images \cite{HGFT}. We then apply the HGFRFT to project these signals onto graph fractional Laplacian eigenvectors \cite{GFRFTspectral} and fractional Chebyshev polynomials. After introducing Gaussian noise $\mathcal{N}(0, 25^2)$ to the images, we apply HGFRFT, combining joint fractional Chebyshev polynomials with GFRFT, for filtering and signal reconstruction. By analyzing the weighted coefficients and eigenvector contributions during reconstruction, we achieve the restored signals. As illustrated in Fig. \ref{fig15}, slight adjustments to the parameters $(\alpha, \beta)$ can enhance denoising effectiveness. Moreover, when $\alpha$ takes a negative value, the image undergoes both color inversion and digit reversal.
\begin{figure}[h!]
	\begin{center}
		\begin{minipage}[t]{0.9\linewidth}
			\centering
			\includegraphics[width=\linewidth]{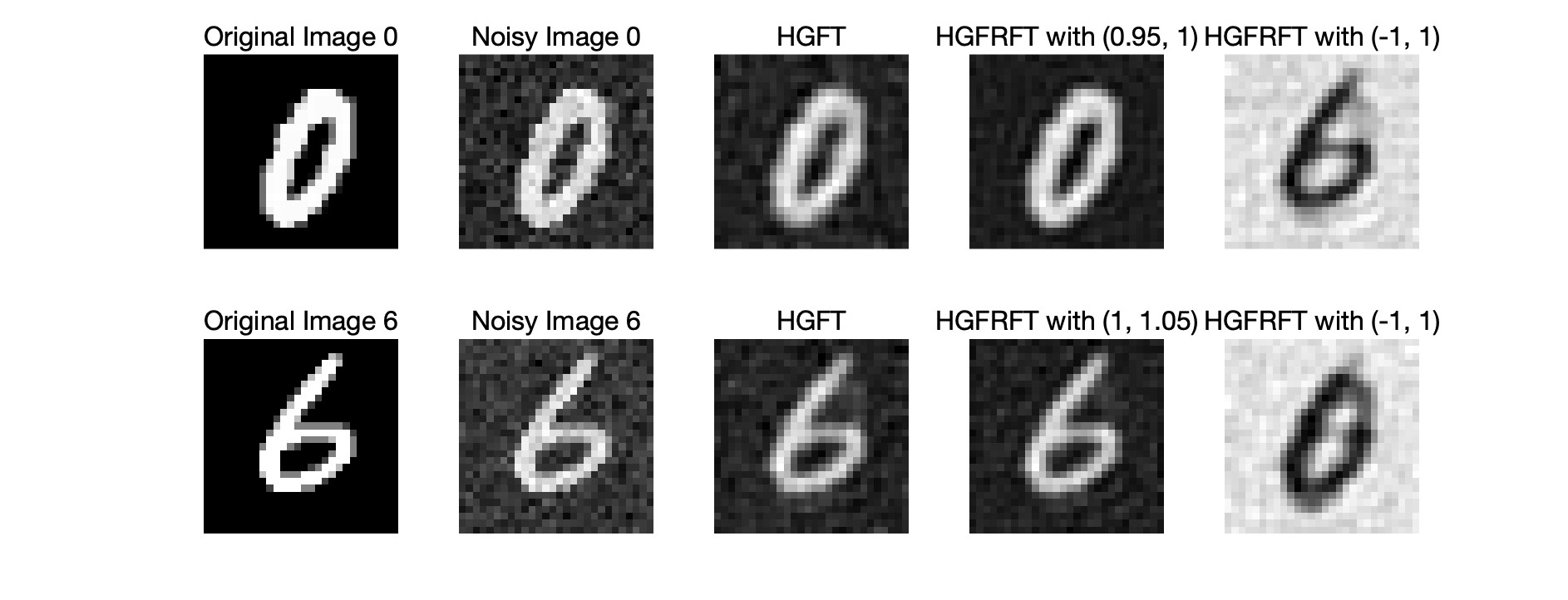}
		\end{minipage}
	\end{center}
	\caption{Digital images 0 and 6.}
	\label{fig15}
\end{figure}

By mapping signals from infinite-dimensional Hilbert space to an approximate continuous domain and incorporating fractional parameters $\alpha$ and $\beta$, we establish a versatile signal transformation framework. HGFRFT provides a multifaceted tool for signal analysis, while the denoising and reconstruction steps significantly enhance signal recovery quality and efficacy.

\section{Conclusion}
This paper introduces the graph fractional Fourier transform in Hilbert space (HGFRFT), a generalization of the Fourier transform in generalized graph signal processing, akin to how the fractional Fourier transform extends the Fourier transform. Key properties, filtering, and sampling theories are developed, enabling the reduction of signals from infinite continuous domains to finite ones. A product graph example highlights the advantages of the HGFRFT. The framework is not only mathematically elegant and practical but has also demonstrated its effectiveness through applications to real-world datasets and various scenarios, highlighting its utility and versatility.

\appendices
\section{Uniqueness of Fractional Powers}
\label{appendixA}
To ensure the uniqueness of fractional powers, for a complex number $z$, the principal logarithm is defined as $\text{Log}(z) = \log(r) + i \theta$, where $\theta \in (-\pi, \pi]$. The fractional power $z^\beta$ is then computed as $e^{\beta \text{Log}(z)}$. For a matrix $ \mathbf{\Lambda}$, we compute the principal logarithm $\text{Log}(\mathbf{\Lambda})$ and apply the fractional power
\[\mathbf{\Lambda}^\beta = e^{\beta \text{Log}(\mathbf{\Lambda})}.\]
To ensure the principal logarithm is properly calculated, for negative real eigenvalues $\lambda_j$, the principal logarithm is given by
\[\text{Log}(\lambda_j) = \log(|\lambda_j|) + i \text{arg}(\lambda_j),\]
where $\text{arg}(\lambda_j)$ is the principal argument of the negative real number, constrained to $(-\pi, \pi]$ \cite{Matrices}.

\section{Spectral Decomposition of Compact, Self-Adjoint Operators}
\label{appendixB}
Given that $\mathbf{B}$ is compact, self-adjoint, and injective, the spectral theorem ensures that $\mathbf{B}$ has a countable set of real eigenvalues $\{\lambda_i\}_{i \geq 1}$ and corresponding orthonormal eigenvectors $\Psi=\{ \psi_i \}_{i \geq 1}$, forming a complete basis for the Hilbert space \(\mathcal{H}\) \cite{Functional,Hilbert}. Consequently, \(\mathbf{B}\) can be expressed as
\[
\mathbf{B} = \sum_{i \geq 1} \lambda_i \psi_i \psi_i^{\ast},
\]
where $\psi_i$ are the orthonormal basis vectors corresponding to the eigenspaces of $\mathbf{B}$.\\
In an infinite-dimensional Hilbert space, the orthonormal basis $\Psi$ satisfies orthogonality and normalization, specifically \(\langle \psi_i, \psi_j \rangle = 0\) for \(i \neq j\) and \(\|\psi_i\| = 1\) for all \(i\). Based on this orthonormal basis, we can define the fractional power $\Psi^{\alpha}$ as
\[
\Psi^{\alpha} := \sum_{i \geq 1} \kappa _i^{\alpha}\xi_i  \xi_{i}^{\ast},
\]
where $\kappa_i$ represents the eigenvalues associated with the basis vectors $\Xi=\{ \xi_i \}_{i \geq 1}$. This expression preserves the integrity of the spectral decomposition while enabling effective application of fractional operators in the infinite-dimensional Hilbert space.

\section*{Acknowledgments}
This work were supported by the National Natural Science Foundation of China [No. 62171041], the BIT Research and Innovation Promoting Project [No. 2023YCXY053], and Natural Science Foundation of Beijing Municipality, China [No. 4242011]. The authors would like to thank the editor, reviewers, and Prof. Linyu Peng of Keio University for their comments that improved the technical quality of this work.

\newpage

%If you have an EPS/PDF photo (graphicx package needed), extra braces are needed around the contents of the optional argument to biography to prevent the LaTeX parser from getting confused when it sees the complicated $\backslash${\tt{includegraphics}} command within an optional argument. (You can create your own custom macro containing the $\backslash${\tt{includegraphics}} command to make things simpler here.)
 

%\bf{If you include a photo:}\vspace{-33pt}
%\begin{IEEEbiography}[{\includegraphics[width=1in,height=1.25in,clip,keepaspectratio]{fig1}}]{Michael Shell}
%Use $\backslash${\tt{begin\{IEEEbiography\}}} and then for the 1st argument use $\backslash${\tt{includegraphics}} to declare and link the author photo.
%Use the author name as the 3rd argument followed by the biography text.
%\end{IEEEbiography}



%\vfill

\end{document}